\documentclass[10pt,a4paper]{amsart}
\usepackage{verbatim}
\usepackage{hyperref}
\usepackage[toc]{appendix}
\usepackage[T1]{fontenc}
\usepackage{graphicx}
\usepackage{enumerate}
\usepackage{amsmath,amsfonts,amssymb}
\usepackage{color}
\def\loc{\operatorname{loc}}
\usepackage{cite}
\definecolor{citation}{rgb}{0.11,0.67,0.84}
\definecolor{formula}{rgb}{0.1,0.2,0.6}
\definecolor{url}{rgb}{0.11,0.67,0.84}
\usepackage{pgf,tikz}
\usepackage{mathrsfs}
\usepackage{fancyhdr}
\usepackage{dutchcal}

\usepackage{dsfont}

\newcommand{\reqnomode}{\tagsleft@false}

\def\dx{\,{\rm d}x}

\def\dt{\,{\rm d}t}
\def\dy{\,{\rm d}y}
\def\dla{\,{\rm d}\lambda}

\def \d{\,{\rm d}}
\def \diver{\,{\rm div}}
\def\dist{\,{\rm dist}}
\def\deb{\rightharpoonup}

\def\supp{\,{\rm supp}}

\newcommand\ccc{\mathfrak{c}}

\allowdisplaybreaks
\makeatletter
\DeclareRobustCommand*{\bfseries}{%
  \not@math@alphabet\bfseries\mathbf
  \fontseries\bfdefault\selectfont
  \boldmath
}

\DeclareMathOperator*{\osc}{osc}

\makeatother

\newlength{\defbaselineskip}
\newcommand{\setlinespacing}[1]
           {\setlength{\baselineskip}{#1 \defbaselineskip}}
\newcommand{\mint}{\mathop{\int\hskip -1,05em -\, \!\!\!}\nolimits}

\newtheorem{theorem}{Theorem}
\newtheorem{corollary}{Corollary}
\newtheorem{definition}{Definition}
\newtheorem{remark}{Remark}
\newtheorem{lemma}{Lemma}[section]
\newtheorem{proposition}{Proposition}[section]
\numberwithin{equation}{section}
\allowdisplaybreaks[4]

\newcommand{\hhh}{\textnormal{\texttt{h}}}
\newcommand{\ssf}{\mathfrak{s}}
\newcommand{\sst}{\mathfrak{t}}

\newcommand{\ppp}{\mathfrak p}

\newcommand{\kk}{\kappa}

\def\en{\mathbb N}
\def\er{\mathbb R}
\def\FF{\mathbb{F}}

\newcommand{\gggg}{\textnormal{\texttt{g}}}

\newcommand{\BB}{\mathcal{B}_1}
\newcommand{\BBB}{\mathcal{B}_{1/2}}
\newcommand\eps\varepsilon

\def\eqn#1$$#2$${\begin{equation}\label#1#2\end{equation}}

\newcommand{\data}{\textnormal{\texttt{data}}}
\newcommand{\datae}{\textnormal{\texttt{data}}_{\textnormal{e}}}

\newcommand{\be}{\begin{equation}}
\newcommand{\alpham}{\alpha_{\rm m}}
\newcommand{\pppm}{\ppp_{\rm m}}
\newcommand{\ee}{\end{equation}}

\newcommand{\rr}{\varrho}

\newcommand{\snr}[1]{\lvert #1\rvert}

\newcommand{\nr}[1]{\lVert #1 \rVert}

\newcommand{\una}{\mathds{1}_p}

\newcommand{\N}{\mathbb{N}}

\def\name[#1, #2]{#1 #2}

\newcommand{\rif}[1]{(\ref{#1})}
\newcommand{\trif}[1] {\textnormal{\rif{#1}}}

\newcommand{\stackleq}[1]{\stackrel{\rif{#1}}{ \leq}}

\title[Nonuniformly elliptic Schauder theory]{Nonuniformly elliptic Schauder theory}

\author[De Filippis]{Cristiana De Filippis}  \address{Cristiana De Filippis\\Dipartimento SMFI, Universit\`a di Parma, Viale delle Scienze 53/a, Campus, 43124 Parma, Italy} \email{\url{cristiana.defilippis@unipr.it}}
\author[Mingione]{Giuseppe Mingione}  \address{Giuseppe Mingione\\Dipartimento SMFI, Universit\`a di Parma, Viale delle Scienze 53/a, Campus, 43124 Parma, Italy} \email{\url{giuseppe.mingione@unipr.it}}
\begin{document}

\subjclass[2010]{49N60, 35J60\vspace{1mm}} 

\keywords{Regularity, nonuniform ellipticity, Schauder estimates\vspace{1mm}}

\thanks{{\it Acknowledgements.}\ This work is supported by the University of Turin via the project "Regolarit\'a e propriet\'a qualitative delle soluzioni di equazioni alle derivate parziali" and by the University of Parma via the project "Regularity, Nonlinear Potential Theory and related topics".
\vspace{1mm}}

\maketitle

\centerline{To Arrigo Cellina, with admiration for his pioneering}
\centerline{ work in the Calculus of Variations}

\begin{abstract}
Local Schauder estimates hold in the nonuniformly elliptic setting. Specifically, first derivatives of solutions to nonuniformly elliptic variational problems and elliptic equations are locally H\"older continuous, provided coefficients are locally H\"older continuous. 
\end{abstract}
\vspace{3mm}
\setcounter{tocdepth}{1}
{\small \tableofcontents}

\setlinespacing{1.08}
\section{Introduction}\label{si}
In this paper we give the first general solution to two different, but yet connected, longstanding and classical open problems in the regularity theory of variational integrals and elliptic equations. To begin with, we prove the first results concerning local gradient H\"older regularity of minimizers of nonuniformly elliptic integrals, that are not necessarily equipped with a Euler-Lagrange equation. In fact, in the cases we are going to consider here, such equations might not exist. Such type of results are classical in the uniformly elliptic case \cite{gg1, gg2, gg3}, whilst nothing is known in the general nonuniformly elliptic one. Second, and most importantly, we prove Schauder estimates for nonuniformly elliptic problems. Both for variational problems and for elliptic equations, the gradient of solutions is locally H\"older continuous provided coefficients are locally H\"older continuous. Again, while this is classical in the uniformly elliptic case -- see again \cite{gg2, gg3}, Manfredi's \cite{manth1, manth2} and Lieberman's \cite{liebe} papers for full generality -- no analog is recorded in the nonuniformly elliptic case. The crucial point in this setting is to obtain $L^\infty$-gradient bounds, after which, more classical perturbation methods can be combined with certain specific forms of the a priori estimates obtained, to prove gradient H\"older continuity. See for instance the comments in Lieberman's review of Giaquinta \& Giusti's paper \cite{gg3}\footnote{Indeed, in the MR review of \cite{gg3}, Lieberman states: ``A comment needs to be made concerning their [i.e.,\,of Giaquinta \& Giusti's methods] brief application to equations when their growth properties fail. As they point out, such equations fall under their considerations provided a global gradient bound has been established; however, this gradient bound has only been proved when $A$ [i.e.,\,the operator or functional considered in \cite{gg3}] is differentiable with respect to all its arguments, and in many cases more smoothness of the coefficients is needed. The results of this paper are thus much more striking when applied to uniformly elliptic equations than to nonuniformly elliptic ones" \cite{lieberev}. The missing growth properties forcing smoothness of coefficients Lieberman is pointing at, correspond to nonuniform ellipticity. This can be therefore treated, when coefficients are H\"older continuous, only upon assuming that solutions are a priori Lipschitz (and under certain additional assumptions, like non-degeneracy). In this paper we overcome these points.}. The central role of gradient bounds is also remarked by Ivanov \cite[page 7]{ivanov2}\footnote{Ivanov remarks: ``In view of the results of Ladyzhenskaya and Uraltseva, the problem of solvability of boundary value problems for a nonuniformly elliptic or parabolic equation reduces to the question of constructing a priori estimates of the maximum moduli of the gradients of solutions for a suitable one-parameter family of similar equations". This means finding uniform a priori gradient estimates for regularized problems. This is shown to be possible in this paper without the unnatural assumptions considered before, i.e., without differentiability and smoothness of coefficients, in turn ruling out the Schauder setting. See the proof of Theorem \ref{t6}.} and was exploited by Ladyzhenskaya \& Uraltseva \cite{LU}, as described in \cite[page 15]{ivanov2} too. To achieve our results, we employ a novel hybrid perturbation approach, suited for nonuniformly elliptic problems. This is aimed at replacing the classical ones used in the uniformly elliptic setting, that  are ultimately based on plain freezing arguments. We believe that this approach has potential for applications in several other places. In fact, in this paper we present the main bulk of the technique and apply it in a certain number of different settings. Others are still possible. In particular, the boundary case, as well as the evolutionary one, will be treated in forthcoming papers. Obstacle problems can also be treated. 

So-called Schauder estimates for linear elliptic equations are actually a classic achievement of Hopf \cite{hopf}, Caccioppoli \cite{cacc1} and Schauder \cite{js, js2}. See also \cite{camp, gt, trusc, simon2} for modern proofs. The nonlinear story goes back to the classical papers by Frehse \cite{Fre}, Giaquinta \& Giusti \cite{gg1, gg2, gg3, gg4}, Ivert \cite{ivert1, ivert2} and Manfredi \cite{manth1, manth2}. There the first Schauder type results for nonlinear equations and nondifferentiable integral functionals,  asserting local H\"older continuity of the gradient for some exponent, were proved. For the sake of simplicity, let us consider the following classical model example, \cite{Ce, Ce1, gg2, KM, HS, phil, stamp, tromba}:
\begin{flalign}\label{ggg}
w\mapsto 
\int_{\Omega}[F(Dw)+\hhh(x,w)] \dx\,.
\end{flalign}
Here $F(\cdot)\geq 0$ is a sufficiently regular, uniformly elliptic integrand with $p$-growth - take for instance $F(Dw)\equiv |Dw|^p$, $p>1$. Instead, $\hhh\colon \Omega \times \er\to \er$ is a bounded, merely H\"older continuous function; $\Omega \subset \er^n$ denotes a bounded open subset, $n\geq 2$. By uniform ellipticity of $F(\cdot)$ here we mean that 
the ellipticity ratio $\mathcal R_{F} (z)$ remains bounded for $\snr{z}$ large \cite{LU, tru1967, ivanov2}, i.e., 
\eqn{ellratio}
$$
\sup_{\snr{z}\geq 1}\, \mathcal R_{F} (z)< \infty, \quad  \mathcal R_{F} (z):=  \frac{\mbox{highest eigenvalue of}\ \partial_{zz}  F(z)}{\mbox{lowest eigenvalue of}\  \partial_{zz} F(z)} \,.
$$
This happens for instance in the $p$-Laplacean case $F(z)=\snr{z}^p$, i.e., when 
\eqn{pcase}
$$
\partial_{zz}F(z) \approx \snr{z}^{p-2} \mathbb{I}_{\textnormal{d}}
$$
holds for $\snr{z}$ large. 
As $\hhh(\cdot)$ is not assumed to be differentiable, the Euler-Lagrange equation 
\eqn{eueu}
$$
-\diver\, \partial_zF(Du)+ \partial_u\hhh(x,u)=0
$$
of the functional in \rif{ggg} just cannot be derived. Yet, in \cite{gg2} a method is devised to get local gradient H\"older continuity of minima only using minimality, and without passing through \rif{eueu}. This goes roughly as follows. Given a minimizer $u$ of the functional in \rif{ggg}, one defines its lifting $v$ on the ball $B \Subset \Omega$ by solving
\eqn{lift}
$$
v \to \min_{w\in u+W^{1,p}_0(B)}\, \int_{B}F(Dw) \dx\,.
$$ 
Solutions to corresponding Euler-Lagrange equation 
$\diver\, \partial_zF(Dv)=0$, are $C^{1,\alpha}$-regular and enjoy good {\em homogeneous} decay estimates. This is a direct consequence of the uniform ellipticity \rif{ellratio}. Such estimates can then be matched with comparison ones between $u$ and $v$ on $B$ as the enjoy {\em the same degree of homogeneity} of the reference ones and this is another consequence of \rif{ellratio}. Thanks to this common homogeneity, combining the two ingredients finally leads to transfer, at all scales, the $C^{1,\alpha}$-estimates available for solutions to \rif{lift}, to the original minimizer $u$. This comparison scheme, based on minimality, is in spirit close to the one used to derive classical Schauder estimates and relying on Korn's argument. 

Such classical perturbation schemes uniformly fail in the nonuniformly elliptic setting. In this paper we show a different route aimed at bypassing the lack of homogeneous estimates typical of nonuniformly elliptic problems. For this, we will consider a general class of integrands for which \rif{ellratio} fails, and for which $\mathcal R_{F} (z)$ grows at most polynomially
\eqn{polygro}
$$
\mathcal R_{F} (z) \lesssim \snr{z}^{\delta}+1\,, \qquad \delta>0\,.
$$
In fact, our main ellipticity  assumption, replacing \rif{pcase}, will be of the type
\eqn{nonuni}
$$
\snr{z}^{p-2} \mathbb{I}_{\textnormal{d}} \lesssim \partial_{zz}F(z) \lesssim \snr{z}^{q-2} \mathbb{I}_{\textnormal{d}}
$$
for $\snr{z}\geq 1$, so that it is $\delta=q-p$ in \rif{polygro}; conditions \rif{nonuni} are the most general to describe \rif{polygro}. Polynomial nonuniform ellipticity as in \rif{polygro} is a standard topic since the classical works of Ladyzhenskaya \& Uraltseva \cite{LU, LUcpam}, Hartman \& Stampacchia \cite{hast}, Trudinger \cite{tru1967, tru}, Ivočkina \& A.\,P. Oskolkov \cite{IO}, Oskolkov \cite{osk}, Serrin \cite{serrin}, A.\,V. Ivanov \cite{ivanov0, ivanov1, ivanov2},  Leon Simon \cite{Simongrad}, Uraltseva \& Urdaletova \cite{UU}, Lieberman \cite{liebe0}, just to mention a few. In the variational setting, conditions \rif{nonuni} were systematically studied by Marcellini in a series of pioneering papers \cite{M1, M2, M3}, who introduced functionals with so-called $(p,q)$-growth conditions, referring to the formulation in \rif{nonuni}. Today a huge literature is devoted to such problems. With the current techniques, differentiability of coefficients is unescapable; H\"older regularity is forbidden. 

To describe the methods employed here, we recall that a traditional, {\em non-perturbative and direct} way to get that minima are Lipschitz when equation \rif{eueu} exists, is to first differentiate \rif{eueu}, and then invoking De Giorgi-Nash-Moser theory. Sticking to De Giorgi's method, this means to get first a Caccioppoli inequality involving derivatives of $Du$, and then to run a geometric iteration leading to gradient boundedness. Here we use, in a sense, both the direct and the perturbative approach, and that's where the word hybrid stems from. We also take advantage of various regularity tools and viewpoints developed over the last years in the Calculus of Variations \cite{BM, KM, ciccio} and in Nonlinear Potential Theory \cite{kilp, mi}. The approach proposed here goes along the following bullet points:
\begin{itemize}
\item We still prove $Du \in L^\infty_{\loc}$ via a direct De Giorgi type geometric iteration involving, up to minor corrections, truncations of a certain convex function of $\snr{Du}$ (Bernstein method). The iteration is based on a Caccioppoli type inequality, that this time does not involve full derivatives of $Du$, as in the classical case. This is because the functionals/equations we are dealing with involve nondifferentiable coefficients, and therefore cannot be differentiated. Instead, notwithstanding  the problem is local, the Caccioppoli inequality we use involves fractional derivatives of $Du$; see Section \ref{ibrida}. 
\item To iterate the Caccioppoli inequality, we use a renormalization that makes it homogeneous, as in uniformly elliptic problems. The price we pay is a controlled increase of the involved multiplicative constants. They now incorporate an additional, direct dependence on $\|Du\|_{L^\infty}^{\sigma}$, with $\sigma\equiv \sigma (q/p)$. Keeping track of such constants is a crucial part of the proof. For this, we impose a moderate polynomial growth rate on $\mathcal R_{F} (z)$ as in \rif{polygro}, that implies that $\delta \equiv q-p$ must be small enough, in turn making $\sigma$ small too. This kind of assumption is not technical and it is necessary (see Remark  \ref{gapre}). 

\item The perturbative part, making the approach hybrid. This is in the proof of the Caccioppoli inequality. For this we exploit a delicate atomic like decomposition in Nikolski spaces. The argument still employs comparisons with liftings $v$ as in \rif{lift}, used in a sense as atoms (see Section \ref{hyb2} below). At this stage, crucial use is made of precise a priori Lipschitz estimates for autonomous problems (see Lemma \ref{marcth}). 
\item We rely on nonlinear potentials of the type originally introduced by Havin \& Maz'ya \cite{HM}; see Section \ref{losec}. Indeed, certain quantities apparently uncontrollable in the nonuniformly elliptic setting, are now treated by means of an optimized  splitting between the $L^{\infty}$-norm of $Du$, and integral remainder terms building up nonlinear potentials along iterations (see Propositions \ref{caccin3}-\ref{caccin4}).  As an additional benefit, this nonlinear potential theoretic approach allows to treat cases involving unbounded data. In particular, it leads to discover new borderline conditions for regularity in Lorentz spaces. These extend known ones from standard settings and connect to a very large literature on borderline cases; see Remark \ref{lolorere}.   
\end{itemize}
This approach allows to settle the open problem of proving Schauder estimates in the nonuniformly elliptic setting. To fix the ideas, consider the model functional
\eqn{modello}
$$
\begin{cases}
\, \displaystyle w \mapsto \mathcal S_{\texttt{x}}(w, \Omega ):=\int_{\Omega}\ccc(x)F(Dw) \dx, \quad 0< \nu \leq \ccc(\cdot) \leq L\\
\,  |\ccc(x_1)-\ccc(x_2)| \leq L|x_1-x_2|^{\alpha}, \quad  \alpha \in (0,1],
 \end{cases}
$$
for every choice of $ x_1, x_2\in \Omega$, where $F(\cdot)$ satisfies \rif{nonuni}. When $p=q$, gradient H\"older regularity of minima can be found in \cite{gg1, gg2, liebe, manth1, manth2}. When $p\not= q$, the first result is in Theorem \ref{t2} below. 
Further developing our approach, 
we treat also nondifferentiable functionals of the type 
\eqn{modellou}
$$
\begin{cases}
\, \displaystyle w \mapsto \mathcal S (w, \Omega ):=\int_{\Omega}\left[\ccc(x,w)F(Dw)+\hhh(x, w) \right] \dx, \quad 0< \nu \leq \ccc(\cdot) \leq L\\
\, |\ccc(x_1, y_1)-\ccc(x_2, y_2)| \leq L|x_1-x_2|^{\alpha}+L|y_1-y_2|^{\alpha}, \quad    \alpha \in (0,1],
 \end{cases}
$$
for every choice of $x_1, x_2\in \Omega, y_1, y_2\in \er$, and thereby falling outside the realm of traditional Schauder estimates. 
In this case we assume the additional lower bound $p>n$, that we suspect to be necessary; see Theorem \ref{t3} and subsequent Remark \ref{gapremark}. We note that Theorem \ref{t3} is completely new already when $y \mapsto \ccc(\cdot, y)$ is smooth.

Our techniques apply to general nonuniformly elliptic equations in divergence forms of the type considered in \cite{LU, gg3, liebe, manth1, manth2} in the uniformly elliptic case. In order to present the main ideas and to keep presentation at a reasonable length, we confine ourselves to equations of the type 
\eqn{equazioneM}
$$
-\diver\, A (x,Du)=0\,,
$$
but not necessarily stemming from variational integrals; see Section \ref{equazioni}. Again, more general cases, for instance involving non-zero right-hand sides, can be treated by our means.

\section{Results}\label{risultati}
\subsection{Basic notation}\label{basicn}
We deal with integral functionals of the form
\eqn{Fx}
$$
W^{1,1}_{\loc}(\Omega) \ni w\mapsto \mathcal{F}(w,\Omega):=\int_{\Omega}\left[\mathbb{F}(x,w,Dw)+\hhh(x, w) \right]  \dx\,.
$$
Here, as in the rest of the paper, $\Omega\subset \er^n$, $n\geq 2$, denotes a fixed open and bounded subset, $\FF\colon \Omega \times \er \times \er^n \to [0, \infty)$ and $\hhh \colon \Omega \times \er \to\er$ are Carath\'eodory regular functions, and $\hhh(\cdot)$ is such that $\hhh(\cdot, w)\in L^1_{\loc}(\Omega)$, whenever $w \in W^{1,1}_{\loc}(\Omega)$. Under such premises, we adopt the following
\begin{definition}\label{defi-min} A function $u \in W^{1,1}_{\loc}(\Omega)$ is a \emph{(local) minimizer} of the functional $\mathcal F$ in \eqref{Fx} if, for every ball $B\Subset \Omega$, we have $\FF(\cdot, u, Du) \in L^1(B)$ and $\mathcal F(u;B)\leq \mathcal F(w;B)$ holds for every competitor $w \in u + W^{1,1}_0(B)$. 
\end{definition}
For the rest of the paper, we denote by $c, \chi_1, \chi_2$ general constants such that $c, \chi_1, \chi_2\geq 1$. Different occurrences from line to line will be still denoted using the same letters. Special occurrences of $c$ will be denoted by $c_*,  \tilde c$ or likewise. Relevant dependencies on parameters will be as usual emphasized by putting them in parentheses; for instance $c\equiv c (n,p,q)$ means that $c$ depends on $n,p,q$. Next, we fix a set of real parameters denoted by $
\data \equiv (n, p, q,\alpha, \nu, L)$, where $n\geq 2$ is an integer, $1<p\leq q$, $0 < \nu \leq  L$, 
and $\alpha \in (0,1]$, and also set $\datae \equiv (n, p, q,\alpha)$. With $\nu, L$ being fixed, 
we denote by $\tilde \nu \equiv \tilde \nu(n,p,\nu)$ and $\tilde L \equiv \tilde L(n,q, L)$ two quantities such that $0 < \tilde \nu \leq \nu\leq L  \leq \tilde L$. While $\nu, L$ are fixed here, the exact value of the quantities $\tilde \nu, \tilde L$ might vary in different occurrences, but still keeping the dependence on the constants specified above. For this reason, dependence on $\tilde \nu, \tilde L$ will be often be incorporated in the dependence on $\data$ or on $\nu, L$. Unless otherwise specified, $\mu$ denotes a fixed constant such that $\mu\in [0,1]$; we also denote 
\eqn{defiH}
$$\mbox{$H_{s}(z):=\snr{z}^{2}+s^{2}$, \ \  for $z\in \er^n$,  \quad $\mu_s:= \mu+s$ \ \ and $s\geq0$\,.}
$$ 
Further notation can be found in Section \ref{further}. 
\subsection{Functionals without the Euler-Lagrange equation}\label{sezione}
Here we concentrate on nondifferentiable functionals of the form
\begin{flalign}\label{ggg2}
w\mapsto \mathcal{G}(w,\Omega):=\int_{\Omega}[F(Dw)+\gggg(x,w, Dw)+ \hhh(x, w)] \dx\,.
\end{flalign}
The integrand $F\colon \mathbb{R}^{n}\to [0, \infty)$ satisfies the growth and (nonuniform) ellipticity conditions
\begin{flalign}\label{assFF}
\qquad\begin{cases}
\,F(\cdot) \in C^1(\er^n)\cap C^2(\er^n\setminus\{0\})\\
\, \nu [H_{\mu}(z)]^{p/2}\le F(z)\le L[H_{\mu}(z)]^{q/2}+L[H_{\mu}(z)]^{p/2}\\
\, \nu [H_{\mu}(z)]^{(p-2)/2}\snr{\xi}^{2}\le \partial_{zz}F(z)\xi\cdot \xi \\
\, \snr{\partial_{zz}F(z)}\le L[H_{\mu}(z)]^{(q-2)/2}+L[H_{\mu}(z)]^{(p-2)/2},
\end{cases}
\end{flalign}
for all $z\in \mathbb{R}^{n}\setminus \{0\}$, $\xi\in \mathbb{R}^{n}$ \footnote{Assumption \rif{assFF}$_2$, in conjunction with Definition \ref{defi-min}, implies that any minimizer $u$ of $\mathcal{G}$ belongs to $W^{1,p}_{\loc}(\Omega)$. It will always be the case for the rest of this paper, also when dealing with different functionals.}. The function $ \gggg\colon \Omega\times \mathbb{R}\times \mathbb{R}^{n}\to [0, \infty)$ satisfies
\begin{flalign}\label{assgg}
\begin{cases}
\, z\mapsto \gggg(x,y,z) \  \mbox{is convex and of class $C^1(\er^n)\cap C^2(\er^n\setminus\{0\})$}\\
\, \gggg(x,y,z)+H_{\mu}(z)\snr{\partial_{zz}\gggg(x,y,z)}\leq L[H_{\mu}(z)]^{p/2} \\
\, \snr{\gggg(x_{1},y_{1},z)-\gggg(x_{2},y_{2},z)}\le L\left(\snr{x_{1}-x_{2}}^{\alpha}+\snr{y_{1}-y_{2}}^{\alpha}\right)(\snr{z}^2+1)^{\gamma/2}\\
\, \alpha +\gamma <p, \ \ \gamma \geq 0
\end{cases}
\end{flalign}
for all $x,x_{1},x_{2}\in \Omega$, $y,y_{1},y_{2}\in \mathbb{R}$, $z\in \mathbb{R}^{n}\setminus \{0\}$\footnote{The function $\gggg(\cdot)$ is continuous. This follows from \rif{assgg} and Lemma \ref{marclemma}.}.
The function $\gggg(\cdot)$ is therefore only H\"older continuous with respect to $(x,y)$; a typical example can be $\gggg(x,y,z)=\ccc(x,y)[H_{\mu}(z)]^{p/2}[H_{1}(z)]^{(\gamma-p)/2}$, where $\ccc(\cdot)$ is as in \rif{modellou}. The function 
$\hhh\colon \Omega\times \mathbb{R}\to \mathbb{R}$ is assumed to be Carath\'eodory regular and to satisfy conditions involving Lorentz spaces (see Section \ref{losec})
\eqn{asshh}
$$
\begin{cases}
\, \snr{\hhh(x,y_{1})-\hhh(x,y_{2})}\le f(x)\snr{y_{1}-y_{2}}^{\alpha}\\
\,  f\in L\left(n/\alpha,1/2\right)(\Omega)\\
\, \hhh (\cdot, 0)\in L^1(\Omega)
\end{cases}
$$
for a.e.\,$x\in\Omega$ and every $y_{1}, y_{2}\in \er$ \footnote{Assumptions \rif{asshh} imply that $\hhh(\cdot, w)\in L^1_{\loc}(\Omega)$, whenever $w \in W^{1,1}_{\loc}(\Omega)$, as required by the discussion before Definition \ref{defi-min}. This follows using Sobolev embedding theorem and \rif{crescehh} below.}. This time $\hhh(\cdot)$ is only measurable with respect to the $x$-variable; an example is $\hhh(x,y)\approx f(x) h(y)$, where $  h(\cdot)$ is any H\"older continuous function. 
\begin{theorem}\label{t1}
Let $u\in W^{1,1}_{\loc}(\Omega)$ be a minimizer of the functional $\mathcal{G}$ in \trif{ggg2}, under assumptions \eqref{assFF}-\eqref{asshh} and 
\begin{eqnarray}\label{bound1}
\frac{q}{p}\leq  1+\frac 15\left(1-\frac{\alpha+\gamma}{p}\right)\frac{\alpha}{n}\,.
\end{eqnarray}
Then $Du\in L^{\infty}_{\loc}(\Omega,\er^n)$. Moreover, 
\eqn{stima1}
$$
\nr{Du}_{L^{\infty}(B_{t})}\le \frac{c}{(r-t)^{\chi_1}}\left[ \mathcal{G}(u, B_{r}) +\|\hhh(\cdot,u)\|_{L^1(B_{r})}+\|f\|_{n/\alpha,1/2;B_{r}}+1\right]^{\chi_2}
$$
holds whenever $B_{t}\Subset B_{r} \Subset \Omega$ are concentric balls with $r\leq 1$, where $c\equiv c(\data, \gamma)$ and $ \chi_1, \chi_2\equiv \chi_1, \chi_2(\datae, \gamma)$. If $f \in L^{\mathfrak{q}}_{\loc}(\Omega)$ for some $\mathfrak{q}>n/\alpha$, then $Du$ is locally H\"older continuous in $\Omega$. 
\end{theorem}
\begin{remark}[Gap bounds]\label{gapre}{\em Gap bounds of the form $
q/p < 1 + {\rm o}(n)$ for $ {\rm o}(n) \approx 1/n$, as in \rif{bound1}, 
are known to be necessary already for boundedness of minimizers, as shown in \cite{M1, M2}, and in the case $w \mapsto \int F(Dw)\dx$. Already in such a plain situation, the optimal bound on $q/p$ implying $Du \in L^\infty_{\loc}$ remains unknown. See \cite{BS0, BS, BB, cammello, HS, sha} for recent work in this direction. Counterexamples in \cite{sharp}, involving non-homogeneous integrands show that the condition
\eqn{sharpbound}
$$
\frac{q}{p}\leq  1+\frac{\alpha}{n}
$$
is necessary already for continuity of minima. This is crucial in our setting. Specifically, it implies that, when departing from standard ellipticity conditions, Schauder estimates do not hold in general. A subtle balance between regularity of $x\mapsto F(x, \cdot)$ and ellipticity of $z\mapsto F( \cdot, z)$ is needed. The bound in \rif{bound1} reflects this fact and exhibits the same sharp asymptotic with respect to $\alpha/n$ of \rif{sharpbound}. A more delicate interaction occurs with the regularity of $y \mapsto \gggg(\cdot, y, \cdot)$, and reflects in conditions $\alpha + \gamma <p$ and \rif{bound1}, a condition that disappears in the case of functionals as in \eqref{ggg}, where we can take $\gamma=0$. See Remark \ref{gapremark} for more discussion. The bound in \rif{bound1} can be further (slightly) improved by introducing a correction function $\kappa_{1}(\cdot)$. For this see Proposition \ref{priori1} and subsequent Remark \ref{raffina}. Anyway this does change the asymptotic of \rif{bound1} with respect to $\alpha/n$ and $p-(\alpha+\gamma)$, that is what we are mainly interested in at this stage}. 
\end{remark}
\begin{remark}[New borderline conditions]\label{lolorere} {\em When $p=q$ (uniformly elliptic case), the assumptions of Theorem \ref{t1} are standard \cite{gg1, gg2, gg3, gg4, KM, manth1, manth2}, but the one on $f$, which is usually taken in $L^\infty$. Assuming $f\in L^\infty$ would  not improve the gap bound in \rif{bound1}. Passing from $L^\infty$ to Lebesgue spaces, and eventually to Lorentz, is an automatic side benefit of our approach. By scaling arguments, we do not expect H\"older continuity of $Du$ without assuming that $f \in L^{\mathfrak{q}}(\Omega)$ for $\mathfrak{q}>n/\alpha$, and do not expect that $Du \in L^\infty$ when $f \in L^{n/\alpha}$. The new Lorentz condition \rif{asshh} connects Theorem \ref{t1}  to a large literature devoted to find optimal conditions on data $f$ implying regularity. For differentiable, uniformly elliptic functionals like
$
w\mapsto\int_{\Omega} [F(Dw)+fw] \dx, 
$
and equations $\diver\, a(Du)=f$, the $L(n,1)$-regularity of $f$ implies that solutions are locally Lipschitz; $L^n$ is not sufficient. This can be considered as a nonlinear version of a classical result of Stein \cite{St} on solutions to $\Delta u=f$. See \cite{KMstein} for local estimates, and \cite{CM0, CM1} for global ones. The same holds when assuming that $F(\cdot)$ is non-uniformly elliptic in the sense of \rif{nonuni} and \rif{assFF}, as in \cite{BM, cammello, ciccio}. The remarkable fact is that the $L(n,1)$-condition is independent of the integrand/operator considered, exactly as the condition $f \in L(n/\alpha, 1/2)$ considered here for \rif{ggg2}. Anyway, this can be still improved in certain situations, see Theorem \ref{t5} below. As far as we know, due to the presence of the Lorentz condition, Theorem \ref{t1} is completely new already when $p=q$.}
\end{remark}

\subsection{Nonuniformly elliptic Schauder and more nondifferentiable functionals.}\label{sezionemodello}
\begin{theorem}\label{t2}
Let $u\in W^{1,1}_{\loc}(\Omega)$ be a minimizer of the functional $\mathcal{S}_{\texttt{x}}$ in \trif{modello}, under assumptions \eqref{assFF}; in particular, $\ccc(\cdot) \in C^{0, \alpha}(\Omega)$. If 
\eqn{bound3}
$$
\frac{q}{p}\leq  1+  \frac{\alpha^2}{5n^2}\,,
$$
then $Du$ is locally H\"older continuous in $\Omega$. Moreover
\eqn{stima2m}
$$
\nr{Du}_{L^{\infty}(B_{t})}\le \frac{c}{(r-t)^{\chi_1}}\left[ \mathcal{S}_{\texttt{x}}(u, B_{r}) +1\right]^{\chi_2}
$$
holds whenever $B_{t}\Subset B_{r} \Subset \Omega$ are concentric balls with $r\leq 1$, and $c\equiv c(\data)$, $ \chi_1, \chi_2\equiv \chi_1, \chi_2(\datae)$. 
\end{theorem}
\begin{corollary}[Hopf-Schauder-Caccioppoli, reloaded - I]\label{c1}
In the setting of Theorem \ref{t2}, assume also that $p\geq 2$, $\mu>0$, and that $\partial_{zz}F(\cdot)$ is continuous. Then $
u\in C^{1, \alpha}_{\loc}(\Omega).
$ 
\end{corollary}
In the linear case $F(Dw)\equiv |Dw|^2$, Corollary \ref{c1} is nothing but the content of classical (interior) Schauder estimates, where local $C^{0, \alpha}$-regularity of coefficients sharply reflects in local $C^{1, \alpha}$-regularity of solutions. Assuming the non-degeneracy condition $\mu>0$ is necessary, as otherwise shown by counterexamples occurring in the $p$-Laplacean setting \cite{le1, Ur}. In the uniformly elliptic case $p=q=2$, Corollary \ref{c1} is a classical result of Giaquinta \& Giusti \cite{gg2,gg3}. 
\begin{theorem}\label{t3}
Let $u\in W^{1,1}_{\loc}(\Omega)$ be a minimizer of the functional $\mathcal{S}$ in \trif{modellou}, under assumptions \eqref{assFF}, with $\hhh(\cdot)$ as in \trif{asshh} and $p>n$. If 
\eqn{bound4}
$$
 \frac{q}{p} \leq   1 +\frac 15\left(1-\frac {n}{p}\right)
\frac{\alpha^2}{n^2},$$
then $Du$ is locally H\"older continuous in $\Omega$. Moreover, 
\eqn{stima2mm}
$$
\nr{Du}_{L^{\infty}(B_{t})}\le \frac{c}{(r-t)^{\chi_1}}\left[ \mathcal{S}(u, B_{r})+\|\hhh(\cdot,u)\|_{L^1(B_{r})}+\|f\|_{n/\alpha,1/2;B_{r}}+1\right]^{\chi_2}
$$
holds with the notation relative to \trif{stima2m}. 
\end{theorem}
\begin{remark}\label{gapremark} {\em 
In Theorem \ref{t3}, assuming $p>n$ ensures, via Sobolev embedding, some degree of H\"older continuity of $x \mapsto \ccc(x, u(x))$, and rebalances the otherwise measurable interaction between coefficients and gradient terms. When $p=q$, this is not necessary since De Giorgi-Nash-Moser theory ensures that minimizers are a priori H\"older continuous, also when $x\mapsto \ccc(x, \cdot)$ is measurable. This is in general false when $p\not=q$. The situation shares  similarities with the uniformly elliptic vectorial theory, where De Giorgi type results are not available, and indeed singularities occur for minimizers no matter dependence on $u$-coefficients is smooth. In this respect, we note that some counterexamples of irregular minimizers in the scalar, nonuniformly elliptic case \cite{sharp} resemble those occurring in the uniformly elliptic vectorial one \cite{giumi}. By such a potential analogy, we tend to believe that the assuming $p>n$ is unavoidable. For the same reasons, \rif{bound4} connects to \rif{bound1}, where $\gamma$ quantifies the interaction between coefficients and gradient terms.} 
\end{remark}
\begin{corollary}\label{c2}
In the setting of Theorem \ref{t3}, assume also that $p\geq 2$, $\mu>0$, and that $\partial_{zz}F(\cdot)$ is continuous; finally,  assume that $f\in L^{\infty}_{\loc}(\Omega)$. Then $
u\in C^{1, \alpha/2}_{\loc}(\Omega).
$
\end{corollary}
In comparison to Corollary \ref{c1}, Corollary \ref{c2} exhibits a loss in the H\"older exponent of $Du$. This is typical when dealing with nondifferentiable functionals. It is not a technical fact, as $C^{1,\alpha}$-regularity cannot be reached in general, as shown in \cite{phil}. In the uniformly elliptic case $p=q=2$, Corollary \ref{c2} is another classical result of Giaquinta \& Giusti \cite{gg2,gg3,gg4}. 

\subsection{General functionals, relaxation and the Lavrentiev phenomenon}\label{completasec}
In order to deal with cases more general than \rif{modello}, that is with integrals of the type
\eqn{FMx}
$$
 w\mapsto \mathcal{F}_{\texttt{x}}(w,\Omega):=\int_{\Omega}F(x,Dw) \dx\,,
$$
we need to recast a few basic facts concerning relaxed functionals and Lavrentiev phenomenon. 
Here, we assume that the integrand $F\colon \Omega \times \er^n \to [0, \infty)$ is Carath\'eodory regular and such that 
\eqn{xx.0}
$$
\begin{cases}
\, z \mapsto F(x, z) \ \mbox{satisfies \rif{assFF} uniformly with respect to $x \in \Omega$}\\
\, \snr{\partial_{z}F(x_{1},z)-\partial_{z}F(x_{2},z)}\le L\snr{x_{1}-x_{2}}^{\alpha}([H_{\mu}(z)]^{(q-1)/2}+[H_{\mu}(z)]^{(p-1)/2})
\end{cases}
$$
for every $x_1, x_2 \in \Omega$ and $z\in \er^n$ \footnote{$\partial_z F(\cdot)$ is continuous. This follows from \rif{xx.0}$_2$ and the upper bound on $\partial_{zz}F(\cdot)$ in \rif{assFF}$_4$ via \rif{xx.0}$_1$. In fact $\partial_z F(\cdot)$ is uniformly continuous on $\Omega\times \mathcal{B}_M$, for every $M < \infty$. In fact, there is no loss of generality is assuming that $F(\cdot)$ is continuous. For this, it is sufficient to observe that we can assume $F(x, 0_{\er^n})=0$ (eventually passing to the new integrand $(x, z) \mapsto F(x, z)-F(x, 0_{\er^n})$). Then we can use \rif{xx.0} in combination with Lemma \ref{marclemma}.}. In this case, a natural obstruction to regularity of minimizers is the potential occurrence of the Lavrentiev phenomenon. For instance, it might happen that 
\eqn{lav}
$$
\inf_{w \in  u_0+W^{1,p}_0(B)}\, \mathcal{F}_{\texttt{x}}(w, B ) <  \inf_{w \in  u_0+W^{1,p}_0(B) \cap W^{1,q} (B)}\, \mathcal{F}_{\texttt{x}}(w, B)
$$
even when $u_0$ is a Lipschitz regular function and $B \Subset \Omega$ is a ball \cite{sharp}. Examples of  Lavrentiev phenomenon related to our setting, were given by Zhikov \cite{Z0, Z1, Z2}; see also \cite{sharp}. In the case energy gaps as \rif{lav} occur, one is led to consider the so-called relaxed functional
\eqn{rilassato}
$$
\overline{\mathcal F_{\texttt{x}}}(w,U)  \, := \inf_{\{w_k\}\subset W^{1,q}(U)}  \left\{ \liminf_k  \mathcal F_{\texttt{x}}(w_k,U)  \, \colon \,  w_k \deb w\  \mbox{in} \ W^{1,p}(U)  \right\} \footnote{By the very definition in \rif{rilassato}, it follows that $w \in W^{1, p}(U)$ whenever $ \overline{\mathcal F_{\texttt{x}}}(w,U)$ is finite. The idea of considering this type of lower semicontinuous envelope goes back to Lebesgue, Caccioppoli, Serrin and De Giorgi. In the nonuniformly elliptic setting, it appears for the first time in the work of Marcellini \cite{M0, ma5}.}
$$
for every $w\in W^{1,1}(U)$ and every open subset $U\subset \Omega$. Accordingly, the Lavrentiev gap functional is defined by
\eqn{lav1}
$$
\mathcal {L}_{\mathcal{F}_{\texttt{x}}}(w,U):= \overline{\mathcal F_{\texttt{x}}}(w,U)- \mathcal{F}_{\texttt{x}}(w,U)
$$
for every $w\in W^{1,1}(U)$ such that $\mathcal{F}_{\texttt{x}}(w,U)< \infty$; we set $\mathcal {L}_{\mathcal{F}_{\texttt{x}}}(w,U)=0$ otherwise. The functional $\mathcal {L}_{\mathcal{F}_{\texttt{x}}}(\cdot, U)$ provides a possible way to quantify phenomena like \rif{lav}.   
Note that by $W^{1,p}$-weak lower semicontinuity of $\mathcal{F}_{\texttt{x}}(\cdot,U)$ \footnote{Lower semicontinuity, when $z \mapsto F(\cdot, z)$ is convex, follows by results of De Giorgi and Ioffe, see \cite[Theorem 4.5]{giu}.}, we have $\mathcal{F}_{\texttt{x}}(\cdot,U)\leq \overline{\mathcal{F}_{\texttt{x}}}(\cdot,U)$ so that $\mathcal {L}_{\mathcal{F}_{\texttt{x}}}(\cdot,U)\geq 0$. A minimizer of $\overline{\mathcal F_{\texttt{x}}}(\cdot, \Omega)$ is a function $u \in W^{1, p}(\Omega)$ such that $\overline{\mathcal F_{\texttt{x}}}(u, \Omega)$ is finite and $\overline{\mathcal F_{\texttt{x}}}(u, \Omega) \leq \overline{\mathcal F_{\texttt{x}}}(w, \Omega)$ holds whenever $w \in u+ W^{1,1}_0(\Omega)$. 
\begin{theorem}\label{t4}
Let $u\in W^{1,p}(\Omega)$ be a minimizer of the functional $\overline{\mathcal F_{\texttt{x}}}(\cdot, \Omega)$, where $\Omega$ is a Lipschitz regular domain and $\mathcal F_{\texttt{x}}$ is defined in \trif{FMx}. Assume \trif{bound3} and \eqref{xx.0}. Then $Du$ is locally H\"older continuous in $\Omega$. Moreover
\eqn{stima3}
$$
\nr{Du}_{L^{\infty}(\Omega_0)}\le \frac{c}{[\dist(\Omega_0, \partial \Omega)]^{\chi_1}} \left[ \overline{\mathcal F_{\texttt{x}}}(u, \Omega) +1\right]^{\chi_2}
$$
holds whenever $\Omega_0\Subset \Omega $ is an open subset, where $c, \chi_1, \chi_2$ are as in Theorem \ref{t2}. 
\end{theorem}
If $u\in W^{1,1}_{\loc}(\Omega)$ is a minimizer of the original functional $\mathcal F$ in \rif{FMx} such that $\mathcal {L}_{\mathcal{F}_{\texttt{x}}}(\cdot,B)\equiv 0$ for every ball $B\Subset \Omega$, then
$
 \overline{\mathcal F_{\texttt{x}}}(u,B) =\mathcal{F}_{\texttt{x}}(u,B)  \leq \mathcal{F}_{\texttt{x}}(w,B) \leq \overline{\mathcal F_{\texttt{x}}}(w,B)  
$ holds whenever $w \in u+W^{1,1}_0(B)$. 
Therefore $u$ is also a minimizer of $w\mapsto \overline{\mathcal F_{\texttt{x}}}(w, B)$, for every ball $B\Subset \Omega$. It follows
 \begin{corollary}\label{c3}
Let $u\in W^{1,1}_{\loc}(\Omega)$ be a minimizer of the functional $\mathcal F_{\texttt{x}}$, under assumptions \trif{bound3} and \eqref{xx.0}. Assume that $\mathcal {L}_{\mathcal{F}_{\texttt{x}}}(w,B)=0$ for every ball $B$. Then $Du$ is locally H\"older continuous in $\Omega$. Moreover, \eqn{stima3m}
$$
\nr{Du}_{L^{\infty}(B_{t})}\le \frac{c}{(r-t)^{\chi_1}}\left[ \mathcal F_{\texttt{x}}(u, B_{r})+1\right]^{\chi_2}
$$ holds with the notation relative to \trif{stima2m}. 
\end{corollary}

The Lavrentiev gap \rif{lav1} vanishes in several common situations. For instance, if there exists a convex function $G\colon \er^n \to [0, \infty)$ such that 
$
G(z) \lesssim F(x,z) \lesssim G(z) +1$, then $\mathcal {L}_{\mathcal{F}_{\texttt{x}}}(\cdot,B)\equiv 0$ holds for every ball $B \Subset \Omega$. This is the case of Theorem \ref{t2}, that in fact follows from Corollary \ref{c3}. See  \cite{sharp, Z0, Z1} for cases where the Lavrentiev gap is zero.

\subsection{More Lorentz conditions}\label{nonlimitati} In some cases, \rif{asshh}$_2$ can still be improved in 
 \eqn{lo}
 $$
 f\in L(n/\alpha,\textnormal{\texttt{l}})(\Omega), \ \ \mbox{where} \ \ 
 \textnormal{\texttt{l}}
 := \min\left\{ \frac{p}{2(p-\alpha)}, \frac{1}{2-\alpha} \right\}\,.
 $$
\begin{theorem}\label{t5}
Let $u\in W^{1,1}_{\loc}(\Omega)$ be a minimizer of the functional $\mathcal{G} $ in \trif{ggg2}, under assumptions \eqref{assFF}-\eqref{asshh}, and replace \trif{asshh}$_2$ by the weaker \trif{lo}. 
If \eqn{bound10}
$$
\frac qp \leq 1 +\frac 15\left(1-\frac{\alpha+\gamma}{p}\right)\min\left\{\frac \alpha n, \frac{2(p-\alpha)}{p(2-\alpha)}\right\}\,,
$$
then $Du\in L^{\infty}_{\loc}(\Omega,\er^n)$ and moreover 
\eqn{stima11}
$$
\nr{Du}_{L^{\infty}(B_{t})}\le \frac{c}{(r-t)^{\chi_1}}\left[\mathcal{G}(u, B_{r}) +\|\hhh(\cdot,u)\|_{L^1(B_{r})}+\|f\|_{n/\alpha, \textnormal{\texttt{l}};B_{r}}+1\right]^{\chi_2}
$$
holds whenever $B_{t}\Subset B_{r} \Subset \Omega$ are concentric balls with $r\leq 1$, where $c\equiv c(\data, \gamma)$ and $ \chi_1, \chi_2\equiv \chi_1, \chi_2(\datae, \gamma)$. \end{theorem}
Just note that, when $p>2n\alpha/[2n-2\alpha+\alpha^2]$, the upper bounds in \rif{bound1} and \rif{bound10} coincide; in particular, this happens when $p\geq 2$. 

\subsection{Equations}\label{equazioni} Here we deal with equations of the type \rif{equazioneM}. 
The vector field $A\colon \Omega \times \er^n \to \er^n$ is of class $C^1(\er^n\setminus\{0\})$ with respect to gradient variable, and satisfies
\begin{flalign}\label{assAA}
\qquad\begin{cases}
 \,|A(x,  z)|+ [H_{\mu}(z)]^{1/2}\snr{\partial_z A(x,  z)} \le L[H_{\mu}(z)]^{(q-1)/2}+L[H_{\mu}(z)]^{(p-1)/2}\\
\, \nu [H_{\mu}(z)]^{(p-2)/2}\snr{\xi}^{2}\le \partial_{z}A(x,z)\xi\cdot \xi \\
\, \snr{A(x_{1},z)- A(x_{2}, z)}  \le L \snr{x_{1}-x_{2}}^{\alpha }([H_{\mu}(z)]^{(q-1)/2}+[H_{\mu}(z)]^{(p-1)/2})
\end{cases}
\end{flalign}
whenever $x_1, x_2 \in \Omega$ and $z\in \er^n\setminus\{0\}$, $\xi \in \er^n$ \footnote{Assumptions \rif{assAA} imply that $A(\cdot)$ is uniformly continuous on $\Omega\times \mathcal{B}_M$, for every $M < \infty$.}. In the nonuniformly elliptic setting, passing from minimizers to functionals raises additional issues. There are essentially two approaches available in the literature. The former prescribes to get a priori estimates for more regular, i.e. $W^{1,q}$-solutions \cite{LU, M2, simon2}. The latter is to prove, simultaneously, the existence of regular solutions for assigned boundary value problems, say for instance Dirichlet problems \cite{BM, ivanov2, M2}. These alternatives are ultimately linked to growth conditions \rif{assAA}$_1$, that imply that the distributional form of \rif{equazioneM} can be tested only by $W^{1,q}$-regular functions. Therefore an ambiguity arises concerning the space where to initially pick the solution from, and on the same concept of energy solution. See Remark \ref{coti} below. Such an ambiguity does not exist when $p=q$. Of the two approaches mentioned above, we follow the second, and consider the Dirichlet problem
\eqn{dir1}
$$
\left\{
\begin{array}{c}
-\diver\, A (x,Du)=0\quad \mbox{in}\ \Omega\\ [3 pt]
u \equiv u_0  \quad \mbox{on}\ \partial \Omega\;,
 \end{array}\right.
\qquad \qquad 
u_0 \in W^{1,\frac{p(q-1)}{p-1}}(\Omega)\;,
$$
where $\Omega\subset \er^n$ is a bounded and Lipschitz domain. 
\begin{theorem}\label{t6}
Assume that the vector field $A(\cdot)$ satisfies \eqref{assAA}. 
If 
\eqn{bound5}
$$
\frac{q}{p}\leq  1+  \frac{p-1}{10 p}\frac{\alpha^2}{n^2}\,,
$$
then there exists a solution $u\in W^{1,p}(\Omega)$ to the Dirichlet problem  \eqref{dir1}, such that $Du$ is locally H\"older continuous in $\Omega$.
 Moreover, the estimate
\eqn{stima}
$$
\nr{Du}_{L^{\infty}(\Omega_0)}\le \frac{c}{[\dist(\Omega_0, \partial \Omega)]^{\chi_1}} \left(\int_{\Omega} (|Du_0|+1)^{\frac{p(q-1)}{p-1}}\dx +1  \right)^{ \chi_2}$$
holds whenever $\Omega_0\Subset \Omega $ is an open subset, where $c\equiv c(\data)$ and $ \chi_1, \chi_2\equiv \chi_1, \chi_2(\datae)$.
\end{theorem} 
\begin{corollary}[Hopf-Schauder-Caccioppoli, reloaded - II]\label{c4} 
In the setting of Theorem \ref{t6}, assume also that $p\geq 2$, $\mu>0$, and that $\partial_z A(\cdot)$ is continuous on $\Omega \times \er^n$. Then $
u\in C^{1, \alpha}_{\loc}(\Omega).
$\end{corollary}
When $p=q=2$, Corollary \ref{c1} is classical in the uniformly elliptic theory \cite{cacc1, hopf, gg2,gg3, js}. 
\begin{remark}\label{coti} {\em Already in the case of the classical $p$-Laplace equation $\diver\, (|Du|^{p-2}Du)=0$, starting from distributional solutions that only belong to $W^{1,s}$ for $s<p$, prevents higher regularity, and even $W^{1,p}$-regularity. This is a recent, striking achievement of Colombo \& Tione \cite{colombotione}.}
\end{remark}
\subsection{Recent, related cases available in the literature}\label{casispeciali} Schauder type estimates in the nonuniformly elliptic case attracted a lot of attention over decades, especially in the last years. For functionals with nonstandard polynomial growth of the type in \rif{FMx}, and connected  equations, recent related results hold under certain special structure assumptions \cite{AM, BCM, BO, BO0, HHT, HO, HO2}. 
These include for instance the variable exponent case $ F(x,Dw) \equiv  |Dw|^{p(x)}$, $p(x)>1$ as in \cite{AM}, and the double phase case $F(x,Dw) \equiv |Dw|^{p}+a(x)|Dw|^q$, $a(x)\geq 0$, as in \cite{BCM}. The common point of all these papers is that the frozen integrand $z\mapsto F(x_0,z)$ is still uniformly elliptic, for every fixed point $x_0 \in \Omega$. This means that \rif{ellratio} is satisfied upon taking $F(z)\equiv F(x_0,z)$. On the contrary, in this paper we deal with real, pointwise nonuniform ellipticity, allowing that
$$
\sup_{\snr{z}\geq 1}\,  \frac{\mbox{highest eigenvalue of}\ \partial_{zz}  F(x_0,z)}{\mbox{lowest eigenvalue of}\  \partial_{zz} F(x_0,z)} = \infty \,.
$$
In \rif{modello}, note that $z\mapsto  \ccc(x_0)F(z)$ is still nonuniformly elliptic in the sense of \rif{polygro} if so is $F(\cdot)$. We mention that, when applying the techniques of this paper in the known settings mentioned above, we come up with the same sharp results available in the literature \cite{DMforth}. Another approach was described in \cite{EMM}, where the authors considered integrands depending on $\snr{z}$ and with Sobolev-regular coefficients. This means that $x \mapsto F(x,\cdot)$ belongs to $W^{1,d}$, with $d\equiv d(p,q)>n$. In this case H\"older continuity of coefficients follows by Sobolev-Morrey embedding, but, again, differentiability of coefficients must be assumed; see also \cite{ciccio}. 

\section{Preliminary facts and notation} \label{notsec}
\subsection{Further notation}\label{further}
We denote $\N_{0}:= \N \cup \{0\}$, and by 
$B_{r}(x_0):= \{x \in \er^n  :   |x-x_0|< r\}$ the open ball with center $x_0$ and radius $r>0$; we omit denoting the center when it is not necessary, i.e., $B \equiv B_{r} \equiv B_{r}(x_0)$; this especially happens when various balls in the same context will share the same center. We also denote
\eqn{notfont}
$$
\mathcal B_{r} \equiv B_{r}(0):= \{x \in \er^n  :   |x|< r\}\,.
$$
Finally, with $B$ being a given ball with radius $r$ and $\gamma$ being a positive number, we denote by $\gamma B$ the concentric ball with radius $\gamma r$ and by $B/\gamma \equiv (1/\gamma)B$. Given a number $s\geq 1$, its Sobolev conjugate exponent is denoted by
\eqn{ilconiugato}
$$
s^*:= \begin{cases}
\frac{ns}{n-s} \ \mbox{if $s<n$}\\
\infty \ \mbox{if $s\geq n$}\,.
\end{cases}
$$
In denoting several function spaces like $L^s(\Omega), W^{1,s}(\Omega)$, we will denote the vector valued version by $L^s(\Omega,\er^k), W^{1,p}(\Omega,\er^k)$ in the case the maps considered take values in $\er^k$, $k\in \en$. When clear from the contest, we will also abbreviate $L^s(\Omega,\er^k) , W^{1,s}(\Omega,\er^k)\equiv 
L^s(\Omega) , W^{1,s}(\Omega)$ and so on.
With $\mathcal U \subset \er^{n}$ being a measurable subset with bounded positive measure $0<|\mathcal U|<\infty$, and with $g \colon \mathcal U \to \er^{k}$, $k\geq 1$, being a measurable map, we denote  
$$
   (g)_{\mathcal U} \equiv \mint_{\mathcal U}  g(x) \dx  :=  \frac{1}{|\mathcal U|}\int_{\mathcal U}  g(x) \dx\;,
$$
and also 
$
\|g\|_{L^{\gamma}(\mathcal U)}^\gamma= \int_{\mathcal U}|g|^\gamma \dx$, for every $\gamma \geq 0$. Given a function $h\colon A \to \er$ we define its oscillation on $A$
\eqn{defiosc}
$$
{\rm osc}_{A}\, h:= \textnormal{ess sup}_{x, y\in A} \, \snr{h(x)-h(y)}\,.
$$ 
With $\beta \in (0,1]$ and $A \subset \er^n$, we use the standard notation
$$
[w]_{0, \beta;A}:=\sup_{x, y \in A, x\not = y} \frac{|w(x)-w(y)|}{|x-y|^{
\beta}} \,.
$$ 
Given a ball $B\subset \er^n$, we denote by $Q_{\textnormal{inn}}\equiv Q_{\textnormal{inn}}(B)$ and $Q_{\textnormal{out}}\equiv Q_{\textnormal{out}}(B)$ the inner and outer hypercubes of $B$. These are defined as the largest and the smallest hypercubes, with sides parallel to the coordinate axes and concentric to $B$, that are contained and containing $B$, respectively: 
\eqn{nestate}
$$Q_{\textnormal{inn}}(B)\subset B \subset Q_{\textnormal{out}}(B)\,.$$ 
If $B$ has radius $r$, then the sidelength of $Q_{\textnormal{inn}}(B)$ is $2r/\sqrt{n}$ while that of $Q_{\textnormal{out}}(B)$ is $2r$. 

\subsection{Fractional Sobolev spaces} Classical fractional Sobolev-Slobodevski spaces are defined via Gagliardo norms as follows: 
\begin{definition}\label{fra1def}
Let $\alpha_{0} \in (0,1)$, $s \in [1, \infty)$, $k \in \en$, $n \geq 2$, and let $\Omega \subset \er^n$ be an open subset. The space $W^{\alpha_{0} ,s}(\Omega,\er^k )$ consists of maps $w \colon \Omega \to \er^k$ such that
\begin{eqnarray}
\notag
\| w \|_{W^{\alpha_{0} ,s}(\Omega )} & := &\|w\|_{L^s(\Omega)}+ \left(\int_{\Omega} \int_{\Omega}  
\frac{\snr{w(x)
- w(y)}^{s}}{|x-y|^{n+\alpha_{0} s}} \dx \dy \right)^{1/s}\\
&=:& \|w\|_{L^s(\Omega)} + [w]_{\alpha_{0}, s;\Omega} < \infty\,.\label{gaglia}
\end{eqnarray}
The local variant $W^{\alpha_{0} ,s}_{\loc}(\Omega,\er^k )$ is defined by requiring that $w \in W^{\alpha_{0} ,s}_{\loc}(\Omega,\er^k )$ iff $w \in W^{\alpha_{0} ,s}(\tilde{\Omega},\er^k)$ for every open subset $\tilde{\Omega} \Subset \Omega$. 
\end{definition}
Given $w \colon \Omega \to \mathbb{R}^{k}$, $k\ge 1$, an open subset $\Omega\subset \er^n$, and a vector $h \in \mathbb{R}^n$, we denote by $\tau_{h}\colon L^1(\Omega,\mathbb{\er}^{k}) \to L^{1}(\Omega_{|h|},\mathbb{R}^{k})$ the standard finite difference operator
$
\tau_{h}w(x):=w(x+h)-w(x)$, $\mbox{for } \ x\in \Omega_{\snr{h}},
$
where $\Omega_{|h|}:=\{x \in \Omega \, : \, 
\dist(x, \partial \Omega) > |h|\}$. We will several times use the following elementary properties of finite difference quotients (with $B_{r}(x_0)\subset \er^n$ being a fixed ball):
\eqn{diffbasic}
$$
\begin{cases}
\displaystyle \|\tau_{h} w\|_{L^s(B_{r}(x_0))}  \leq c(s) \|w\|_{L^s(B_{r+|h|}(x_0))}   & \quad  \forall \, w \in L^s(B_{r+|h|}(x_0)), \ s \geq 1\\
\displaystyle \|\tau_{h} w\|_{L^s(B_{r}(x_0))}   \leq c(n,s)|h| \|Dw\|_{L^s(B_{r+|h|}(x_0))}& \quad \forall \, w \in W^{1,s}(B_{r+|h|}(x_0)), \ s \geq 1\,.
\end{cases}
$$
Finite difference operators can be used to detect maps from fractional Sobolev spaces, as described in the following Lemma, that in fact quantifies, locally, the imbedding properties of Nikolski spaces into Sobolev-Slobodevski spaces $W^{\alpha_0,s}$ (see \cite{dm1}):
\begin{lemma}\label{l4}
Let $B_{\varrho} \Subset B_{r}\subset \er^n$ be concentric balls with $r\leq 1$, $w\in L^{s}(B_{r},\mathbb{R}^{k})$, $s\geq 1$ and assume that, for $\beta \in (0,1]$, $\mathcal H\ge 1$, there holds
\eqn{cru1}
$$
\nr{\tau_{h}w}_{L^{s}(B_{\rr})}\le \mathcal H\snr{h}^{\beta } \quad \mbox{
for every $h\in \mathbb{R}^{n}$ with $0<\snr{h}\le \frac{r-\rr}{K}$, where $K \geq 1$}\;.
$$
Then, for $c\equiv c(n,s)$, it holds that 
\eqn{cru2}
$$
\nr{w}_{W^{\alpha_{0},s}(B_{\rr})}\le\frac{c}{(\beta -\alpha_{0})^{1/s}}
\left(\frac{r-\rr}{K}\right)^{\beta -\alpha_{0}}\mathcal H+c\left(\frac{K}{r-\rr}\right)^{n/s+\alpha_{0}} \nr{w}_{L^{s}(B_{r})}\,.
$$
\end{lemma}
In the case the domain considered is the ball $\BBB$ (which is the only one needed here) the fractional Sobolev embedding reads as
\eqn{immersione}
$$
\nr{w}_{L^{\frac{ns}{n-s\alpha_{0}}}(\BBB)}\leq c\nr{w}_{W^{\alpha_{0},s}(\BBB)}\,,
$$
that holds provided $s\geq 1, \alpha_{0} \in (0,1)$ and $s\alpha_{0}<n$, where $c\equiv c (n,p)$. 
\subsection{Additional preliminary results} We will use the vector field $V_{\mu}\colon \er^n \to \er^n$ defined by 
\eqn{defiV}
$$
V_{\mu}(z):=(\snr{z}^{2}+\mu^2)^{(p-2)/4}z 
$$
where $\mu \in [0, 2]$ and $p>1$ is defined in Section \ref{basicn}. Whenever 
$z_{1},z_{2}\in \mathbb{R}^{n}$, there holds that
\eqn{ineV}
$$
\snr{V_{\mu}(z_{1})-V_{\mu}(z_{2})}^{2}\approx_p (\snr{z_{1}}^{2}+\snr{z_{2}}^{2}+\mu^{2})^{(p-2)/2}\snr{z_{1}-z_{2}}^{2},
$$
see \cite[Lemma 2.1]{ha}. As a consequence we find
\eqn{sopradue}
$$
\snr{z_{1}-z_{2}}^{p}\lesssim_p  \snr{V_{\mu}(z_{1})-V_{\mu}(z_{2})}^{2} +  \una\snr{V_{\mu}(z_{1})-V_{\mu}(z_{2})}^p(\snr{z_1}+\mu)^{p(2-p)/2}\,,
$$
where 
\eqn{unap}
$$
\una:=\begin{cases}
0 \  \mbox{if} \   p\geq 2 \\
1 \ \mbox{if $p<2$}\,.
\end{cases}
$$ 
When $p\geq 2$,  \rif{sopradue} follows directly from \rif{ineV}. Instead, for $1<p<2$ inequality in \rif{sopradue} follows mimicking the proof of \cite[Lemma 2]{KMstein}, which is given in the case $\mu=0$. We also record the following inequality, which is a direct consequence of mean value theorem:
\eqn{disH}
$$
\snr{[H_{\mu}(z_2)]^{p/2}-[H_{\mu}(z_1)]^{p/2}} \leq p (\snr{z_{1}}^{2}+\snr{z_{2}}^{2}+\mu^{2})^{(p-1)/2}|z_2-z_1|\,.
$$
See \rif{defiH} for the definition of $H_{\mu}(\cdot)$. Combining this last inequality with \rif{ineV} yields
\eqn{diffH}
$$
\snr{[H_{\mu}(z_2)]^{p/2}-[H_{\mu}(z_1)]^{p/2}} \lesssim_p
(\snr{z_1}^{2}+\snr{z_2}^{2}+\mu^2)^{p/4}\snr{V_{\mu}(z_1)-V_{\mu}(z_2)}\,.
$$
The last algebraic result of elementary nature we include is the following:
\eqn{elemint}
$$
\int_{0}^1 [H_{\mu}(tz_1+(1-t)z_2)]^{s}\dt \approx_{s} (|z_1|^2+|z_2|^2+\mu^2)^{s/2}  
$$
that holds whenever $s>-1$, $\mu \in [0,2]$ and $z_1, z_2\in \er^n$; see for instance \cite{ha}. Next, two classical iteration lemmas from Campanato, and Giaquinta \& Giusti.
\begin{lemma} \label{l5}\, \hspace{-2.5mm} {\em \cite[Lemma 1.1]{gg1}}
\, Let $h\colon [t,s]\to \mathbb{R}$ be a non-negative and bounded function, and let $a,b, \gamma$ be non-negative numbers. Assume that the inequality 
$ 
h(\tau_1)\le  (1/2) h(\tau_2)+(\tau_2-\tau_1)^{-\gamma}a+b,
$
holds whenever $t\le \tau_1<\tau_2\le s$. Then $
h(t)\le c( \gamma)[a(s-t)^{-\gamma}+b]
$, holds too. 
\end{lemma}
\begin{lemma}\label{l5bis} \hspace{-2.5mm} {\em \cite[Lemma 2.2]{gg2}}
\, Let $h\colon [0,r_0]\to \mathbb{R}$ be a non-negative and non-decreasing function such that the inequality 
$
h(\rr) \leq a[ (\rr/r )^{\beta_*} +\eps]h(r) + br^{\beta},
$
holds whenever $0 \leq \rr \leq r\leq r_0$, where $a,b$ are positive constants, and $0 < \beta < \beta_*$. There exists $\eps_0 \equiv \eps_0(a,\beta_*, \beta)$ such that, if $\eps\leq \eps_0$, then 
$
h(\rr) \leq c(\rr/r )^{\beta}  [h(r) + br^{\beta}]
$
holds too, whenever $0 \leq \rr \leq r\leq r_0$, where $c \equiv c (a,\beta_*, \beta)$. 
\end{lemma}
Finally, a convexity lemma of Marcellini, that follows as in \cite[Lemma 2.1]{M1}. 
\begin{lemma}\label{marclemma} Let $\FF\colon \Omega\times \er\times \er^n\to [0, \infty)$ be a Carath\'eory function such that $z \mapsto \FF(x,y, z)$ is convex for every $(x,y) \in  \Omega\times \er$ and that $\FF(x,y,z) \leq c_*[H_{\mu}(z)]^{q/2}+c_*[H_{\mu}(z)]^{p/2}$, for some $c_*\geq 1$. Then 
$| \partial_z\FF(x,y, z)| \leq c(n,p,q,c_*)([H_{\mu}(z)]^{(q-1)/2}+[H_{\mu}(z)]^{(p-1)/2})$ holds for every $(x,y,z)\in \Omega \times \er \times \er^n$. 
\end{lemma}
\section{Nonlinear potentials, Lorentz spaces, and iterations}\label{losec}
Let us fix $t,\delta>0$, $m, \theta\geq 0$, $f\in L^{1}(B_{r}(x_{0}))$ such that $\snr{f}^{m} \in L^{1}(B_{r}(x_0))$ with $B_{r}(x_0)\subset \er^n$. The following quantity will play a crucial role in this paper:
\eqn{defi-P} 
$$
{\bf P}_{t,\delta}^{m,\theta}(f;x_0,r) := \int_0^r \varrho^{\delta} \left(  \mint_{B_{\varrho}(x_0)} \snr{f}^{m} \dx \right)^{\theta/t} \frac{\d\varrho}{\varrho} \;.
$$
This is a nonlinear potential of Havin-Mazya-Wolff type \cite{HM}, used for instance in \cite{BM, cammello, ciccio,KMstein}. ${\bf P}_{t,\delta}^{m,\theta}(f, \cdot)$ can be seen as a nonlinear generalization of the classical Riesz potential ${\bf I}_1(f, \cdot)$, that in its so-called truncated form is in turn defined as
$$
\int_{B_{r/2}(x_0)}\frac{|f(x)|}{|x-x_0|^{n-1}}\dx \lesssim {\bf I}_1(f;x_0,r) := \int_0^r \frac{1}{\varrho^{n-1}}\int_{B_{\rr}(x_0)} \snr{f} \dx\,   \frac{\d\varrho}{\varrho} \approx {\bf P}_{1,1}^{1,1}(f;x_0,r) \,.
$$
The variety of parameters considered in \rif{defi-P} makes ${\bf P}_{t,\delta}^{m,\theta}$ useful in settings where ${\bf I}_1$ cannot be directly employed, as we will see via its use in this paper. We refer to \cite[Section 2]{BM} for more details and references.
Note that 
\eqn{trivialpot}
$$
{\bf P}_{t,\delta}^{m,\theta}(f;x_0,r) =r^\delta/\delta\,, \  \mbox{when $m=0$ or when $\snr{f}\equiv 1$ or when $\theta =0$}\,.
$$
The usual definition of Lorentz space $L(s,\gamma)(\Omega)$, with $\gamma,s\in (0,\infty)$, prescribes that the quantity \begin{flalign*}
\|f\|_{s, \gamma;\Omega} &= \left(\int_0^{\infty} \left(f^*(\varrho)\varrho^{1/s}\right)^\gamma\, \frac{d\varrho}{\varrho}\right)^{1/\gamma}= \left(s\int_0^\infty (\lambda^s|\{x \in \Omega \, : \, |f(x)|> \lambda\}|)^{\gamma/s}\, \frac{d\lambda}{\lambda}\right)^{1/\gamma}
\end{flalign*}
is finite, see \cite[1.4.2]{gra}; other basic references in this setting are \cite{oneil, sw}. Here  $f^*\colon [0,\infty]\to \er$ denotes the non-increasing rearrangement of $f$, i.e.,  
$f^*(\tau):=\inf\, \{v> 0 \, : \, |\{x \in \Omega \, : \, |f(x)|> v\}|\leq \tau\}$. 
 In Lorentz spaces, the latter index tunes the former in the sense that, when $\Omega\subset \mathbb{R}^{n}$ has finite measure, it then holds that
 \eqn{lorentzbasic}
$$
\begin{cases}
\, \mbox{$L(s_{1}, \gamma_{1})(\Omega)\subset L(s_{2}, \gamma_{2})(\Omega)$ for all $0<s_{2}<s_{1}<\infty$, $\gamma_{1},\gamma_{2}\in (0,\infty]$}\\
\, \mbox{$L(s,\gamma_{1})(\Omega)\subset L(s,\gamma_{2})(\Omega)$ for all $s\in (0,\infty)$, $0<\gamma_{1}\leq \gamma_{2}\le\infty$}\\
\,  \mbox{$L(s,s)(\Omega)=L^{s}(\Omega)$ for all $s >0$\,,}
 \end{cases}
$$
with continuous inclusions. When $s>1$, it is possible to define another quantity, which is equivalent to $\|f\|_{s, \gamma;\Omega}$, by considering the maximal operator of $f^{*}$, that is
\eqn{massimale} 
$$
f^{**}(\rr):=\frac{1}{\rr}\int_{0}^{\rr}f^{*}(\tau) \d\tau \quad \mbox{for $\rr>0$}\,.$$ 
It turns out that
\begin{flalign}\label{lo.2}
\|f\|_{s, \gamma;\Omega}^\gamma\le \int_{0}^{\infty}\left(f^{**}(\varrho)\varrho^{1/s}\right)^{\gamma} \ \frac{\d\varrho}{\varrho}\le c\|f\|_{s, \gamma;\Omega}^\gamma\,,
\end{flalign}
holds provided $s>1$, where $c\equiv c(\gamma,s)$; see \cite[(6.8)]{oneil}. 
Nonlinear potentials and Lorentz spaces naturally connect, as for instance shown  in the next lemma; see \cite{BM, ciccio, KMstein} for similar results. 
\begin{lemma}\label{crit}
Let $n \geq 2$, $t,\delta,m,\theta>0$ be numbers such that
\begin{flalign}\label{lo.2.1}
\frac{n\theta }{t\delta}>1\,.
\end{flalign}
Let $B_{\rr}\Subset B_{\rr+r_{0}}\subset \mathbb{R}^{n}$ be two concentric balls with $\varrho, r_0\leq 1$, and let $f\in L^{1}(B_{\rr+r_{0}})$ be such that $\snr{f}^{m}\in L^{1}(B_{\rr+r_{0}})$. Then
\eqn{11}
$$
\nr{\mathbf{P}^{m,\theta}_{t,\delta}(\cdot,r_{0})}_{L^{\infty}(B_{\rr})}\le \tilde c\|f\|_{\frac{mn\theta}{t\delta},\frac{m\theta}{t};B_{\rr+r_{0}}}^{\frac{m\theta}{t}} \leq c(\eps)\tilde c\|f\|_{L^{\frac{(1+\eps)mn\theta}{t\delta}}(B_{\rr+r_{0}})}^{\frac{m\theta}{t}}
$$
holds for every $\eps>0$, with $\tilde c\equiv \tilde c(n,t,\delta,\theta)$. 
\end{lemma}
\begin{proof}
Basic properties of rearrangements give
\begin{eqnarray}\label{10}
\varrho^{t\delta/\theta}\mint_{B_{\varrho}(x_{0})}\snr{f}^{m}  \dx \leq \frac{\varrho^{t\delta/\theta}}{|\BB|\varrho^{n}}\int_{0}^{|\BB|\varrho^{n}}(\snr{f}^{m})^{*}(\tau) \ \d\tau\stackleq{massimale} \varrho^{t\delta/\theta}(\snr{f}^{m})^{**}(|\BB|\varrho^{n})\,.
\end{eqnarray}
We further estimate
\begin{eqnarray*}
\mathbf{P}_{t,\delta}^{m,\theta}(x_{0},r_{0})& \stackrel{\eqref{10}}{\le}& \int_{0}^{r_{0}}[\varrho^{t\delta/\theta}(\snr{f}^{m})^{**}(|\BB|\varrho^{n})]^{\theta/t} \ \frac{\d\varrho}{\varrho}\le c\int_{0}^{\infty}[\varrho^{\frac{t\delta}{n\theta}}(\snr{f}^{m})^{**}(\varrho)]^{\theta/t} \ \frac{\d\varrho}{\varrho} \\
&\stackrel{\eqref{lo.2}}{\le} &c\|\snr{f}^{m}\|_{\frac{n\theta}{t\delta},\frac{\theta}{t};B_{\varrho+r_{0}}(x_{0})}^{\theta/t}=c(n,t,\delta,\theta)\|f\|_{\frac{mn\theta}{t\delta},\frac{m\theta}{t};B_{\varrho+r_{0}}(x_{0})}^{\frac{m\theta}{t}}
\end{eqnarray*}
so that the first inequality in \eqref{11} follows (recall \rif{lorentzbasic} and see \cite{sw} for the second). 
\end{proof}
Next lemma extends \cite[Lemma 3.1]{BM}. It features a pointwise version of classical De Giorgi's iteration that finds its origins in the work in Nonlinear Potential Theory of Kilpel\"ainen \& Mal\'y \cite{kilp}. We report the full proof as the crucial point here is the explicit dependence on the constants. 
\begin{lemma}\label{revlem}
Let $B_{r_{0}}(x_{0})\subset \mathbb{R}^{n}$ be a ball, $n\ge 2$, and $v\in L^t$ and consider functions $f_i$, 
$
|f_i|^{m_i} \in L^1(B_{2r_0}(x_{0}))$, and constants $\chi >1$, $t \geq 1$, $\delta_i, m_i, \theta_i>0$ and $c_*,M_0 >0$,  $\kappa_0, M_i\geq 0$, 
for $i \in \{1, \ldots, h\}$, $h \in \en$. Assume that $v \in L^t(B_{r_0}(x_0))$ is such that for all $\kk\ge \kk_{0}$, and for every ball $B_{\rr}(x_{0})\subset B_{r_{0}}(x_{0})$, the inequality
\begin{flalign}\label{rev}
\notag \left(\mint_{B_{\rr/2}(x_{0})}(v-\kk)_{+}^{t\chi}  \dx\right)^{1/\chi}&\le c_{*}M_{0}^{t}\mint_{B_{\rr}(x_{0})}(v-\kk)_{+}^{t}  \dx\\  &
\qquad +c_{*}\sum_{i=1}^{h} M_{i}^{t}\rr^{t\delta_{i}}\left(\mint_{B_{\rr}(x_{0})}\snr{f_i}^{m_i}  \dx\right)^{\theta_i}
\end{flalign}
holds, where we denote, as usual, $ (v-\kk)_{+}:=\max\{v-\kk,0\}$. If $x_{0}$ is a Lebesgue point of $v$, then
\eqn{3388}
$$
 v(x_{0})  \le\kk_{0}+cM_{0}^{\frac{\chi}{\chi-1}}\left(\mint_{B_{r_{0}}(x_{0})}(v-\kk_{0})_{+}^{t}  \dx\right)^{1/t}
+cM_{0}^{\frac{1}{\chi-1}}\sum_{i=1}^{h} M_{i}\mathbf{P}^{m_i,\theta_i}_{t,\delta_{i}}(f_i;x_{0},2r_{0})
$$
holds with $c\equiv c(n,\chi,\delta_i,\theta_i,c_{*})$.  
\end{lemma}
\begin{proof} We can assume that the right-hand side in \rif{3388} is finite, otherwise, there is nothing to prove. In the following all the balls will be centered at $x_0$. 
We define radii $\{\rr_{j}\}_{j\in \N_{0}}$, where $\rr_{j}:=r_{0}/2^{j}$ integer $j\geq 0$ (so that $\rr_{0}=r_{0}$), and, for every $i \leq h$, numbers $\{W_{i,j}\}_{j\in \N_{0}}$, via
\eqn{reww}
$$
W_{i,j}:=\rr_{j}^{\delta_{i}}\left(\mint_{B_{\rr_{j}}}\snr{f_i}^{m_i}  \dx\right)^{\theta_i/t}\,.
$$
The next two sequences of numbers $\{\kk_{j}\}_{j\in \N_{0}}$ and $\{V_{j}\}_{j\in \N_{0}}$ are defined inductively, with $\kk_{0}$ given by the statement. With $\kk_j$ having been defined, we set $V_{j}$ and then $\kk_{j+1}$ as follows:
\begin{flalign}\label{3}
V_{j}:=\left(\mint_{B_{\rr_{j}}}(v-\kk_{j})_{+}^{t}  \dx\right)^{1/t}\,, \qquad \kk_{j+1}:=\kk_{j}+V_{j}/\tau,
\end{flalign}
where $\tau>0$ is going to determined in due course of the proof as a function of $n,c_{*},\chi,t$; see \rif{tautau} below. It follows that $\{\kappa_j\}_j$ is non-decreasing, and  $V_{j+1}\le 2^{n/t}V_{j}$; therefore
\begin{flalign}\label{4}
\kk_{j+2}-\kk_{j+1}\le 2^{n/t}(\kk_{j+1}-\kk_{j}),
\end{flalign}
holds for every $j \geq 0$. Using \eqref{rev} and the definitions in \rif{reww}, yields, for every $j \geq 0$
$$
\left(\mint_{B_{\rr_{j+1}}}(v-\kk_{j})_{+}^{t\chi}  \dx\right)^{1/\chi}\le c_{*}M_{0}^{t}V_{j}^{t}  \nonumber +c_{*}\sum_{i=1}^{h} M_{i}^{t}W_{i,j}^{t}
\,.
$$
By $\kk_{j+1}\geq \kk_j$ for every $j$ and \eqref{3}, we estimate
\begin{flalign*}
&(\kk_{j+1}-\kk_{j})^{(\chi-1)/\chi}V_{j+1}^{1/\chi}=(\kk_{j+1}-\kk_{j})^{(\chi-1)/\chi}\left(\mint_{B_{\rr_{j+1}}}(v-\kk_{j+1})_{+}^{t}  \dx\right)^{\frac{1}{t\chi}}\nonumber \\
&\qquad \qquad \leq \left(\mint_{B_{\rr_{j+1}}}(v-\kk_{j})_{+}^{t\chi-t}(v-\kk_{j+1})_{+}^{t}  \dx\right)^{\frac{1}{t\chi}}\nonumber \leq\left(\mint_{B_{\rr_{j+1}}}(v-\kk_{j})^{t\chi}_{+}  \dx\right)^{\frac{1}{t\chi}}\,.
\end{flalign*}
Recalling the definition in \rif{reww}, the last two displays combine in 
\begin{flalign}\label{5}
(\kk_{j+1}-\kk_{j})^{(\chi-1)/\chi}V_{j+1}^{1/\chi}\le c_{*}^{1/t}M_{0}V_{j}+c_{*}^{1/t}\sum_{i=1}^{h} M_{i}W_{i,j}\,.
\end{flalign}
Now, in the case it is
\begin{flalign}\label{6}
\kk_{j+2}-\kk_{j+1}\ge \frac{1}{2}(\kk_{j+1}-\kk_{j})
\end{flalign}
we deduce
\begin{flalign}
 2^{-1/\chi}\tau^{1/\chi}(\kk_{j+1}-\kk_{j})&\stackrel{\eqref{6}}{\le} \tau^{1/\chi}(\kk_{j+1}-\kk_{j})^{(\chi-1)/\chi}(\kk_{j+2}-\kk_{j+1})^{1/\chi}\nonumber \\
& \stackrel{\eqref{3}}{=}(\kk_{j+1}-\kk_{j})^{(\chi-1)/\chi}V_{j+1}^{1/\chi}\nonumber\\
&
\stackrel{\eqref{5}}{\le}c_{*}^{1/t}M_{0}V_{j}+c_{*}^{1/t}\sum_{i=1}^{h} M_{i}W_{i,j} \nonumber \\
&\stackrel{\eqref{3}}{=}c_{*}^{1/t}M_{0}\tau(\kk_{j+1}-\kk_{j})+c_{*}^{1/t}\sum_{i=1}^{h} M_{i}W_{i,j}\,. \label{8}
\end{flalign}
We now take  
\eqn{tautau}
$$\tau=2^{-\frac{1+\chi}{\chi-1}}M_{0}^{-\frac{\chi}{\chi-1}}c_*^{-\frac{\chi}{t(\chi-1)}}\Longrightarrow  
2^{1/\chi}c_*^{1/t}M_{0} \tau^{\frac{\chi-1}{\chi}} =\frac 12
$$ 
so that, reabsorbing the first term in the right-hand side in the left-hand side of \rif{8}, and recalling \rif{4}, we get
\begin{flalign*}
\kk_{j+2}-\kk_{j+1}\le c^{*}M_{0}^{\frac{1}{\chi-1}}\sum_{i=1}^{h} M_{i}W_{i,j}, \quad   c^{*}:= 2^{\frac nt+\frac{\chi+1}{\chi-1}}c_*^{\frac{\chi}{t(\chi-1)}}\,.
\end{flalign*}
As this last inequality holds under condition \eqref{6}, we can work in any case with
\begin{eqnarray*}
\kk_{j+2}-\kk_{j+1}\le \frac{\kk_{j+1}-\kk_{j}}{2}+c^{*}M_{0}^{\frac{1}{\chi-1}}\sum_{i=1}^{h} M_{i}W_{i,j}
\end{eqnarray*}
for all integers  $j \geq 0$. Summing up such inequalities for $0\leq j \leq N$, re-absorbing terms, and then letting $N \to \infty$, gives
\begin{flalign*}
\sum_{j=0}^{\infty}(\kk_{j+2}-\kk_{j+1})\le \kk_{1}-\kk_{0}+2c^*M_{0}^{\frac{1}{\chi-1}}
\sum_{i=1}^{h} \sum_{j=0}^{\infty}M_{i}W_{i,j}\,,
\end{flalign*}
so that, recalling also the definitions of $V_0$ and $\tau$, in \rif{3} and \rif{tautau}, respectively, 
\begin{flalign}
\lim_{j\to \infty}\kk_{j}&=\sum_{j=0}^{\infty}(\kk_{j+2}-\kk_{j+1})+\kk_{1}\nonumber \le\kk_{0}+2(\kk_{1}-\kk_{0})+2c^*M_{0}^{\frac{1}{\chi-1}}
\sum_{i=1}^{h} \sum_{j=0}^{\infty}M_{i}W_{i,j}\nonumber \\
&\le\kk_{0}+\frac{2V_{0}}{\tau}+2c^*M_{0}^{\frac{1}{\chi-1}}
\sum_{i=1}^{h} \sum_{j=0}^{\infty}M_{i}W_{i,j}\le\kk_{0}+2c^*M_{0}^{\frac{\chi}{\chi-1}}V_{0}+2c^*M_{0}^{\frac{1}{\chi-1}}
\sum_{i=1}^{h} \sum_{j=0}^{\infty}M_{i}W_{i,j}\,.\label{7788}
\end{flalign}
Setting $\rr_{-1}:=2r_{0}$, we have, for every $i \leq h$, 
\begin{flalign}
\sum_{j=0}^{\infty}W_{i,j}&=\sum_{j=0}^{\infty}\rr_{j}^{\delta_{i}}\left(\mint_{B_{\rr_{j}}}\snr{f_i}^{m_i}  \dx\right)^{\theta_i/t}=\frac{\delta_{i}}{2^{\delta_{i}}-1}\sum_{j=0}^{\infty}\int_{\rr_{j}}^{\rr_{j-1}}\rr^{\delta_{i}}\frac{\d\rr}{\rr}\left(\mint_{B_{\rr_{j}}}\snr{f_i}^{m_i}  \dx\right)^{\theta_i/t}\nonumber \\
&\le\frac{2^{n\theta_{i}/t}\delta_{i}}{2^{\delta_{i}}-1}\sum_{j=1}^{\infty}\int_{\rr_{j}}^{\rr_{j-1}}\varrho^{\delta_{i}}\left(\mint_{B_{\varrho}(x_{0})}\snr{f_i}^{m_i}  \dx\right)^{\theta_i/t}\frac{\d\varrho}{\varrho}\nonumber \\
&\le\frac{2^{n\theta_{i}/t}\delta_{i}}{2^{\delta_{i}}-1}\int_{0}^{2r_{0}}
\varrho^{\delta_{i}}\left(\mint_{B_{\varrho}(x_{0})}\snr{f_i}^{m_i}  \dx\right)^{\theta_i/t}\frac{\d\varrho}{\varrho}=\frac{2^{n\theta_{i}/t}\delta_{i}}{2^{\delta_{i}}-1}\mathbf{P}^{m_i,\theta_{i}}_{t,\delta_{i}}(f_i;x_{0},2r_{0})\,.\label{9}
\end{flalign}
The content of the last three displays implies that $\{\kappa_j\}_j$ converges to a finite limit and therefore $V_j\to 0$ by its definition in \rif{3}. Using this fact and 
$x_0$ is a Lebesgue point of $v$, we have 
$$
v(x_{0})
\le\limsup_{j\to \infty}\left(\mint_{B_{\rr_{j}}}(v-\kk_{j})_{+}^{t}  \dx\right)^{1/t}+\limsup_{j\to \infty}\kk_{j}\nonumber\leq \lim_{j\to \infty}V_{j}+\lim_{j\to \infty}\kk_{j}=\lim_{j\to \infty}\kk_{j}
$$
and \rif{3388} follows by using \rif{7788}-\eqref{9} to estimate the right-hand side in the above display. 
\end{proof}

\section{Hybrid Fractional Caccioppoli inequalities}\label{ibrida}
In this section we provide various Caccioppoli type inequalities for minima of variational integrals and solutions to nonlinear equations. These are basic tools in order to prove Lipschitz estimates.  Although the underlying principle is common to all cases, the specific shape of these inequalities varies according to the setting considered. The basic prototype is provided in Propositions \ref{caccin1}-\ref{caccin2}. These contain fractional Caccioppoli type estimates of hybrid type. The word hybrid accounts this time for the fact that on the right-hand hand sides of \rif{40prima} and \rif{40}, there still appears the $L^\infty$-norm of $Du$, (implicit in the presence of $M$). Inequalities \rif{40prima} and \rif{40} will be proved under the form of a priori estimates for minima of more regular, uniformly elliptic integrals with standard polynomial growth. They will be then incorporated in a suitable approximation scheme in order to cover the case of nonuniformly elliptic integrals. Propositions \ref{caccin1}-\ref{caccin2} refer to functionals of the type in \rif{ggg2}. Later on, in Propositions \ref{caccin3}-\ref{caccin5}, we will present additional Caccioppoli inequalities, valid also for functionals of the type in \rif{modellou} and \rif{FMx}, and for general elliptic equations of the type in \rif{equazioneM}. Although the basic scheme of proofs is the same, the estimation of the various terms must be different in each case, as every particular type of structure considered needs a specific treatment, eventually leading to different bounds on $q/p$. In the rest of the whole Section \ref{ibrida}, $B_{r}\Subset \Omega$ will always denote a ball such that $r\leq1$. Moreover, when dealing with minimizers, we will always assume that $p$ and $q$ satisfy at least the bound
\eqn{marcbound}
$$
\frac qp < 1+\frac{2}{n}\,.
$$
Accordingly, we consider the number 
\eqn{marcexp}
$$
\mathfrak {s} := \frac{2q}{(n+2)p-nq}\geq 1\,,
$$
which in fact is well-defined by \rif{marcbound}. Note that $\mathfrak{s}=1$ if and only if $p=q$. 
\subsection{Model Caccioppoli estimates}\label{laprima}
In this section we consider the functional $\mathcal{G} (\cdot, B_{r})$ in \rif{ggg2}. The integrand $F\colon \mathbb{R}^{n}\to \mathbb{R}$ is assumed to satisfy stronger conditions that those considered in \rif{assFF}, namely
\begin{flalign}\label{32.2}
\qquad\begin{cases}
\, F(\cdot) \in C^{2}(\er^n)\\
\, \nu_0 [H_{\mu}(z)]^{q/2}+\tilde \nu[H_{\mu}(z)]^{p/2}\le F(z)\le \tilde L[H_{\mu}(z)]^{q/2}+\tilde L[H_{\mu}(z)]^{p/2} \qquad \\
\,   \nu_0 [H_{\mu}(z)]^{(q-2)/2}\snr{\xi}^{2}+\tilde \nu [H_{\mu}(z)]^{(p-2)/2}\snr{\xi}^{2}\le \partial_{zz}F(z)\xi\cdot \xi \\
\, \snr{\partial_{zz}F(z)}\le \tilde L[H_{\mu}(z)]^{(q-2)/2}+\tilde L[H_{\mu}(z)]^{(p-2)/2}, \qquad 0< \mu \leq 2\,,
\end{cases}
\end{flalign}
for all $z,\xi\in \mathbb{R}^{n}$ and where $\nu_0>0$ and the numbers $0< \tilde \nu \leq \tilde L$ are as in Section \ref{basicn}. The function $\gggg\colon B_{r} \times \mathbb{R}\times \er^n \mapsto \mathbb{R}$ satisfies the following reinforcement of \rif{assgg}:
\eqn{assgg2}
$$
\begin{cases}
\, \mbox{$\gggg(\cdot)$ satisfies \trif{assgg} with $\mu$ as in \rif{32.2}$_4$ and $L$ replaced by $\tilde L$ in \rif{32.2}
}\\
\, \mbox{$z \mapsto \gggg(x,y,z) \in C^{2}(\er^n)$ for every $(x,y)\in \Omega \times \er$} \,.
 \end{cases}
$$
We will denote by $c_{\gggg}$ a generic constant, depending on parameters that will be specified, but such that $c_{\gggg}=0$ when $\gggg$ is identically zero. Finally, the function $\hhh\colon B_{r}\times \mathbb{R}\mapsto \mathbb{R}$ is such that
\eqn{asshh22}
$$
\begin{cases}
\, \snr{\hhh(x,y_{1})-\hhh(x,y_{2})}\le f(x)\snr{y_{1}-y_{2}}^{\alpha},
\,  |f(x)| \leq L_0, \ \ \alpha \in (0,1]\\
\, \snr{\hhh(x,y)} \leq L_0(|y|+1)
\end{cases}
$$
hold for all $x\in B_{r}$, $y,y_{1},y_{2}\in \mathbb{R}$, where $ L_{0}\geq 1$ is a fixed constant. Denoting $\tilde F (x,y,z):=F(z)+\gggg(x,y,z)$, using also that $\mu>0$, this new integrand is seen to satisfy
\begin{flalign}\label{F_0}
\qquad\begin{cases}
\, \frac1{c_*} [H_{1}(z)]^{q/2}\le \tilde F(x,y,z)\le c_*[H_{1}(z)]^{q/2}\\
\,  \frac1{c_*} [H_{1}(z)]^{(q-2)/2}\snr{\xi}^{2}\le \partial_{zz}\tilde  F(x,y,z)\xi\cdot \xi\,,  \quad \snr{\partial_{zz}\tilde  F(x,y,z)}\le c_*[H_{1}(z)]^{(q-2)/2}\\
\, \snr{\tilde F(x_{1},y_{1},z)-\tilde F(x_{2},y_{2},z)}\le c_*\left(\snr{x_{1}-x_{2}}^{\alpha}+\snr{y_{1}-y_{2}}^{\alpha}\right)[H_{1}(z)]^{\gamma/2}
\end{cases}
\end{flalign}
with the same meaning of \rif{32.2}-\rif{asshh22}, but this time we have $c_*\equiv c_*(n,q,L,\mu,\nu_0)$.  
In \rif{F_0} the constant $c$ is such that $c_* \to \infty$ when either $\mu \to 0$ or $ \nu_0 \to 0$. Therefore we have that $\tilde F(\cdot)$ is a regular and non-degenerate integrand with $q$-growth. It follows that if $u\in W^{1,q}(B_{r})$ is a minimzer of the functional  $\mathcal{G} (\cdot, B_{r})$ in \rif{ggg2} and \rif{32.2}-\rif{asshh22} are in force, then standard regularity arguments give 
\begin{flalign}\label{33.1bis}
u\in C^{1,\alpha_{1}}_{\loc}(B_{r})\quad \mbox{for some}  \ \alpha_{1}\equiv \alpha_{1}(n,q,c_*, \alpha, L_0,\mu )\in (0,1)\,.
\end{flalign} 
This is shortly detailed in Section \ref{giustifica}. 
We shall also use a couple of parameters $(\beta, \chi)$ such that
\eqn{bechi}
$$
0<\beta<\frac{\alpham}{2+\alpham} \qquad  \mbox{and} \qquad \chi\equiv \chi (\beta):=\frac{n}{n-2\beta}>1\,,
$$
where $\alpham\leq 1$ is a positive number that will be later specified. 
By denoting 
\eqn{Hnot}
$$
E_{\mu}(z):=\frac{1}{p}\left[(\snr{z}^{2}+\mu^{2})^{p/2}-\mu^{p}\right]= \frac{1}{p}\left[[H_{
\mu}(z)]^{p/2}-\mu^{p}\right], \qquad z\in \er^n
$$
we now have 
\begin{proposition}[Hybrid Fractional Caccioppoli]\label{caccin1}
Let $u\in W^{1,q}(B_{r})$ be a minimizer of the functional in $\mathcal G(\cdot, B_{r})$ in \trif{ggg2}, under assumptions \eqref{marcbound} and \trif{32.2}-\eqref{asshh22}. Let $B_{\rr}(x_{0}) \Subset B_{r}$ and let $M\geq 1$ be a constant such that $ \nr{Du}_{L^{\infty}(B_{\rr}(x_{0}))}\leq M$. Then, for every number $\kk \geq 0$
\begin{eqnarray}\label{40prima}
\notag &&\varrho^{2\beta-n}[(E_{\mu}(Du)-\kk)_{+}]_{\beta,2;B_{\varrho/2}(x_0)}^2+ \left(\mint_{B_{\rr/2}(x_{0})}(E_{\mu}(Du)-\kk)_{+}^{2\chi}  \dx\right)^{1/\chi}\\
&&\qquad    \le cM^{\ssf(q-p)}\mint_{B_{\rr}(x_{0})}(E_{\mu}(Du)-\kk)_{+}^{2}  \dx +c_{\gggg}M^{\ssf q+ \alpha+\gamma q/p}\rr^{\alpha}+cM^{\ssf q+\alpha}\rr^{\alpha} \mint_{B_{\rr}(x_{0})}f \dx
 \end{eqnarray}
holds whenever $(\beta, \chi)$ are as in \trif{bechi} with $\alpham:=\alpha$. 
The constants $c, c_{\gggg}$ in \eqref{40prima} depend on $\data, \gamma$ and $\beta$, but are otherwise independent of $\nu_0$; dependence on $\gamma$ only occurs when $\gggg(\cdot)\not \equiv 0$. The number $\ssf$ is defined in \trif{marcexp}.\end{proposition}
An alternative, and more flexible version of \rif{40prima}, is in the next
\begin{proposition}\label{caccin2}
Under the same assumptions of Proposition \ref{caccin1}, for every number $\kk \geq 0$ 
\begin{eqnarray}\label{40}
\notag &&\varrho^{2\beta-n}[(E_{\mu}(Du)-\kk)_{+}]_{\beta,2;B_{\varrho/2}(x_0)}^2+\left(\mint_{B_{\rr/2}(x_{0})}(E_{\mu}(Du)-\kk)_{+}^{2\chi}  \dx\right)^{1/\chi}\\&& \notag
\qquad  \le cM^{\mathfrak{s}(q-p)}\mint_{B_{\rr}(x_{0})}(E_{\mu}(Du)-\kk)_{+}^{2}  \dx  +c_{\gggg}M^{\ssf  q+\alpha+\gamma q/p}\rr^{\alpha} \\&&\qquad  \quad+cM^{\ssf  q}\rr^{\frac{p\alpha}{p-\alpha}}\left(\mint_{B_{\rr}(x_{0})}f^{\frac{\ppp}{\ppp-\alpha}}  \dx\right)^{\theta(\ppp)}
+c\una M^{\ssf  q+\frac{\alpha(2-p)}{2-\alpha}}\rr^{\frac{2\alpha}{2-\alpha}}\left(\mint_{B_{\rr}(x_{0})}f^{\frac{\ppp}{\ppp-\alpha}} \dx\right)^{\sigma(\ppp)} \end{eqnarray}
holds whenever $\ppp\in [p, p^*)$,  
where
\eqn{tetesisi}
$$ 
\theta(\ppp)
:= \frac{\ppp-\alpha}{\ppp}\frac{p}{p-\alpha}\geq 1 \qquad  \mbox{and} \qquad   \sigma(\ppp)
:= \frac{\ppp-\alpha}{\ppp}\frac{2}{2-\alpha}\,,
$$ 
and whenever $(\beta, \chi)$ are as in \trif{bechi}, with this time
\eqn{alphagen}
$$
 \alpham \equiv  \alpham(\ppp) := \alpha 
\min\left\{1,\frac{p\mathfrak{a}(\ppp)}{p-\alpha}, \frac{2\mathfrak{a}(\ppp)}{2-\alpha}\right\} \quad \mbox{and}
\quad \mathfrak{a}(\ppp) :=\frac n\ppp- \frac np+1\,.
$$
The constants $c, c_{\gggg}, \ssf$ are as in Proposition \ref{caccin1} and $\una$ is defined in \trif{unap}.
\end{proposition}
\begin{remark}\label{differente}{\em 
From the proof of Proposition \ref{caccin2} it follows that in the case $\gggg(\cdot)\equiv 0$, so that $c_{\gggg}=0$ in \rif{40}, then the definition of $\alpham$ in \rif{alphagen} can be changed into
\eqn{alphagen2}
$$
\alpham = \alpha 
\min\left\{\frac{p\mathfrak{a}(\ppp)}{p-\alpha}, \frac{2\mathfrak{a}(\ppp)}{2-\alpha}\right\}\,.$$ }
\end{remark}
\subsection{Preliminaries on equations.}
For the proof of Propositions  \ref{caccin1}-\ref{caccin2}, we need a few preliminary results concerning {\em non-degenerate} elliptic equations of the type
\eqn{solvina}
$$
-\diver\, A_0(Dv)=0 \qquad \mbox{in $B\subset \er^n$}\,,
$$
where $B$ is a ball. 
Here the vector field $A_0\colon \er^{n} \mapsto \er^n$ satisfies
\eqn{bebe}
$$
\begin{cases}
\, A_0(\cdot) \ \mbox{is $C^1$-regular} \\
\,  | A_0(z)|+[H_{\mu}(z)]^{1/2}| \partial_{z}A_0(z)| \leq \tilde L[H_{\mu}(z)]^{(q-1)/2}+\tilde L[H_{\mu}(z)]^{(p-1)/2}\\
\,\nu_0 [H_{\mu}(z)]^{(q-2)/2}|\xi|^2+\tilde \nu[H_{\mu}(z)]^{(p-2)/2}|\xi|^2\leq \partial_{z}A_0(z)\xi\cdot \xi,  \quad 0 < \mu \leq 2\,,
\end{cases} 
$$
whenever $z, \xi\in \er^n$, and where $0 < \tilde \nu \leq  \tilde L$ are as in Section \ref{basicn}. The assumption $\mu>0$ implies that $A_0(\cdot)$ is non-degenerate elliptic. 
\begin{lemma}\label{cacc-class1}
Let $v\in W^{1,q}(B)$ be a weak solution to \eqref{solvina} under assumptions \trif{bebe}, then $
 [H_{\mu}(Dv)]^{p/2}, E_{\mu}(Dv)\in W^{1,2}_{\loc}(B)$, and $Dv$ is locally H\"older continuous in $B$. Moreover, assume also that there exists $\mathfrak{M}\geq 1$ such that $ \nr{Dv}_{L^{\infty}(B)}\leq\mathfrak{M}$. 
\begin{itemize}
\item Then
\eqn{pri-cacc-2}
$$\int_{B/2}  \snr{D(E_{\mu}(Dv)-\kk)_{+}}^{2}  \dx \leq \frac{c\mathfrak{M}^{2(q-p)}}{|B|^{2/n}}
\int_{B}  (E_{\mu}(Dv)-\kk)_{+}^{2}  \dx
$$ 
holds for every $\kk \geq 0$, where $c\equiv c (n,p, q, \nu, L)$, and is otherwise independent of $\nu_0$.
\item If, in addition, $\partial A_0(\cdot)$ is symmetric, then
\eqn{pri-cacc}
$$\int_{B/2}  \snr{D(E_{\mu}(Dv)-\kk)_{+}}^{2}  \dx \leq \frac{c\mathfrak{M}^{q-p}}{|B|^{2/n}}
\int_{B}  (E_{\mu}(Dv)-\kk)_{+}^{2}  \dx
$$ 
holds with $c, \kappa$ as in \trif{pri-cacc-2}. 
\end{itemize}
\end{lemma}
\begin{proof} The standard regularity theory gives that $Dv$ is locally H\"older continuous in $\Omega$, as well as the differentiability of $ H_{\mu}(Dv)$, and therefore of $E_{\mu}(Dv)$. For the first result  see for instance \cite{manth1} and Lemma \ref{manflemma}. As for the second, let us outline the argument for completeness. Standard difference quotients techniques give that $V_{\mu}(Du)\in W^{1,2}_{\loc}(B,\er^n)$; see for instance \cite[Chapter 8]{giu}, and recall that $\mu>0$ to adapt from there (still observe that by approximation $\mu>0$ is not really needed at this stage). By \rif{ineV} we have 
$\snr{\tau_h Dv}^{2} \lesssim (\snr{Dv(\cdot+h)}^{2}+\snr{Dv}^{2}+\mu^{2})^{(2-p)/2}\snr{\tau_hV_{\mu}(Dv)}^{2}$ in $\{x \in B \, : \, 
\dist(x, \partial B) > |h|\}$, so that $\mu>0$ and $Dv\in L^{\infty}_{\loc}(B,\er^n)$ imply $Dv\in W^{1,2}_{\loc}(B,\er^n)$ via difference quotients. In turn, \rif{disH} implies
$
\snr{\tau_{h}[H_{\mu}(Dv)]^{p/2}}^2 \lesssim (\snr{Dv(\cdot+h)}^2+\snr{Dv}^2+1)^{p-1}\snr{\tau_{h}Dv}^2
$, from which $[H_{\mu}(Dv)]^{p/2}\in W^{1,2}_{\loc}(B)$ again follows via difference quotients.  
We go for the proof of \rif{pri-cacc-2}-\rif{pri-cacc} and we first consider \rif{pri-cacc}. By \rif{bebe} we are in the setting of  \cite[Lemma 4.5]{BM}; in particular, by $\mu>0$, $A_0(\cdot)$ satisfies \cite[(4.26)]{BM} with $\bar T=0$ and 
\eqn{adatta}
$$
\begin{cases}
\, g_{2,\eps}(t)\equiv g_{2}(t)\equiv \tilde L(t^2+\mu^2)^{(q-2)/2}+\tilde L(t^2+\mu^2)^{(p-2)/2} \\
\, g_1(t)\equiv  \tilde \nu(t^2+\mu^2)^{(p-2)/2}\,.
\end{cases}
$$
Note that, with respect to \cite{BM}, we have that $g_{2,\eps}(\cdot)$ is actually independent of $\varepsilon$ as indicated in \rif{adatta} (that in fact appears only in the context of \cite{BM}, where a family of vector fields is considered). Moreover, thanks to \rif{bebe},  we can in fact take $\bar T=0$ (alternatively, take any $\bar T\in (0,\mathfrak{M})$ in \cite{BM} and then let $\bar T\to 0$ in \cite{BM}). 
Again by \rif{bebe}, $v$ enjoys the regularity in \rif{33.1bis} and we can use  \cite[(4.29)]{BM} with $f\equiv 0$, that gives, with the notation of \cite{BM}
\eqn{lavecchia}
$$
 \int_{B/2} \snr{D(G_{\bar T}(\snr{Dv}) - k)_+}^2 \dx
  \leq \frac{c(n)}{|B|^{2/n}} \frac{g_{2}(\mathfrak{M})}{g_{1}(\mathfrak{M})} \int_B (G_{\bar T}(\snr{Dv}) - k)_+^2  \dx\,.$$
In this setting (recall $\bar T=0$) it is
  $$G_{\bar T}(\snr{z}):=\int_{0}^{\snr{z}}g_{1}(s)s \d s =\frac{\tilde \nu}{p} \left[(\snr{z}^{2}+\mu^{2})^{p/2}-\mu^{p}\right]=\tilde \nu E_{\mu}(\snr{z})$$
for $z\in \er^n$. Recalling \rif{adatta} and that $\mathfrak{M}\geq 1$, we estimate
\eqn{adatta2}
$$\frac{g_{2}(\mathfrak{M})}{g_{1}(\mathfrak{M})}\leq c(\mathfrak{M}^{q-p}+1)\leq c \mathfrak{M}^{q-p}\,,$$  
with $c \equiv c (n,p,q, \tilde \nu, \tilde L) $, so that \rif{pri-cacc} follows from \rif{lavecchia}. For \rif{pri-cacc-2}, we use \cite[(4.37)]{BM}, that is 
$$
 \int_{B/2} \snr{D(G_{\bar T}(\snr{Dv}) - k)_+}^2 \dx
  \leq \frac{c}{|B|^{2/n}}\left[\frac{g_{2}(\mathfrak{M})}{g_{1}(\mathfrak{M})}\right]^2  \int_B  (G_{\bar T}(\snr{Dv}) - k)_+^2  \dx  \,,$$
and \rif{pri-cacc-2} follows via \rif{adatta2}. For the constant dependence, recall that $\tilde \nu, \tilde L$ depend on $n,p,q,\nu, L$. 
 \end{proof}
\begin{lemma}\label{oscilemma}Let $v\in W^{1,q}(B)$ be a weak solution to $\diver\, A(x,Dv)=0$,  such that $v \in u+ W^{1,q}_0(B)$, with $u \in W^{1,q}(B)\cap L^\infty(B)$. Assume that $z \mapsto A(x, z)\equiv A_0(z)$ satisfies \trif{bebe} uniformly with respect to $x\in B$. Then 
$
\osc (v,B) \leq \osc (u,B).
$
\end{lemma}
\begin{proof}
This is a small variant of the classical maximum principle, see for instance \cite[Lemma 3.1]{manth1}. We sketch the proof for completeness. By \rif{bebe} it follows that
$(\snr{z_{1}}^{2}+\snr{z_{2}}^{2}+\mu^{2})^{(q-2)/2}\snr{z_{1}-z_{2}}^{2} \lesssim  [A(x,z_2)-A(x,z_1)]\cdot (z_2-z_1)$. 
Note that we can assume $A_0(0_{\er^n})=0_{\er^n}$, so that  
$(\snr{z}^{2}+\mu^{2})^{(q-2)/2}\snr{z}^{2} \lesssim A(x,z)\cdot z$. Let $M, m$ be such that $m\leq u\leq M$ a.e. in $B$, and use $\varphi := (v-M)_+$ as test function in $\int_B A_0(Dv)\cdot D \varphi\dx=0$, which is valid whenever $\varphi \in W^{1,q}_0(B)$. It follows that $D(v-M)_+\equiv 0$ a.e. and therefore $v \leq M$ a.e. follows via Poincaré inequality. Similarly, testing with $\varphi := (v-m)_-:=\max\{m-v,0\}$, with $m :=\inf_B u$, we get $v \geq m$ a.e., and the proof is complete.   
\end{proof} 
We proceed with a few a priori estimates that will play a central role in our analysis. These are suitable modifications of estimates that can be found in the literature since \cite{M1, M2}. To get the full statements we need, we will appeal to \cite{BM}.  
\begin{lemma}\label{marcth}
Let $v\in W^{1,q}(B)$ be a weak solution to \eqref{solvina}, under assumptions \trif{bebe}. \begin{itemize}
\item If $q/p < 1+ 1/n$, then
\eqn{stimaprimissima}
$$
\|Dv\|_{L^{\infty}(B/2)} \leq c \left(\mint_{B}(\snr{Dv}+1)^p\dx\right)^{ \frac{1}{(n+1)p-nq}}
$$
holds whenever $B\Subset \Omega \subset \er^n$ is a ball, where $c\equiv c(n,p,q,\nu, L)$ is  independent of $\nu_0$.
\item If $q/p < 1+ 2/n$, that is \trif{marcbound} holds, and in addition $\partial A_0(\cdot)$ is symmetric, then 
\eqn{stimaprimissima2}
$$
\|Dv\|_{L^{\infty}(B/2)} \leq c \left(\mint_{B}(\snr{Dv}+1)^p\dx\right)^{ \frac{2}{(n+2)p-nq}}
$$
holds in place of \trif{stimaprimissima}. 
\end{itemize}
\end{lemma}
\begin{proof} We start proving \rif{stimaprimissima}-\rif{stimaprimissima2}. The information concerning the local boundedness of $Dv$ can be found for instance in \cite{BM}, but it is also a direct consequence of the arguments developed to prove \rif{stimaprimissima}-\rif{stimaprimissima2} in the following lines, when using arbitrary balls $\tilde B\Subset B$ rather than $B/2$. So, we concentrate on the proof of \rif{stimaprimissima}-\rif{stimaprimissima2}. 
First, note that we can reduce to the case $B\equiv \BB$ by a standard scaling argument (which is for instance detailed in Section \ref{riscalasezione} a few lines below). The proof of \rif{stimaprimissima} can be obtained following the estimates in \cite[Lemma 7.2]{BM}, where we must take $f\equiv 0$ and $\Omega \equiv \BB$, and the choices are those in \rif{adatta}; as in \cite{BM}, we first treat the case $n>2$. Specifically, using \cite[(4.38)]{BM} with $\vartheta=1$ , we arrive at
\eqn{stimaux1}
$$
\|Dv\|_{L^{\infty}(\mathcal B_{\tau_1})}^p  \leq \frac{c}{(\tau_2-\tau_1)^{n/2}} \left[
\|Dv\|_{L^\infty(\mathcal B_{\tau_2})}^{\frac{(q-p)n+p}{2}}+1\right]\left[\|Dv\|_{L^p(\mathcal B_1)}^{p/2}+1 \right]\,,
$$
which holds whenever $1/2 \leq \tau_1 < \tau_2 \leq 1$. 
This is nothing but the estimate in the second display at \cite[Page 1026]{BM}, where $\mathcal D$ is replaced by $\|Dv\|_{L^p(\mathcal B_1)}+1$. Using $q/p < 1+ 1/n$ - implying $[(q-p)n+p]/2<p$ - allows to apply Young's inequality in \rif{stimaux1} and 
$$
\|Dv\|_{L^{\infty}(\mathcal B_{\tau_1})} \leq \frac 12 \|Dv\|_{L^{\infty}(\mathcal B_{\tau_2})}  +  \frac{c}{(\tau_2-\tau_1)^{ \frac{n}{(n+1)p-nq}}} \left[\|Dv\|_{L^p(\mathcal B_1)}+1 \right]^{ \frac{p}{(n+1)p-nq}}
$$
from which \rif{stimaprimissima} follows - modulo rescaling back to $B$ - applying Lemma \ref{l5} with the choice $h(\tau)\equiv  \|Du\|_{L^{\infty}(B_{\tau})}$, $1/2 \leq \tau \leq 1$. For the case $n=2$, again from \cite[(4.38)]{BM} we find that \rif{stimaux1} still holds replacing the exponent $[(q-p)n+p]/2= q-p/2$ with any larger number, chosen to still be smaller that one, and the conclusion follows as in the case $n>2$. For \rif{stimaprimissima2} the argument is similar (and we outline it when $n>2$).
We again go back to \cite[Lemma 7.2]{BM} and apply
\cite[(4.38)]{BM} with $\vartheta=0$, that this time gives the following analog of \rif{stimaux1}:
$$
\|Dv\|_{L^{\infty}(\mathcal B_{\tau_1})}^p  \leq \frac{c}{(\tau_2-\tau_1)^{n/2}} \left[
\|Dv\|_{L^\infty(\mathcal B_{\tau_2})}^{\frac{(q-p)n+2p}{4}}+1\right]\left[\|Dv\|_{L^p(\mathcal B_1)}^{p/2}+1 \right]\,.
$$ 
Applying Young's inequality as in \rif{stimaux1}, thanks to \trif{marcbound}, we arrive at \rif{stimaprimissima2}. The case $n=2$ can again be dealt with as for \rif{stimaprimissima}.\end{proof}
\subsection{Proof of Propositions \ref{caccin1}-\ref{caccin2}}\label{primacaccio} The proof goes in five different steps; in the first four, given in Sections \ref{riscalasezione}-\ref{hyb2}, we complete the proof of \rif{40}. The last one, in Section \ref{hybprima}, is instead devoted to the proof of \rif{40prima}. 
\subsubsection{Scaling}\label{riscalasezione}
In order to prove Proposition \ref{caccin1}, we can reduce to the case $B_{\varrho}(x_0)\equiv B_1(0)\equiv \BB$ (recall the notation in \rif{notfont}). Indeed, take $u_{\rr}\in W^{1,q}(\BB)$ defined as 
\eqn{riscalata}
$$u_{\rr}(x):=u(x_{0}+\rr x)/\rr\,, \qquad x \in \BB\,.$$
It follows that 
\eqn{ereditato}
$$
\|Du_\varrho\|_{L^{\infty}(\BB)} \leq M
$$
and that $u_{\rr}$ is a minimizer of the functional (defined on $W^{1,q}(\BB)$)
\eqn{definito}
$$
w\mapsto
 \mathcal{G}_{\varrho}(w, \mathcal{B}_1) :=
\int_{\BB} [F(Dw)+\gggg_{\rr}(x,w,Dw) +\hhh_{\rr}(x,w)]  \dx\,,
$$
where $\gggg_{\rr}(x,y,z):=\gggg(x_0+\rr x,\rr y,z)$ and $\hhh_{\rr}(x,y):=\hhh(x_0+\rr x,\rr y)$. By \eqref{assgg2} we have
\begin{flalign}\label{assaar}
\begin{cases}
\, z\mapsto \gggg_{\rr}(x,y,z) \  \mbox{is convex, non-negative and $z\mapsto \gggg_{\rr}(x,y,z)\in C^{2}(\mathbb{R}^{n})$}\\
\, \gggg_{\rr}(x,y,z)+H_{\mu}(z)\snr{\partial_{zz}\gggg_{\rr}(x,y,z)}\le \tilde L[H_{\mu}(z)]^{p/2}\\
\, \snr{\gggg_{\rr}(x_{1},y_{1},z)-\gggg_{\rr}(x_{2},y_{2},z)}\le \tilde L\varrho^{\alpha}\left(\snr{x_{1}-x_{2}}^{\alpha}+\snr{y_{1}-y_{2}}^{\alpha}\right)[H_{1}(z)]^{\gamma/2}
\end{cases}
\end{flalign}
for all $x,x_{1},x_{2}\in \BB$, $y,y_{1},y_{2}\in \mathbb{R}$, $z\in \mathbb{R}^{n}$. As for $\hhh_{\rr}(\cdot)$, by \rif{asshh22} we instead have
\eqn{assgr}
$$
\ \snr{\hhh_{\rr}(x,y_{1})-\hhh_{\rr}(x,y_{2})}\le f_{\rr}(x)\snr{y_{1}-y_{2}}^{\alpha},\qquad f_{\rr}(x):=\rr^{\alpha}f(x_0+\rr x)
$$
hold for all $x \in \BB$, $y_{1},y_{2}\in \mathbb{R}$. As $B_{\rr}(x_0) \Subset B_{r}$, \rif{33.1bis} implies $u_{\varrho}\in C^{1,\alpha_1}(\overline{\BB})$. Therefore we only need to prove that
\begin{flalign}\label{40-res}
&\notag [(E_{\mu}(Du_{\varrho})-\kk)_{+}]_{\beta,2;\BBB}+ \|(E_{\mu}(Du_{\varrho})-\kk)_{+}\|_{L^{2\chi}(\BBB)}\\
\notag & \qquad  \le cM^{\ssf (q-p)/2}\|E_{\mu}(Du_{\varrho})-\kk)_{+}\|_{L^2(\BB)} +c_{\gggg}M^{(\ssf q+ \alpha+\gamma q/p)/2}\varrho^{\alpha/2}\\  \qquad & \qquad \qquad +cM^{\frac {\ssf q}2}\|f_{\rr}\|_{L^{\frac{\ppp}{\ppp-\alpha}}(\BB)}^{\frac{\ppp\theta(\ppp)}{2(\ppp-\alpha)}} +c\una M^{\frac {\ssf q}2+\frac{\alpha(2-p)}{2(2-\alpha)}}
\|f_{\rr}\|_{L^{\frac{\ppp}{\ppp-\alpha}}(\BB)}^{\frac{\ppp\sigma(\ppp)}{2(\ppp-\alpha)}} 
\end{flalign}
holds for $c, c_{\gggg}$ as in Proposition \ref{caccin1}, and then \rif{40} follows scaling back \rif{40-res} from $u_{\varrho}$ to $u$. 
\subsubsection{Estimates on balls}\label{hyb0} Consider a ball $B \subset \BB$, centered at $x_{\rm c}$, and define the functional 
\eqn{congela}
$$
w \mapsto \int_{B} F_{0}(Dw)\dx \qquad \mbox{where} \qquad  F_{0}(z):=F(z)+\gggg_{\rr}(x_{\rm c},(u_{\rr})_{B},z) \,.
$$
As a consequence of \rif{32.2},\rif{assaar} and Lemma \ref{marclemma}, and eventually choosing new $\tilde \nu \leq \tilde L$ still as in Section \ref{basicn}, the vector field $\partial_zF_0(\cdot)\equiv A_0(\cdot)$ satisfies \rif{bebe}. Therefore, following \cite{ha, KM}, we have that 
\eqn{hamine}
$$
\frac{1}{\tilde c}\snr{V(z_2)-V(z_1)}^{2} \leq  F_{0}(z_2)- F_{0}(z_1)-\partial_z F_{0}(z_1)\cdot(z_2-z_1)
$$
holds whenever $z_1, z_2 \in \er^n$, for $\tilde c\equiv \tilde c(n,p, \tilde \nu)\geq 1$. 
From now on, and in the rest of Section \ref{appi}, we will abbreviate $E_{\mu}(\cdot)\equiv E(\cdot)$, $V_{\mu}(\cdot)\equiv V(\cdot)$ and $H_{\mu}(\cdot)\equiv H(\cdot)$. To proceed with the proof \rif{40-res}, we fix $\beta_{0}\in (0,1)$, and $h\in \mathbb{R}^{n}$, such that \begin{flalign}\label{hhh}
0< \snr{h}\leq  \frac{1}{2^{8/\beta_{0}}}\,.
\end{flalign}
We take 
\eqn{sceltacentro}
$$x_{\rm c}\in \mathcal B_{1/2+ 2|h|^{\beta_0}}$$ 
and fix a ball centered at $x_{\rm c}$ with radius $\snr{h}^{\beta_{0}}$, denoted by $B_{h}:= B_{\snr{h}^{\beta_{0}}}(x_{\rm c})$. By \eqref{hhh} we have $8B_{h}\Subset \BB$. We then define $v\equiv v_{B_{h}}\in u_{\rr}+W^{1,q}_0(8B_{h})$ as the solution to
\begin{flalign}\label{pd}
 v\mapsto \min_{w \in u_{\rr}+W^{1,q}_0(8B_{h})} \int_{8B_{h}} F_{0}(Dw)  \dx
\end{flalign}
where the integrand $ F_{0}(\cdot)$ has been fixed in \rif{congela}, with $B \equiv B_{h}$, and verifies \rif{F_0}. 
It follows that $v$ solves the Euler-Lagrange equation
\eqn{euleri}
$$\int_{8B_{h}}\partial_z F_{0}(Dv)\cdot D\varphi  \dx=0 \qquad \mbox{for every  $\varphi \in W^{1,q}_0(8B_{h})$\,.}
$$
By \rif{32.2} and \rif{assgg2}, the vector field $\partial_zF_0(\cdot)\equiv A_0(\cdot)$ satisfies \rif{bebe} for suitable constants $0 < \tilde \nu \leq \tilde L$ depending as described in Section \ref{basicn}, and, needless to say, $\partial_zA_0(\cdot)\equiv\partial_{zz}F_0(\cdot)$ is symmetric. Therefore equation \rif{euleri} is of the type considered in \rif{solvina} and, via \rif{marcbound}, Lemmas \ref{cacc-class1}-\ref{marcth} apply.  
In particular, Lemma \ref{oscilemma} implies
\eqn{conosc}
$$
\osc (v,8B_{h}) \leq \osc (u_{\rr},8B_{h}) \,.
$$
As a consequence of \rif{ereditato} and \rif{conosc}, we find
\eqn{conosc2}
$$
\begin{cases}
\, 
\snr{v -(v)_{8B_{h}}} \leq \osc (u_{\rr}, 8B_{h}) \leq 16M|h|^{\beta_0}\\ 
\,  \|u_{\rr}-v\|_{L^{\infty}(8B_h)}\leq  2\osc (u_{\rr}, 8B_{h}) \leq 32M|h|^{\beta_0}\,.
\end{cases}
$$
Note that, in order to obtain \rif{conosc2}$_2$, we also use that $v$ is continuous up to the boundary $\partial (8B_h)$. In fact, since $u_{\rr}$ is Lipschitz regular in  $\BB$ by \rif{ereditato}, it follows that $v$ is H\"older continuous (for some exponent) up to the boundary, by classical Maz'ya-Wiener theory for equations with standard $q$-growth and degeneracy. See for instance \cite{kilp} or \cite[Theorem 7.8]{giu}, in order to obtain the result from the theory of Quasi Minima, as indicated in \cite[Chapter 6]{giu}; see also \cite{li}. Also note that such a result is here used only in a qualitative form, not in any quantitative one. Next, applying Lemma \ref{marcth}, estimate \rif{stimaprimissima2}, yields
\eqn{globale0}
$$
\|Dv\|_{L^{\infty}(4B_{h})} \leq c \left(\mint_{8B_{h}}(\snr{Dv}+1)^p\dx\right)^{\frac{2}{(n+2)p-nq}}\,.
$$
On the other hand, by minimality of $v$, \rif{32.2}$_2$ and \rif{assgg2}, we have 
\eqn{globale}
$$
\mint_{8B_{h}}\snr{Dv}^p\dx \leq c\mint_{8B_{h}}F_0(Dv)\dx \leq c\mint_{8B_{h}}F_0(Du_{\rr})\dx \stackleq{ereditato} c(M^q+1)\leq c M^q\,.
$$
Matching the content of the last two inequalities, yields
\eqn{fuckbsc}
$$
\|Dv\|_{L^{\infty}(4B_{h})} \leq \tilde cM^{\frac{2q}{(n+2)p-nq}}\stackrel{\rif{marcexp}}{=} \tilde c M^{\mathfrak {s}}\,, \qquad \tilde c \equiv \tilde c (n,p,q,\nu, L)\,.
$$
Finally, by \rif{fuckbsc} we apply Lemma \ref{cacc-class1} with $\mathfrak{M}\equiv \tilde cM^{\mathfrak{s}}$ as follows: 
\begin{eqnarray}\label{cacc}
\nonumber \int_{B_{h}}\snr{\tau_{h}(E(Dv)-\kk)_{+}}^{2}&  \stackrel{\eqref{diffbasic}_2}{\leq} &c |h|^2\int_{2B_h}  \snr{D(E(Dv)-\kk)_{+}}^{2}  \dx
\\
&\stackleq{pri-cacc}  & c\snr{h}^{2(1-\beta_{0})}M^{\mathfrak {s}(q-p)}\int_{4B_{h}}(E(Dv)-\kk)_{+}^{2}  \dx\,,
\end{eqnarray}
for $c\equiv c(n,p,q,\nu,L)$. Note that in the first line we have used \rif{hhh}, that ensures
\eqn{includi}
$$B_h\equiv B_{|h|^{\beta_0}}(x_{\rm c}) \subset B_{|h|^{\beta_0}+|h|}(x_{\rm c})\subset B_{2|h|^{\beta_0}}(x_{\rm c})=2B_h\,.$$ 
Let us now quantify the $L^{2}$-distance between $V(Du_{\rr})$ and $V(Dv)$. Recalling that $u_{\rr}$ minimizes $ \mathcal{G}_{\varrho}(\cdot, \BB)$ defined in \rif{definito}, we have
\begin{flalign}
\notag& \hspace{-5mm}\frac 1{\tilde c}\int_{8B_{h}}\snr{V(Du_{\rr})-V(Dv)}^{2}  \dx\\
& \stackleq{hamine} \int_{8B_{h}}\left[ F_{0}(Du_{\rr})- F_{0}(Dv)-\partial_z F_{0}(Dv)\cdot(Du_{\rr}-Dv)\right]  \dx \nonumber \\
\notag& \stackrel{\rif{euleri}}{=} \int_{8B_{h}}\left[ F_{0}(Du_{\rr})- F_{0}(Dv)\right]  \dx\\
\notag&  \ \,= \mathcal{G}_{\varrho}(u_{\rr}, 8B_{h})- \mathcal{G}_{\varrho}(v, 8B_{h}) \nonumber
+\int_{8B_{h}}[\gggg_{\rr}(x_{\rm c},(u_{\rr})_{8B_{h}},Du_{\rr})-\gggg_{\rr}(x,u_{\rr},Du_{\rr})]  \dx\\ 
\notag& \quad \quad  +\int_{8B_{h}}[\gggg_{\rr}(x,v,Dv)-\gggg_{\rr}(x_{\rm c},(u_{\rr})_{8B_{h}},Dv)]\dx +\int_{8B_{h}}[\hhh_{\rr}(x,v)-\hhh_{\rr}(x,u_{\rr})]\dx\\ 
&\, \ \,  =:\stackrel{\leq 0}{\left[ \mathcal{G}_{\varrho}(u_{\rr}, 8B_{h})- \mathcal{G}_{\varrho}(v, 8B_{h})\right]} +\mbox{(I)}+\mbox{(II)}+\mbox{(III)}\leq \mbox{(I)}+\mbox{(II)}+\mbox{(III)}\,. \label{mmaa}
\end{flalign}
We proceed estimating (I), using \eqref{assaar}$_{3}$ as follows:
\begin{flalign*}
\mbox{(I)}& \leq c_{\gggg}\varrho^{\alpha}\int_{8B_{h}}\left(\snr{x-x_{\rm c}}^{\alpha}+\snr{u_{\rr}-(u_{\rr})_{8B_{h}}}^{\alpha}\right)(\snr{Du_{\rr}} +1)^{\gamma}  \dx \\
& \leq c_{\gggg}\snr{h}^{\beta_{0}\alpha}M^{\gamma}\varrho^{\alpha}\snr{B_{h}} +c_{\gggg} \snr{h}^{\beta_{0}\alpha}M^{\alpha+\gamma}\varrho^{\alpha}\snr{B_{h}}
\\ & \leq c_{\gggg}\snr{h}^{\beta_{0}\alpha}M^{\alpha+\gamma}\varrho^{\alpha}\snr{B_{h}}\leq 
c_{\gggg}\snr{h}^{\beta_{0}\alpha}M^{\alpha+\gamma q/p}\varrho^{\alpha}\snr{B_{h}}\,.
\end{flalign*}
For (II) we have
\begin{eqnarray*}
\mbox{(II)} & \stackrel{\eqref{assaar}_{3}}{\leq} & c_{\gggg}\varrho^{\alpha}\int_{8B_{h}}\left(\snr{x-x_{\rm c}}^{\alpha}+\snr{v-(v)_{8B_{h}}}^{\alpha}+\snr{(v)_{8B_{h}}-(u_{\rr})_{8B_{h}}}^{\alpha}\right)(\snr{Dv}+1)^{\gamma}  \dx\\
& \stackrel{\eqref{conosc2}}{\leq} &c_{\gggg} |h|^{\alpha\beta_0}M^{\alpha} \rr^{\alpha}\mint_{8B_h}(\snr{Dv}+1)^{\gamma}  \dx\, |B_h|\\
& \leq &c_{\gggg} |h|^{\alpha\beta_0}M^{\alpha} \rr^{\alpha}\left(\mint_{8B_h}(\snr{Dv}+1)^{p}  \dx\right)^{\gamma/p}|B_h|\\
&  \stackleq{globale}  &  c_{\gggg}\snr{h}^{\beta_{0}\alpha}M^{\alpha+\gamma q/p}\varrho^{\alpha}\snr{B_{h}}\,.
\end{eqnarray*}
For (III) we use Sobolev and Morrey embeddings in the form
\eqn{banalina}
$$
\|u_{\rr}-v\|_{L^{\ppp}(8B_{h})} \leq c|h|^{\beta_0\mathfrak{a}(\ppp)}\|Du_{\rr}-Dv\|_{L^{p}(8B_{h})} \,,
$$
where $\mathfrak{a}(\ppp)$ has been defined in \rif{alphagen}. Note that, according to the definition \rif{ilconiugato}, inequality \rif{banalina} holds for every $\ppp\geq p$ when $p \geq n$. 
Using also \rif{sopradue}, \rif{assgr} and \rif{banalina}, we find
\begin{flalign}
\notag \mbox{(III)} &
\le c\int_{8B_{h}}f_{\rr}(x)\snr{u_{\rr}-v}^{\alpha}  \dx\\ 
& \notag  \leq  c\|f_{\rr}\|_{L^{\frac{\ppp}{\ppp-\alpha}}(8B_{h})}\|u_{\rr}-v
\|_{L^{\ppp}(8B_{h})}^{\alpha} \\
\notag&\le c \snr{h}^{\beta_{0}\alpha\mathfrak{a}(\ppp)}\|f_{\rr}\|_{L^{\frac{\ppp}{\ppp-\alpha}}(8B_{h})}\|Du_{\rr}-Dv
\|_{L^{p}(8B_{h})}^{\alpha} \\
&\le c\snr{h}^{\beta_{0}\alpha\mathfrak{a}(\ppp)}\|f_{\rr}\|_{L^{\frac{\ppp}{\ppp-\alpha}}(8B_{h})}\left(\int_{8B_{h}}\snr{V(Du_{\rr})-V(Dv)}^{2}  \dx\right)^{\alpha/p}\notag \\
& \nonumber \quad +c\una \snr{h}^{\beta_{0}\alpha\mathfrak{a}(\ppp)}\|f_{\rr}\|_{L^{\frac{\ppp}{\ppp-\alpha}}(8B_{h})}\left(\int_{8B_{h}}\snr{V(Du_{\rr})-V(Dv)}^{p} (\snr{Du_{\rr}}+\mu)^{p(2-p)/2} \dx\right)^{\alpha/p}\\
&\le c\snr{h}^{\beta_{0}\alpha\mathfrak{a}(\ppp)}\|f_{\rr}\|_{L^{\frac{\ppp}{\ppp-\alpha}}(8B_{h})}\left(\int_{8B_{h}}\snr{V(Du_{\rr})-V(Dv)}^{2}  \dx\right)^{\alpha/p}\notag \\
& \ \ \  +c\una \snr{h}^{\beta_{0}\alpha\mathfrak{a}(\ppp)}M^{\frac{\alpha(2-p)}{2}} |h|^{\frac{\beta_0n}{p}\frac{\alpha(2-p)}{2}}\|f_{\rr}\|_{L^{\frac{\ppp}{\ppp-\alpha}}(8B_{h})} \left(\int_{8B_{h}}\snr{V(Du_{\rr})-V(Dv)}^{2}\dx\right)^{\alpha/2}\notag\\
& \leq \frac{1}{2\tilde c} \int_{8B_{h}}\snr{V(Du_{\rr})-V(Dv)}^{2}\dx+c\snr{h}^{\frac{\beta_{0}\alpha p\mathfrak{a}(\ppp)}{p-\alpha}}\|f_{\rr}\|_{L^{\frac{\ppp}{\ppp-\alpha}}(8B_{h})}^{\frac{\ppp\theta(\ppp)}{\ppp-\alpha}} \notag 
\\ & \quad +c\una\snr{h}^{\frac{\beta_{0}\alpha2\mathfrak{a}(\ppp)}{2-\alpha}}M^{\frac{\alpha(2-p)}{2-\alpha}}|h|^{\frac{\beta_0n}{p}\frac{\alpha(2-p)}{2-\alpha}}\|f_{\rr}\|_{L^{\frac{\ppp}{\ppp-\alpha}}(8B_{h})}^{\frac{\ppp\sigma(\ppp)}{\ppp-\alpha}}\,,
\label{stimaIII}
\end{flalign}
where $\theta(\ppp), \sigma(\ppp)$ are defined in \rif{tetesisi}; in the last line we have used Young's inequality twice.  
Using the estimates found for (I),(II),(III) in \rif{mmaa}, and reabsorbing terms, we come up with
\begin{flalign}
\notag \int_{8B_{h}}\snr{V(Du_{\rr})-V(Dv)}^{2}  \dx & \le c_{\gggg}\snr{h}^{\beta_{0}\alpha}M^{\alpha+\gamma q/p}\varrho^{\alpha}\snr{B_{h}}
+c\snr{h}^{\frac{\beta_{0}\alpha p\mathfrak{a}(\ppp)}{p-\alpha}}\|f_{\rr}\|_{L^{\frac{\ppp}{\ppp-\alpha}}(8B_{h})}^{\frac{\ppp\theta(\ppp)}{\ppp-\alpha}}
\\ & \qquad +c\una\snr{h}^{\frac{\beta_{0}\alpha2\mathfrak{a}(\ppp)}{2-\alpha}}M^{\frac{\alpha(2-p)}{2-\alpha}}|h|^{\frac{\beta_0n}{p}\frac{\alpha(2-p)}{2-\alpha}}\|f_{\rr}\|_{L^{\frac{\ppp}{\ppp-\alpha}}(8B_{h})}^{\frac{\ppp\sigma(\ppp)}{\ppp-\alpha}}\label{37}
\end{flalign}
where $c, c_{\gggg}\equiv c, c_{\gggg}(n,p, q,\nu, L)$. We now use this last inequality to bound $E(Du_{\rr})-E(Dv)$ in $L^2$. Recalling \rif{defiH} and \rif{Hnot}, and that $t \mapsto (t-k)_{+}$ is $1$-Lipschitz regular, we find
\begin{flalign*}\label{1}
\notag& \hspace{-11mm}\snr{(E(Du_{\rr})-\kk)_{+}-(E(Dv)-\kk)_{+}}^{2}   \leq 
\snr{E(Du_{\rr})-E(Dv)}^{2}  
\\ &\ \, \qquad \leq  \snr{[H(Du_{\rr})]^{p/2}-[H(Dv)]^{p/2}}^{2}  \nonumber \\
&  \quad \ \ \, \stackrel{\rif{diffH}}{\lesssim_p}(\snr{Du_{\rr}}^{2}+\snr{Dv}^{2}+\mu^2)^{p/2}\snr{V(Du_{\rr})-V(Dv)}^{2}\,,\quad \forall\ x \in 8B_{h}\,.
\end{flalign*}
By finally using \rif{fuckbsc}, we conclude with 
\eqn{1}
$$
\snr{(E(Du_{\rr})-\kk)_{+}-(E(Dv)-\kk)_{+}}^{2}\leq cM^{\mathfrak{s}p}\snr{V(Du_{\rr})-V(Dv)}^{2}\,, \quad \forall\ x \in 4B_{h}
$$
for $ c \equiv  c (n,p,q,\nu, L)$. 
We then estimate using, in order, \rif{diffbasic}$_1$, \eqref{1} (twice) and \eqref{cacc}
\begin{eqnarray}
\notag &&\int_{B_{h}}\snr{\tau_{h}(E(Du_{\rr})-\kk)_{+}}^{2}  \dx \\
\notag  && \qquad  \leq 2 \int_{B_{h}}\snr{\tau_{h}(E(Dv)-\kk)_{+}}^{2}  \dx  
+2 \int_{B_{h}}\snr{\tau_{h}\left((E(Du_{\rr})-\kk)_{+}-(E(Dv)-\kk)_{+}\right)}^{2}  \dx
\\
\notag && \qquad  \leq 2 \int_{B_{h}}\snr{\tau_{h}(E(Dv)-\kk)_{+}}^{2}  \dx  
+2 \int_{2B_{h}}\snr{(E(Du_{\rr})-\kk)_{+}-(E(Dv)-\kk)_{+}}^{2}  \dx
\\
\notag&&\qquad \le 2\int_{B_{h}}\snr{\tau_{h}(E(Dv)-\kk)_{+}}^{2}  \dx\nonumber +cM^{\mathfrak {s}p}\int_{2B_{h}}\snr{V(Du_{\rr})-V(Dv)}^{2}  \dx\nonumber \\
\notag&&\qquad \leq c\snr{h}^{2(1-\beta_{0})}M^{\mathfrak{s}(q-p)}\int_{4B_{h}}(E(Dv)-\kk)_{+}^{2}  \dx+cM^{\mathfrak{s}p}\int_{2B_{h}}\snr{V(Du_{\rr})-V(Dv)}^{2}  \dx\nonumber \\
&&\qquad \leq c\snr{h}^{2(1-\beta_{0})}M^{\mathfrak{s}(q-p)}\int_{4B_{h}}(E(Du_{\rr})-\kk)_{+}^{2}  \dx+cM^{\mathfrak{s}q}\int_{4B_{h}}\snr{V(Du_{\rr})-V(Dv)}^{2}  \dx\,.\label{1bispre}
\end{eqnarray}
We have again used  \rif{includi}. 
Matching this last inequality with \eqref{37} we get 
\begin{eqnarray}\label{1bis}
\notag &&\int_{B_{h}}\snr{\tau_{h}(E(Du_{\rr})-\kk)_{+}}^{2}  \dx \le c\snr{h}^{2(1-\beta_{0})}M^{\mathfrak{s}(q-p)}\int_{8B_{h}}(E(Du_{\rr})-\kk)_{+}^{2}  \dx \nonumber\\ 
&& \hspace{25mm} +c_{\gggg}\snr{h}^{\beta_{0}\alpha}M^{\mathfrak{s}q+\alpha+\gamma q/p}\varrho^{\alpha}\snr{B_{h}}+c\snr{h}^{\frac{\beta_{0}\alpha p\mathfrak{a}(\ppp)}{p-\alpha}}M^{\ssf q}\|f_{\rr}\|_{L^{\frac{\ppp}{\ppp-\alpha}}(8B_{h})}^{\frac{\ppp\theta(\ppp)}{\ppp-\alpha}}\notag\\
&& \hspace{32mm} +c\una\snr{h}^{\frac{\beta_{0}\alpha2\mathfrak{a}(\ppp)}{2-\alpha}}M^{\mathfrak{s}q+\frac{\alpha(2-p)}{2-\alpha}}|h|^{\frac{\beta_0n}{p}\frac{\alpha(2-p)}{2-\alpha}}\|f_{\rr}\|_{L^{\frac{\ppp}{\ppp-\alpha}}(8B_{h})}^{\frac{\ppp\sigma(\ppp)}{\ppp-\alpha}}
\end{eqnarray}
with $c,c_{\gggg}\equiv c,c_{\gggg}(n,p,q,\nu, L)$. Recalling that $|h|\leq 1$, 
we first estimate
\eqn{asasas}
$$\snr{h}^{\beta_{0}\alpha}+\snr{h}^{\beta_{0}\alpha p\mathfrak{a}(\ppp)/(p-\alpha)}+\snr{h}^{\beta_{0}\alpha2\mathfrak{a}(\ppp)/(2-\alpha)} \leq 3\snr{h}^{\beta_{0}\alpham}$$ 
in \rif{1bis}, with $\alpham$ defined in \rif{alphagen}, and then equalize the resulting exponents by taking
\eqn{beta00} 
$$\beta_{0}=\frac{2}{2+\alpham} \Longleftrightarrow \beta_{0}\alpham= 2(1-\beta_{0})\,.$$ We conclude with 
\begin{flalign}\label{38}
\int_{B_{h}}\snr{\tau_{h}(E(Du_{\rr})-\kk)_{+}}^{2}  \dx&\le c\snr{h}^{\frac{2\alpham}{2+\alpham}}M^{\ssf(q-p)}\int_{8B_{h}}(E(Du_{\rr})-\kk)_{+}^{2}  \dx\nonumber \\
&\qquad  +c_{\gggg}\snr{h}^{\frac{2\alpham}{2+\alpham}}M^{\ssf q+ \alpha+\gamma q/p}\varrho^{\alpha}\snr{B_{h}}+c\snr{h}^{\frac{2\alpham}{2+\alpham}}M^{\ssf q}\|f_{\rr}\|_{L^{\frac{\ppp}{\ppp-\alpha}}(8B_{h})}^{\frac{\ppp\theta(\ppp)}{\ppp-\alpha}}\nonumber\\
&\qquad +c\una\snr{h}^{\frac{2\alpham}{2+\alpham}}M^{\ssf q+\frac{\alpha(2-p)}{2-\alpha}}|h|^{\frac{\beta_0n}{p}\frac{\alpha(2-p)}{2-\alpha}}\|f_{\rr}\|_{L^{\frac{\ppp}{\ppp-\alpha}}(8B_{h})}^{\frac{\ppp\sigma(\ppp)}{\ppp-\alpha}}\,. 
\end{flalign}
\subsubsection{Covering}\label{hyb} We consider a fixed, standard lattice $\mathcal L_{|h|}$ of hypercubes of $\er^n$, with sidelength equal to $2 |h|^{\beta_0}/\sqrt{n}$ (these are mutually disjoint open hypercubes with sides parallel to the coordinate axes, and whose union of closures covers $\er^n$). We consider $\tilde{ \mathcal L}_{|h|} := \{Q \in  \mathcal L_{|h|}\, \colon\,  Q\cap \BBB \not = \emptyset\}$. With $\mathfrak{n}$ denoting the cardinality of $\tilde{ \mathcal L}_{|h|}$, we write $\tilde{ \mathcal L}_{|h|}= \{Q_{k}\equiv Q_k(x_k)\}_{k \leq \mathfrak{n}}$, with $x_k$ being the center of $Q_k$. Obviously, it is 
\begin{flalign}\label{11.1}
\Big| \ \BBB\setminus \bigcup_{k\le \mathfrak{n}}Q_{k} \ \Big|=0,\qquad Q_{i}\cap Q_{j}=\emptyset \ \Leftrightarrow \ i\not =j\,.
\end{flalign}
This family of hypercubes corresponds to a family of balls $\mathcal L_{|h|}:=\{B_{k}\equiv B_k(x_k)\}_{k \leq \mathfrak{n}}$ in the sense of \rif{nestate}, that is, such cubes can be identified as inner cubes of balls, i.e., 
\eqn{innnout}
$$Q_{k}\equiv Q_{k}(x_k)\equiv  Q_{\textnormal{inn}}(B_{k}) \subset B_{k}:= B_{\snr{h}^{\beta_{0}}}(x_{k})\,.$$ 
Note that the centers $x_k \equiv x_{\rm c}$ satisfy \rif{sceltacentro} as the diameter of such cubes is equal to $2|h|^{\beta_0}$. Summarizing, we find
\begin{flalign}\label{safe}
\begin{cases}
8B_{k}\Subset \BB\,, \ x_{ k}\in \mathcal B_{1/2+ 2|h|^{\beta_0}} \ \ \mbox{for all} \ \ k\le \mathfrak{n}\\
\mathfrak{n}\lesssim n^{n/2}2^{-n}\snr{h}^{-\beta_{0}n} \equiv c(n)\snr{h}^{-\beta_{0}n} \,.
\end{cases}
\end{flalign}
Each of the dilated balls $8B_{k}$ intersects the similar ones $8B_{i}$ (including itself) less than a finite number $c_{\rm t}(n)$, depending only on $n$ (uniform finite intersection property). This can be easily seen by observing that the outer cubes of the dilated balls $Q_{\textnormal{out}}(8B_{k})$ in the sense of \rif{nestate}, whose sidelength is $16|h|^{\beta_0}$, touch similar ones $c_{\rm t}(n)$ times, as they are obtained by dilating of a factor $8\sqrt{n}$ the original ones $Q_k \equiv Q_{\textnormal{inn}}(B_{k})$, which are mutually disjoint and have sides parallel to the coordinate axes. By considering the family of enlarged balls $8\mathcal L_{|h|}:=\{8B_{k}\, \colon\, B_k \in \mathcal L_{|h|} \}$, we can therefore write 
$$
\begin{cases}
\, \displaystyle 8\mathcal L_{|h|} = \bigcup_{i\leq \tilde c_{\rm t}(n)} 8\mathcal L_{|h|}^i\,,  \ \   \mbox{with $\tilde c_{\rm t}(n) \leq c_{\rm t}(n)$ and $ 8\mathcal L_{|h|}^i\cap  8\mathcal L_{|h|}^j\not= \emptyset \Longrightarrow i=j$}\\
\, \mbox{$ 8\mathcal L_{|h|}^i$ is made of mutually disjoint balls, for every $i \leq c_{\rm t}(n)$\,.}
\end{cases}
$$
As a consequence, given a Radon measure $\lambda$ defined on $\BB$, recalling \rif{safe}, we find
\eqn{finite}
$$
\sum_{k\leq  \mathfrak n}\lambda(8B_k) = \sum_{i=1}^{\tilde c_{\rm t}(n)} \sum_{8B \in \mathcal 8 \mathcal L_{|h|}^i}
\lambda(8B) \leq  \sum_{i=1}^{\tilde c_{\rm t}(n)} 
\lambda(\BB)  = \tilde c_{\rm t}(n)\lambda(\BB)\,.
$$
Another inequality we will often use is, for $t \in (0,1]$ 
\eqn{discrete}
$$
\sum_{k\leq \mathfrak{n}  }a_k^t \leq \mathfrak{n}^{1-t} \big(\sum_{k\leq \mathfrak{n}  }a_k\big)^t\stackrel{\rif{safe}}{\leq}c(n)\snr{h}^{-\beta_{0}n(1-t)}\big(\sum_{k\leq \mathfrak{n}  }a_k\big)^t\,,
$$
that holds whenever $\{a_k\}_{k\leq \mathfrak{n}  }$ are non-negative numbers. 
This is a simple consequence of the discrete H\"older's inequality. 

\subsubsection{Proof of \eqref{40}}\label{hyb2} By \rif{safe} we can now consider the minimization problems in \rif{pd} for each one of the enlarged balls $8B_k$, i.e., we take $B_{h}\equiv B_{k}$, thereby getting \rif{38}. Inequalities \eqref{38} can be summed over $k\le \mathfrak{n}$, and this yields
\begin{eqnarray}
&&
\notag \hspace{-5mm}\int_{\BBB}\snr{\tau_{h}(E(Du_{\rr})-\kk)_{+}}^{2}  \dx \stackrel{\eqref{11.1}}{\le}\sum_{k\le \mathfrak{n}}\int_{Q_{k}}\snr{\tau_{h}(E(Du_{\rr})-\kk)_{+})}^{2}  \dx\nonumber \\ && \stackleq{innnout} \sum_{k\le \mathfrak{n}}\int_{B_{k}}\snr{\tau_{h}(E(Du_{\rr})-\kk)_{+})}^{2}  \dx\nonumber \\
&&\notag    \stackrel{\eqref{38}}{\le}c\snr{h}^{\frac{2\alpham}{2+\alpham}}M^{\ssf (q-p)}\sum_{k\le \mathfrak{n}}\int_{8B_{k}}(E(Du_{\rr})-\kk)_{+}^{2}  \dx \nonumber +c_{\gggg}\snr{h}^{\frac{2\alpham}{2+\alpham}}M^{\ssf q+ \alpha+\gamma q/p}\varrho^{\alpha}\sum_{k\leq  \mathfrak n}\snr{B_{k}}\nonumber \\
&&   \qquad  \quad   +c\snr{h}^{\frac{2\alpham}{2+\alpham}}M^{\ssf q}\sum_{k\leq  \mathfrak n}\|f_{\rr}\|_{L^{\frac{\ppp}{\ppp-\alpha}}(8B_{k})}^{\frac{\ppp\theta(\ppp)}{\ppp-\alpha}}\notag \\&&
\qquad \quad   +c\una\snr{h}^{\frac{2\alpham}{2+\alpham}}M^{\ssf q+\frac{\alpha(2-p)}{2-\alpha}}|h|^{\frac{\beta_0n}{p}\frac{\alpha(2-p)}{2-\alpha}}\sum_{k\leq  \mathfrak n}\|f_{\rr}\|_{L^{\frac{\ppp}{\ppp-\alpha}}(8B_{k})}^{\frac{\ppp\sigma(\ppp)}{\ppp-\alpha}}.\label{ricopri}
\end{eqnarray}
To estimate the first two sums appearing in \rif{ricopri} we can use \rif{finite} in an obvious way
\eqn{proas0}
$$
\sum_{k\le \mathfrak{n}}\int_{8B_{k}}(E(Du_{\rr})-\kk)_{+}^{2}  \dx \leq \tilde c_{\rm t}(n)
\int_{\BB}(E(Du_{\rr})-\kk)_{+}^{2}  \dx \ \ \mbox{and} \ \  \sum_{k\leq  \mathfrak n}\snr{B_{k}}  \leq \tilde c_{\rm t}(n)|\BB|\,.
$$
For the third we proceed similarly, using that $\theta(\ppp)\geq 1$
\eqn{proas}
$$\sum_{k\leq  \mathfrak n}\|f_{\rr}\|_{L^{\frac{\ppp}{\ppp-\alpha}}(8B_{k})}^{\frac{\ppp\theta(\ppp)}{\ppp-\alpha}} \leq 
\|f_{\rr}\|_{L^{\frac{\ppp}{\ppp-\alpha}}(\BB)}^{\frac{\ppp[\theta(\ppp)-1]}{\ppp-\alpha}}\sum_{k\leq  \mathfrak n}\int_{8B_{k}}\snr{f_{\rr}}^{\frac{\ppp}{\ppp-\alpha}}  \dx \stackleq{finite} \tilde c_{\rm t}(n) \|f_{\rr}\|_{L^{\frac{\ppp}{\ppp-\alpha}}(\BB)}^{\frac{\ppp\theta(\ppp)}{\ppp-\alpha}}\,.
$$
As for the last sum in \rif{ricopri}, this appears only when $p<2$ and here we distinguish two cases. The first is when $\sigma (\ppp)\geq 1$. 
In this case we again argue as in \rif{proas}, using \rif{finite}, and we get (recall $|h|\leq 1$ by \rif{hhh})
\eqn{proas2}
$$
|h|^{\frac{\beta_0n}{p}\frac{\alpha(2-p)}{2-\alpha}}\sum_{k\leq  \mathfrak n}
\|f_{\rr}\|_{L^{\frac{\ppp}{\ppp-\alpha}}(8B_{k})}^{\frac{\ppp\sigma(\ppp)}{\ppp-\alpha}}\leq  \tilde c_{\rm t}(n) \|f_{\rr}\|_{L^{\frac{\ppp}{\ppp-\alpha}}(\BB)}^{\frac{\ppp\sigma(\ppp)}{\ppp-\alpha}}\,. 
$$
The second case is when, instead, $\sigma(\ppp)<1$ and we are led to use \rif{discrete}, with $t\equiv \sigma(\ppp)$ and $a_k \equiv 
\|f_{\rr}\|_{L^{\ppp/(\ppp-\alpha)}(8B_{k})}^{\ppp/(\ppp-\alpha)}$, thereby getting
\begin{flalign}
\notag |h|^{\frac{\beta_0n}{p}\frac{\alpha(2-p)}{2-\alpha}}\sum_{k\leq  \mathfrak n}\|f_{\rr}\|_{L^{\frac{\ppp}{\ppp-\alpha}}(8B_{k})}^{\frac{\ppp\sigma(\ppp)}{\ppp-\alpha}} & \leq c |h|^{\frac{\beta_0n}{p}\frac{\alpha(2-p)}{2-\alpha}-\beta_0n(1-\sigma(\ppp))}\|f_{\rr}\|_{L^{\frac{\ppp}{\ppp-\alpha}}(\BB)}^{\frac{\ppp\sigma(\ppp)}{\ppp-\alpha}}\\
& \leq c\|f_{\rr}\|_{L^{\frac{\ppp}{\ppp-\alpha}}(\BB)}^{\frac{\ppp\sigma(\ppp)}{\ppp-\alpha}}\label{proas3}\,.
\end{flalign}
We have used that $|h|\leq 1$ and the identity
$$
\frac{\beta_0n}{p}\frac{\alpha(2-p)}{2-\alpha}-\beta_0n(1-\sigma(\ppp)) = \frac{2\beta_0n\alpha}{p\ppp(2-\alpha)}(\ppp-p)\stackrel{\ppp\geq p}{\geq} 0\,.
$$
Connecting the content \rif{proas0}-\rif{proas3} to \rif{ricopri}, we finally arrive at
\begin{align*}
 &  \int_{\BBB}\snr{\tau_{h}(E(Du_{\rr})-\kk)_{+}}^{2}  \dx  \\  & \quad \leq c\snr{h}^{\frac{2\alpham}{2+\alpham}}M^{\ssf (q-p)}\int_{\BB}(E(Du_{\rr})-\kk)_{+}^{2}  \dx+c_{\gggg}\snr{h}^{\frac{2\alpham}{2+\alpham}}M^{\ssf q+ \alpha+\gamma q/p}\varrho^{\alpha}\\
 & \qquad \quad    +c\snr{h}^{\frac{2\alpham}{2+\alpham}}M^{\ssf q}\|f_{\rr}\|_{L^{\frac{\ppp}{\ppp-\alpha}}(\BB)}^{\frac{\ppp\theta(\ppp)}{\ppp-\alpha}}+c\una\snr{h}^{\frac{2\alpham}{2+\alpham}}M^{\ssf q+\frac{\alpha(2-p)}{2-\alpha}}\|f_{\rr}\|_{L^{\frac{\ppp}{\ppp-\alpha}}(\BB)}^{\frac{\ppp\sigma(\ppp)}{\ppp-\alpha}}\nonumber
\end{align*}
for $c\equiv c(n,p,q,\nu, L)$. From this last inequality and Lemma \ref{l4}, we deduce that $(E(Du_{\rr})-\kk)_{+}\in W^{\beta,2}(\BBB)$ for all $\beta\in (0,\alpham/(2+\alpham))$ and that the inequality 
\begin{flalign*}
\|(E(Du_{\varrho})-\kk)_{+}\|_{W^{\beta,2}(\BBB)} &\le cM^{\frac{\ssf(q-p)}{2}}\nr{(E(Du_{\varrho})-\kk)_{+}}_{L^{2}(\BB)}+c_{\gggg}M^{\frac{\ssf q+ \alpha+\gamma q/p}{2}}\varrho^{\frac{\alpha}{2}}\\
& \qquad +c M^{\frac {\ssf q}2}\|f_{\rr}\|_{L^{\frac{\ppp}{\ppp-\alpha}}(\BB)}^{\frac{\ppp\theta(\ppp)}{2(\ppp-\alpha)}}+c\una M^{\frac{\ssf q}{2}+\frac{\alpha(2-p)}{2(2-\alpha)}}\|f_{\rr}\|_{L^{\frac{\ppp}{\ppp-\alpha}}(8B_{h})}^{\frac{\ppp\sigma(\ppp)}{2(\ppp-\alpha)}}
 \end{flalign*}
holds with $c\equiv c(\data, \beta)$. From this the full form of \rif{40-res} follows via \rif{immersione}. This brings to a conclusion the proof of Proposition \ref{caccin1}. 
\subsubsection{Proof of \eqref{40prima}}\label{hybprima}
The only modification occurs in the estimate of the term (III) in \rif{stimaIII}, that we can now replace as follows:
\begin{eqnarray}
\notag \mbox{(III)} &\le & c\int_{8B_{h}}f_{\rr}(x)\snr{u_{\rr}-v}^{\alpha} \dx \leq \|u_{\rr}-v\|_{L^{\infty}(8B_{h})}^{\alpha}\|f_{\rr}\|_{L^{1}(8B_{h})}\\ 
&\stackleq{conosc2} & c |h|^{\beta_0\alpha } M^{\alpha} \|f_{\rr}\|_{L^{1}(8B_{h})}\,. \label{civuole}
\end{eqnarray}
Using this estimate in place of \rif{stimaIII}, we can replace \rif{37} by
\eqn{37bbb}
$$
 \int_{8B_{h}}\snr{V(Du_{\rr})-V(Dv)}^{2}  \dx  \le c_{\gggg}\snr{h}^{\beta_{0}\alpha}M^{\alpha+\gamma q/p}\varrho^{\alpha}\snr{B_{h}}
+c |h|^{\beta_0\alpha } M^{\alpha} \|f_{\rr}\|_{L^{1}(8B_{h})}
$$
with $c,c_{\gggg}\equiv c,c_{\gggg}(n,p,q,\nu, L)$. Proceeding as for the proof of \rif{40}, choosing $\beta_0$ as in \rif{beta00} with $\alpham=\alpha$, we arrive at \rif{40prima} and Proposition \ref{caccin1} is proved.  
\subsection{Functionals of the type in \trif{modellou}}\label{variante3} 
We consider functionals $\mathcal{S}(\cdot, B_{r})$ as in \rif{modellou}, 
assuming 
\eqn{assu1}
$$
 F(\cdot) \   \mbox{satisfies \rif{32.2}, $\ccc(\cdot)$ is as in \trif{modellou}$_2$, and $\hhh(\cdot)$ satisfies \rif{asshh22}\,.}
$$
\begin{proposition}\label{caccin3}
Let $u\in W^{1,q}(B_{r})$ be a minimizer of the functional $\mathcal S(\cdot, B_{r})$ in \trif{modellou}, under assumptions \eqref{marcbound} and \eqref{assu1}. Let $B_{\rr}(x_{0}) \Subset B_{r}$ and let $M\geq 1$ be a constant such that $ \nr{Du}_{L^{\infty}(B_{\rr}(x_{0}))}\leq M$. Then, for every number $\kk \geq 0$
\begin{eqnarray}\label{xx.10xu}
&&\notag  \left(\mint_{B_{\rr/2}(x_{0})}(E_{\mu}(Du)-\kk)_{+}^{2\chi}  \dx\right)^{1/\chi} \le cM^{\ssf(q-p)} \mint_{B_{\rr}(x_{0})}(E_{\mu}(Du)-\kk)_{+}^{2} \dx\\
&&\qquad\qquad \qquad  \qquad+cM^{\ssf q+p+\alpha-\mathrm{b}}\rr^{\alpha}\mint_{B_{\rr}(x_{0})}(\snr{Du}+1)^{q-p+\mathrm{b}} \dx+cM^{\ssf q+\alpha}\rr^{\alpha} \mint_{B_{\rr}(x_{0})}f \dx
\end{eqnarray}
holds for every $\mathrm{b}  \in [0, p]$, and $(\beta, \chi)$ as in \trif{bechi} with $\alpham:=\alpha$, where $c\equiv c(\data, \beta)$. 
\end{proposition}
\begin{proof} We keep the notation introduced in Proposition \ref{caccin1}. First we rescale $u$ as in \rif{riscalata}, thereby passing to $u_{\varrho} \in W^{1,q}(\BB)$. This is a minimizer of the functional
\eqn{definito2}
$$
w\mapsto
 \mathcal{S}_{\varrho}(w, \BB) :=
 \int_{\BB}[\ccc_{\rr}(x,w)F(Dw)+\hhh_{\rr}(x,w)] \dx,
$$
which is defined on $W^{1,q}(\BB)$, where $\ccc_{\rr}(x,y):=\ccc(x_0+\rr x,\varrho y)$ and $\hhh_{\rr}(x,y):=\hhh(x_0+\rr x,\rr y)$. $F_{\rr}(\cdot)$ satisfies \rif{assu1} and, as a consequence of \rif{assu1} (and therefore of \trif{modellou}$_2$)
\eqn{riririri}
$$
\snr{\ccc_{\rr}(x_{1},y_{1})-\ccc_{\varrho}(x_{2},y_{2})}\le L\varrho^{\alpha} \left(\snr{x_{1}-x_{2}}^{\alpha }+\snr{y_{1}-y_{2}}^{\alpha}\right), \quad \nu \leq \ccc_{\rr}(\cdot) \leq L
$$
holds for all $x_1, x_2 \in \BB$, $y_1, y_2\in \er$, and $\hhh_{\rr}(\cdot)$ satisfies \rif{assgr}. It is sufficient to prove that
\begin{flalign}
&\notag \nr{(E(Du_{\rr})-\kk)_{+}}_{L^{2\chi}(\BBB)}^2  \le cM^{\ssf(q-p)}\nr{(E(Du_{\rr})-\kk)_{+}}_{L^{2}(\BB)}^2\\ &
\qquad \qquad \qquad \qquad +cM^{\ssf q+p+\alpha-\mathrm{b}}
\varrho^{\alpha}\int_{\BB}(\snr{Du_{\rr}}+1)^{q-p+\mathrm{b}}  \dx+cM^{\ssf q+\alpha}\|f_{\rr}\|_{L^1(\BB)}\label{dimodimo}
\end{flalign}
holds for $c \equiv c(\data, \beta)$, so that \rif{xx.10xu} follows scaling back to $u$. 
We next consider $v\in u_{\rr}+W^{1,q}_0(8B_{h})$ as the solution to
\rif{pd}, 
where this time it is $F_{0}(z):=\ccc_{\varrho}(x_{\rm c},(u_{\rr})_{8B_{h}})F(z)$, and $x_{\rm c}$ is the center of $B_{h}$. Note that, thanks to \rif{assu1}, inequalities \rif{conosc}-\rif{fuckbsc} apply to $v$. 
Recalling \rif{definito2}, we have, similarly to \rif{mmaa}
\begin{eqnarray}
\notag&&\hspace{-3mm} \frac 1{\tilde c}\int_{8B_{h}}\snr{V(Du_{\rr})-V(Dv)}^{2}  \dx \leq  \int_{8B_{h}}\left[ F_{0}(Du_{\rr})- F_{0}(Dv)\right]  \dx\\
\notag&&  \ \,=  \mathcal{S}_{\varrho}(u_{\rr}, \BB) - \mathcal{S}_{\varrho}(v, \BB) 
+\int_{8B_{h}}[\ccc_{\rr}(x_{\rm c},(u_{\rr})_{8B_{h}})F(Du_{\rr})-\ccc_{\rr}(x,u_{\rr})F(Du_{\rr})]  \dx\\ 
\notag&& \qquad   +\int_{8B_{h}}[\ccc_{\rr}(x,v)F(Dv)-\ccc_{\rr}(x_{\rm c},(u_{\rr})_{8B_{h}})F(Dv)]\dx +\int_{8B_{h}}[\hhh_{\rr}(x,v)-\hhh_{\rr}(x,u_{\rr})]\dx\\
&& \ \ =: \mathcal{S}_{\varrho}(u_{\rr}, 8B_{h}) - \mathcal{S}_{\varrho}(v, 8B_{h}) +\mbox{(I)}+\mbox{(II)}+\mbox{(III)}\leq \mbox{(I)}+\mbox{(II)}+\mbox{(III)}\,.\label{nonved}
\end{eqnarray}
In the last line we have used that $u_{\rr}$ minimizes the functional in \rif{definito2}. 
Using \rif{riririri} yields
\begin{flalign}
\notag \mbox{(I)}& \leq c\rr^{\alpha}\int_{8B_{h}}\left(\snr{x-x_{\rm c}}^{\alpha}+\snr{u_{\rr}-(u_{\rr})_{8B_{h}}}^{\alpha}\right)(\snr{Du_{\rr}}+1)^{q}  \dx\notag\\
\notag & \leq c\snr{h}^{\beta_{0}\alpha}(M^\alpha+1) \rr^{\alpha}\int_{8B_{h}}(|Du_{\varrho}|+1)^{q}\dx   \\
& \leq
c\snr{h}^{\beta_{0}\alpha}M^{p+\alpha-\mathrm{b}} \rr^{\alpha}\int_{8B_{h}}(|Du_{\varrho}|+1)^{q-p+\mathrm{b}}\,  dx  \,.\label{nonved2}
\end{flalign}
For $\mbox{(II)}$, note that minimality of $v$ and $\nu \leq \ccc_{\rr}(\cdot) \leq L$ give
\begin{flalign}
\notag \int_{8B_{h}}F(Dv)  \dx &\leq \frac{1}{\nu}\int_{8B_{h}}\ccc_{\rr}(x_{\rm c},(u_{\rr})_{8B_{h}})F(Dv)  \dx \\
&\leq \frac{1}{\nu}\int_{8B_{h}}\ccc_{\rr}(x_{\rm c},(u_{\rr})_{8B_{h}})F(Du_{\rr})  \dx \leq c(\nu, L)  \int_{8B_{h}}(\snr{Du_{\rr}}+1)^q \dx\,.\label{qui}
\end{flalign}
Using the content of the last display, \rif{riririri} and then \rif{conosc2}, we get
\begin{flalign}
\notag \mbox{(II)} & \leq c\rr^{\alpha}\int_{8B_{h}}\left(\snr{x-x_{\rm c}}^{\alpha}+\snr{v-(v)_{8B_{h}}}^{\alpha}+\snr{(v)_{8B_{h}}-(u_{\rr})_{8B_{h}}}^{\alpha}\right)F(Dv)  \dx\\
\notag & \le c \snr{h}^{\beta_{0}\alpha}(M^{\alpha}+1)
\rr^{\alpha}\int_{8B_{h}}F(Dv)   \dx \le c  \snr{h}^{\beta_{0}\alpha} M^{\alpha}
\rr^{\alpha}\int_{8B_{h}}(\snr{Du_{\rr}}+1)^q   \dx\\
& \leq c\snr{h}^{\beta_{0}\alpha}M^{p+\alpha -{\rm b}}\rr^{\alpha}\int_{8B_{h}}(\snr{Du_{\rr}}+1)^{q-p+{\rm b}}  \dx \,.\label{nonved222}
\end{flalign}
The term $\mbox{(III)}$ can be estimated exactly as in \rif{civuole}. Using this together with \rif{nonved}-\rif{nonved222} yields 
\begin{flalign}\label{xx.6}
\notag\int_{8B_{h}}\snr{V(Du_{\rr})-V(Dv)}^{2}\dx &\le c\snr{h}^{\beta_{0}\alpha}M^{p+\alpha-\mathrm{b}}\rr^{\alpha}
\int_{8B_{h}}(\snr{Du_{\rr}}+1)^{q-p+\mathrm{b}} \dx\\
&\quad + c |h|^{\beta_0\alpha } M^{\alpha} \|f_{\rr}\|_{L^{1}(8B_{h})}
\,.
\end{flalign}
Using \rif{1}-\rif{1bispre}, employing \rif{xx.6} in \rif{1bispre}, and choosing $\beta_0:=2/(2+\alpha)$, we arrive at
\begin{flalign*}
&\int_{B_h}\snr{\tau_{h}(E(Du_{\rr})-\kk)_{+}}^{2} \dx
\leq  c\snr{h}^{\frac{2\alpha}{2+\alpha}}M^{\ssf(q-p)}\int_{8B_{h}}(E(Du_{\rr})-\kk)_{+}^{2} \dx\nonumber\\
&\qquad\qquad  \qquad +c\snr{h}^{\frac{2\alpha}{2+\alpha}}M^{\ssf q+p+\alpha-\mathrm{b}}\rr^{\alpha}\int_{8B_{h}}(\snr{Du_{\rr}}+1)^{q-p+\mathrm{b}} \dx+c\snr{h}^{\frac{2\alpha}{2+\alpha}}M^{\ssf q+\alpha} \|f_{\rr}\|_{L^{1}(8B_{h})}
\end{flalign*}
for $c\equiv c(\data)$.
This estimate can be used as a replacement of \rif{38} and proceeding as in the proofs of Proposition \ref{caccin1} we conclude with 
\begin{flalign*}
 & \int_{\BBB}\snr{\tau_{h}(E(Du_{\rr})-\kk)_{+}}^{2}  \dx \leq c\snr{h}^{\frac{2\alpha}{2+\alpha}}M^{\ssf(q-p)}\int_{\BB}(E(Du_{\rr})-\kk)_{+}^{2}  \dx\\
 & \qquad \qquad\qquad+c\snr{h}^{\frac{2\alpha}{2+\alpha}}M^{\ssf q+p+\alpha-\mathrm{b}}\rr^{\alpha}\int_{\BB}(\snr{Du_{\rr}}+1)^{q-p+\mathrm{b}} \dx+c\snr{h}^{\frac{2\alpha}{2+\alpha}}M^{\ssf q+\alpha} \|f_{\rr}\|_{L^{1}(\BB)}\nonumber\,.
\end{flalign*}
As for Proposition \ref{caccin1}, this last inequality, Lemma \ref{l4} and \rif{immersione}, we deduce that $(E(Du_{\rr})-\kk)_{+}\in W^{\beta,2}(\BBB)$ for all $\beta < \alpha / (2+\alpha)$, with \rif{dimodimo} that follows accordingly.
\end{proof}

\subsection{Functionals of the type in \rif{FMx}}\label{variante2f} The assumptions on the integrand $F\colon B_{r} \times \er^n\to [0, \infty)$ we consider here are
\eqn{xx.3}
$$
\begin{cases}
\, z \mapsto F(x, z) \ \mbox{satisfies \rif{32.2} uniformly with respect to $x \in B_{r}$}\\
\, \snr{\partial_{z}F(x_{1},z)-\partial_{z}F(x_{2},z)}\le \tilde L\snr{x_{1}-x_{2}}^{\alpha}([H_{\mu}(z)]^{(q-1)/2}+[H_{\mu}(z)]^{(p-1)/2})
\end{cases}
$$
whenever $x_1, x_2 \in B_{r}$, $z\in \er^n$, and where $0< \mu \leq 2$. In the following, we denote $A(\cdot):= \partial_z F(\cdot)$, so that the Euler-Lagrange equation of $\mathcal{F}_{\texttt{x}}$ reads as $\diver\,  A(x, Du)=0$ and any $W^{1,q}$-regular minimizer is an energy solution.  Moreover, this time $(\beta, \chi)$ might also be such that 
\eqn{bechi2}
$$
\beta< \frac{\alpha}{1+\alpha} \qquad \mbox{and} \qquad \chi:= \frac{n}{n-2\beta}\,.
$$
\begin{proposition}\label{caccin4}
Let $u\in W^{1,q}(B_{r})$ be a minimizer of the functional $\mathcal F_{\textnormal{\texttt{x}}}(\cdot, B_{r})$ in \trif{FMx}, under assumptions \eqref{marcbound} and \eqref{xx.3}. Let $B_{\rr}(x_{0}) \Subset B_{r}$ and let $M\geq 1$ be a constant such that $ \nr{Du}_{L^{\infty}(B_{\rr}(x_{0}))}\leq M$. Let $\kk \geq 0$ be a number. 
\begin{itemize}
\item If  $p\geq 2$, then
\begin{flalign}\label{xx.10}
\notag \left(\mint_{B_{\rr/2}(x_{0})}(E_{\mu}(Du)-\kk)_{+}^{2\chi}  \dx\right)^{1/\chi} & \le cM^{\ssf (q-p)} \mint_{B_{\rr}(x_{0})}(E_{\mu}(Du)-\kk)_{+}^{2} \dx \\ & \qquad  +cM^{\ssf q+p-\mathrm{b}}\rr^{2\alpha}\mint_{B_{\rr}(x_{0})}(\snr{Du}+1)^{2q-2p+\mathrm{b}} \dx
\end{flalign}
holds for every $\mathrm{b} \in [0, p]$ and $(\beta, \chi)$ as in \trif{bechi2}, where $c \equiv c(\data, \beta)$. 
\item If $1<p<2$, then 
\begin{flalign}\label{xx.10-sotto2}
\notag  \left(\mint_{B_{\rr/2}(x_{0})}(E_{\mu}(Du)-\kk)_{+}^{2\chi}  \dx\right)^{1/\chi}  &\le cM^{\ssf(q-p)} \mint_{B_{\rr}(x_{0})}(E_{\mu}(Du)-\kk)_{+}^{2} \dx \\
&\quad  +cM^{(\ssf+1)q-{\rm b}/p}\rr^{\alpha}\left(\mint_{B_{\rr}(x_{0})} (\snr{Du}+1)^{q-p+{\rm b}}\dx\right)^{1/p}
\end{flalign}
holds for every $\mathrm{b} \in [0, p]$, and $(\beta, \chi)$ as in \trif{bechi} with $\alpham:=\alpha$, where $c \equiv c(\data, \beta)$. 
\end{itemize}
\end{proposition}
\begin{proof}[Proof of Proposition \ref{caccin4}] We again build on the general arguments settled down in Proposition \ref{caccin1}. 
The rescaled function $u_{\rr}$ defined in \rif{riscalata} minimizes
$
w\mapsto
\int_{\BB} F_{\rr}(x,Dw) \dx,
$ where $F_{\rr}(x,z)  := F(x_{\rm c}+\rr x,z) $, and therefore solves 
\eqn{risolve}
$$-\diver \, A_{\rr}(x, Du_{\rr})=0 \  \  \mbox{ in $\BB$,  where} \ \   
A_{\rr}(x ,z):=A(x_0+\rr x, z)\,.$$ 
Note that 
\eqn{xx.4}
$$
\snr{A_{\rr}(x_{1},z)-A_{\varrho}(x_{2} , z)}\le  \rr^{\alpha}\tilde L \snr{x_{1}-x_{2}}^{\alpha }([H_{\mu}(z)]^{(q-1)/2}+[H_{\mu}(z)]^{(p-1)/2})
$$
holds for every choice of $x_1, x_2 \in \BB$ and $z \in \er^n$, as a consequence of \rif{xx.3}$_2$. We define $A_{0}(z):=A_{\varrho}(x_{\rm c},z)$, where $B_{h}$ is centred at $x_{\rm c}$, which 
is strictly $p$-monotone, in the sense that
\eqn{hamine2}
$$
\snr{V(z_2)-V(z_1)}^{2} \leq  \tilde c \left(A_{0}(z_2)- A_{0}(z_1)\right)\cdot(z_2-z_1)
$$
holds whenever $z_1, z_2 \in \er^n$, where $\tilde c \equiv \tilde c(n,p, \tilde \nu)$. See for instance \cite{ha, KM} and related references. We then define $v$ as the unique solution to the Dirichlet problem
\eqn{eulerieq}
$$
\begin{cases}
 \displaystyle 
\int_{8B_{h}} A_{0}(Dv)\cdot D\varphi  \dx=0 \quad \mbox{for every} \  \varphi \in W^{1,q}_0(8B_{h})
\\
\, v \in u_{\rr}+W^{1,q}_0(8B_{h})\,,
\end{cases}
$$
that coincides with
 \rif{euleri}, where $F_0(z)=F(x_{\rm c}, z)$, so that we can use \rif{conosc}-\rif{fuckbsc}. We have
\begin{eqnarray}
&&\notag \int_{8B_{h}}\snr{V(Du_{\rr})-V(Dv)}^{2}\dx\stackrel{\eqref{hamine2}}{\le} c\int_{8B_{h}}\left(A_{0}(Du_{\rr})-A_{0}(Dv)\right) \cdot \left(Du_{\rr}-Dv\right)\dx\nonumber \\
&&\notag  \stackrel{\eqref{eulerieq}}{=}c\int_{8B_{h}}\left(A_{0}(Du_{\rr})-A_{\rr}(x, Du_{\rr})\right)\cdot \left(Du_{\rr}-Dv\right)\dx\nonumber\\
&&\stackrel{\eqref{xx.4}}{\le}c\rr^{\alpha}\snr{h}^{\beta_{0}\alpha}\int_{8B_{h}}\left([H(Du_{\rr})]^{(q-1)/2}+[H(Du_{\rr})]^{(p-1)/2}\right)\snr{Du_{\rr}-Dv} \dx =:c\mbox{(I)}  \,,\label{torna}
\end{eqnarray}
where $c\equiv c(\data)$. We focus on \rif{xx.10} and therefore on the case $p\geq 2$. Young's inequality and \rif{ineV} give
\begin{flalign*}
\notag c\mbox{(I)} &  \le\frac{1}{2}\int_{8B_{h}}\snr{V(Du_{\rr})-V(Dv)}^{2}\dx\\
& \qquad    +c\rr^{2\alpha}\snr{h}^{2\beta_{0}\alpha}\int_{8B_{h}}
\left([H(Du_{\rr})]^{q-1}+[H(Du_{\rr})]^{p-1}\right)(\snr{Du_{\rr}}^{2}+\snr{Dv}^{2}+\mu^{2})^{(2-p)/2}\dx\nonumber \\
&  \ \, \leq\frac{1}{2}\int_{8B_{h}}\snr{V(Du_{\rr})-V(Dv)}^{2}\dx +c\snr{h}^{2\beta_{0}\alpha}\rr^{2\alpha}\int_{8B_{h}}\left([H(Du_{\rr})]^{(2q-p)/2}+[H(Du_{\rr})]^{p/2}\right)\dx\nonumber \\
&\ \, \leq\frac{1}{2}\int_{8B_{h}}\snr{V(Du_{\rr})-V(Dv)}^{2}\dx+c\snr{h}^{2\beta_{0}\alpha}M^{p-\mathrm{b}}\rr^{2\alpha}\int_{8B_{h}}(\snr{Du_{\rr}}+1)^{2q-2p+\mathrm{b}} \dx\,.
\end{flalign*}
Connecting this last inequality to \rif{torna} yields
\eqn{nonabbiamo}
$$
 \int_{8B_{h}}\snr{V(Du_{\rr})-V(Dv)}^{2}\dx  \leq c\snr{h}^{2\beta_{0}\alpha}M^{p-\mathrm{b}}\rr^{2\alpha}\int_{8B_{h}}(\snr{Du_{\rr}}+1)^{2q-2p+\mathrm{b}} \dx\,.$$
Inserting this last estimate in \rif{1bispre} we can repeat the scheme of proof of Proposition \ref{caccin1}, and this eventually leads to \rif{xx.10}. The only difference is that this time, in order to equalize the exponents of $|h|$, we choose $\beta_0:=1/(1+\alpha)$ instead of making the choice  in \rif{beta00}. We now consider the range $1<p<2$, to which we specialize for the rest of the proof. We estimate the term $c\mbox{(I)} $ appearing in \rif{torna} in a different way, also using the minimality of $v$, as follows:
\begin{flalign}
 \notag c\mbox{(I)}  & \leq c\snr{h}^{\beta_{0}\alpha}(M^{q-1}+M^{p-1}+1)\rr^{\alpha}\int_{8B_{h}}\snr{Du_{\rr}-Dv} \dx\\
 \notag & \leq 
  c\snr{h}^{\beta_{0}\alpha}M^{q-1}|B_{h}|^{(p-1)/p}\rr^{\alpha}\left(\int_{8B_{h}} |Du_{\rr}-Dv|^p\dx\right)^{1/p}\\
  \notag & \leq 
  c\rr^{\alpha}\snr{h}^{\beta_{0}\alpha}M^{q-1}|B_{h}|^{(p-1)/p}\rr^{\alpha}\left(\int_{8B_{h}} [F_0(Du_{\rr})+F_0(Dv)]\dx\right)^{1/p}\\
  \notag   & \leq 
  c\snr{h}^{\beta_{0}\alpha}M^{q-1}|B_{h}|^{(p-1)/p}\rr^{\alpha}\left(\int_{8B_{h}} F_0(Du_{\rr})\dx\right)^{1/p}\\
    \notag   & \leq 
  c\snr{h}^{\beta_{0}\alpha}M^{q-1}|B_{h}|^{(p-1)/p}\rr^{\alpha}\left(\int_{8B_{h}} (\snr{Du_{\rr}}+1)^q\dx\right)^{1/p}\\
        & \leq 
  c\snr{h}^{\beta_{0}\alpha}M^{q-{\rm b}/p}|B_{h}|^{(p-1)/p}\rr^{\alpha}\left(\int_{8B_{h}} (\snr{Du_{\rr}}+1)^{q-p+{\rm b}}\dx\right)^{1/p}\,. \label{confronta}
\end{flalign}
Connecting the content of the last displays to \rif{torna} yields 
\eqn{nonabbiamo2}
$$
\int_{8B_{h}}\snr{V(Du_{\rr})-V(Dv)}^{2}\dx   \leq c\snr{h}^{\beta_{0}\alpha}M^{q-{\rm b}/p}\rr^{\alpha}|h|^{\frac{\beta_0n(p-1)}{p}} \|\snr{Du_{\rr}}+1\|_{L^{q-p+{\rm b}}(8B_{h})}^{\frac{q-p+{\rm b}}{p}}
$$
where $c\equiv c(\data)$. Using \rif{nonabbiamo2} in \rif{1bispre}, and taking $\beta_0=2/(2+\alpha)$, we arrive at the following analog of \rif{38}:
\begin{flalign}
\notag \int_{B_h}\snr{\tau_{h}(E(Du_{\rr})-\kk)_{+}}^{2} \dx
&\leq  c\snr{h}^{\frac{2\alpha}{2+\alpha}}M^{\ssf(q-p)}\int_{8B_{h}}(E(Du_{\rr})-\kk)_{+}^{2} \dx\nonumber\\
&\quad +c\snr{h}^{\frac{2\alpha}{2+\alpha}}M^{(\ssf+1) q-\mathrm{b}/p}\rr^{\alpha}|h|^{\frac{\beta_0n(p-1)}{p}} \|\snr{Du_{\rr}}+1\|_{L^{q-p+{\rm b}}(8B_{h})}^{\frac{q-p+{\rm b}}{p}}\,.
\label{ricopriancora}
\end{flalign}
The last step is to perform the covering argument in Section \ref{hyb}, summing up inequalities in \rif{ricopriancora} over balls $B_{h}\equiv B_{k}$ for $k \leq \mathfrak{n}$ as done in \rif{ricopri}. The sum of the first terms in the right-hand sides can be dealt with as in \rif{proas0}. For the second ones, we use \rif{discrete} with $t=1/p$ and $a_k \equiv 
\|\snr{Du_{\rr}}+1\|_{L^{q-p+{\rm b}}(8B_{k})}^{q-p+{\rm b}}$, arguing as in \rif{proas2} we obtain
\begin{flalign}
\notag |h|^{\frac{\beta_0n(p-1)}{p}}\sum_{k\leq  \mathfrak n}\|\snr{Du_{\rr}}+1\|_{L^{q-p+{\rm b}}(8B_{k})}^{\frac{q-p+{\rm b}}{p}} &\stackleq{discrete} 
c\left(\sum_{k\leq  \mathfrak n}\|\snr{Du_{\rr}}+1\|_{L^{q-p+{\rm b}}(8B_{k})}^{q-p+{\rm b}}\right)^{1/p}\\
&\stackleq{finite}  c\|\snr{Du_{\rr}}+1\|_{L^{q-p+{\rm b}}(\BB)}^{\frac{q-p+{\rm b}}{p}}\,,\label{sumarg}
\end{flalign}
for $c \equiv c (n,p)$. 
From \rif{ricopriancora}, after summation, we arrive at
\begin{flalign*}
\notag \int_{\BBB}\snr{\tau_{h}(E(Du_{\rr})-\kk)_{+}}^{2} \dx
&\leq  c\snr{h}^{\frac{2\alpha}{2+\alpha}}M^{\ssf(q-p)}\int_{\BB}(E(Du_{\rr})-\kk)_{+}^{2} \dx\nonumber\\
&\qquad +c\snr{h}^{\frac{2\alpha}{2+\alpha}}M^{(\ssf+1) q-\mathrm{b}/p}\rr^{\alpha} \|\snr{Du_{\rr}}+1\|_{L^{q-p+{\rm b}}(\BB)}^{\frac{q-p+{\rm b}}{p}}\,,
\end{flalign*}
and the rest of the proof \rif{xx.10-sotto2} can be now obtained as in Proposition \ref{caccin1}. 
\end{proof}

\subsection{Equations} \label{variante2e} The proof of Proposition \ref{caccin4} makes little use of minimality of $u$. This leads to extend its content to solutions to general equations as
\eqn{eulera}
$$-\diver\,  A(x, Du)=0 \qquad \mbox{in $B_{r}$}\,,$$  that are not necessarily arising as Euler-Lagrange equations of any functional. For this, we consider a general vector field 
$A\colon B_{r} \times \er^n\to \er^n$ satisfying \rif{assAA}, with $\nu, L$ replaced by general constants $\tilde \nu, \tilde L$ as in Section \ref{basicn}, with $0 < \mu \leq 2$ and such that
\eqn{xx.3equ}
$$
\nu_0 [H_{\mu}(z)]^{(q-2)/2}\snr{\xi}^{2}+\tilde \nu [H_{\mu}(z)]^{(p-2)/2}\snr{\xi}^{2}\le \partial_{z}A(x,z)\xi\cdot \xi 
$$  
holds with the same notation of \rif{assAA}. Let us record the following coercivity inequality:
\eqn{coerpq}
$$
\snr{z_2}^{p} \leq c [H_{1 }(z_1)]^{p(q-1)/[2(p-1)]}  + c A(x,z_2)\cdot(z_2-z_1) \,,
$$ which is valid under the assumptions satisfied by $A_0(\cdot)$
for every $z_1, z_2 \in \er^n$, $x\in B_{r}$, where $c\equiv c(\data)\geq 1$. The proof of \rif{coerpq} relies on \rif{hamine2} and is a minor variant of the one in \cite[Lemma 4.4]{M1}. We shall argue under the permanent assumption 
\eqn{marcexp2}
$$
\frac qp < 1+\frac 1n\,.
$$
Accordingly, in the present setting a relevant exponent is given by
\eqn{marcexpeq}
$$
\sst:= \frac{p}{(n+1)p-nq}
\frac{q-1}{p-1} \,.
$$
Note that $\sst\geq 1$ and $\sst=1$ when $p=q$; $\sst$ is well defined thanks to \rif{marcexp2}. This exponent plays for equations the same role that $\ssf$ in \rif{marcexp} plays for functionals; note that $\sst \geq \ssf$. 
\begin{proposition}\label{caccin5}
Let $u\in W^{1,q}(B_{r})$ be a weak solution to \trif{eulera}, under assumptions \eqref{assAA} with $0 < \mu \leq 2$  and $\nu, L$ replaced by $\tilde \nu, \tilde L$ (as in Section \ref{basicn}); also assume \eqref{xx.3equ} and \trif{marcexp2}. Let $B_{\rr}(x_{0}) \Subset B_{r}$ and let $M\geq 1$ be a constant such that $ \nr{Du}_{L^{\infty}(B_{\rr}(x_{0}))}\leq M$. Let $\kk \geq 0$ be a number. 
\begin{itemize}
\item If  $p\geq 2$, then
\begin{flalign}\label{xx.10eq}
\notag \left(\mint_{B_{\rr/2}(x_{0})}(E_{\mu}(Du)-\kk)_{+}^{2\chi}  \dx\right)^{1/\chi} & \le cM^{2\sst (q-p)} \mint_{B_{\rr}(x_{0})}(E_{\mu}(Du)-\kk)_{+}^{2} \dx \\ & \qquad  +cM^{\sst (2q-p)+p-\mathrm{b}}\rr^{2\alpha}\mint_{B_{\rr}(x_{0})}(\snr{Du}+1)^{2q-2p+\mathrm{b}} \dx
\end{flalign}
holds for every $\mathrm{b} \in [0, p]$ and $(\beta, \chi)$ as in \trif{bechi2}, where $c \equiv c(\data, \beta)$. 
\item If $1<p<2$, then 
\begin{flalign}\label{xx.10-sotto2eq}
\notag&  \left(\mint_{B_{\rr/2}(x_{0})}(E_{\mu}(Du)-\kk)_{+}^{2\chi}  \dx\right)^{1/\chi}  \le cM^{2\sst(q-p)} \mint_{B_{\rr}(x_{0})}(E_{\mu}(Du)-\kk)_{+}^{2} \dx \\
&\hspace{35mm} +cM^{(\sst+1)q+\sst(q-p)-{\rm b}/p}\rr^{\alpha}\left(\mint_{B_{\rr}(x_{0})}(\snr{Du}+1)^{\frac{p(q-1)}{p-1}-p+{\rm b}}\dx\right)^{1/p}
\end{flalign}
holds for every $\mathrm{b}  \in [0, p]$, and $(\beta, \chi)$ as in \trif{bechi} with $\alpham:=\alpha$, where $c\equiv c(\data, \beta)$.
\end{itemize}
\end{proposition}
\begin{proof}[Proof of Proposition \ref{caccin5}] We modify the proof of Proposition \ref{caccin4}. Keeping the notation introduced there, we arrive up to \rif{eulerieq}, where again it is $A_{0}(z):=A_{\varrho}(x_{\rm c},z)$.  Now, while we can still use \rif{conosc}-\rif{conosc2}, that hold for solutions to general equations, we have to find replacements for \rif{globale0}-\rif{fuckbsc}, that are linked to minimality. Thanks to \rif{marcexp2}, we can use \rif{stimaprimissima}, that yields
$$
\|Dv\|_{L^{\infty}(4B_{h})} \leq c \left(\mint_{8B_{h}}(\snr{Dv}+1)^p\dx\right)^{ \frac{1}{(n+1)p-nq}}\,.
$$
Using \rif{coerpq} (applied to $A_0(\cdot)$) with $z_2\equiv Dv$, $z_1\equiv Du_{\rr}$, and integrating over $8B_h$, gives
\eqn{marciana}
$$
\mint_{8B_{h}}\snr{Dv}^p\dx \leq c \mint_{8B_{h}}(\snr{Du_{\rr}}+1)^{\frac{p(q-1)}{p-1}}\dx\,.
$$
Matching the inequalities in the last two displays, and recalling the definition in \rif{marcexpeq},  we conclude with
the following analog of
\rif{fuckbsc}:
\eqn{fuckbsceq}
$$
\|Dv\|_{L^{\infty}(4B_{h})} \leq \tilde c M^{\sst}\,, \qquad \tilde c \equiv \tilde c (n,p,q,\nu, L)\,.
$$
This inequality allows to get analogs of \rif{1}-\rif{1bispre}, via this time the use of \rif{pri-cacc-2}, that is 
\begin{eqnarray}
\notag &&\int_{B_{h}}\snr{\tau_{h}(E(Du_{\rr})-\kk)_{+}}^{2}  \dx\\
&&\quad \leq c\snr{h}^{2(1-\beta_{0})}M^{2\mathfrak{t}(q-p)}\int_{4B_{h}}(E(Du_{\rr})-\kk)_{+}^{2}  \dx+cM^{\mathfrak{t}(2q-p)}\int_{4B_{h}}\snr{V(Du_{\rr})-V(Dv)}^{2}  \dx\,.\label{1bispp}
\end{eqnarray}
When $p\geq 2$, combining this last estimate with 
\rif{nonabbiamo}, that holds for general equations too, we finally arrive at \rif{xx.10eq} as in Proposition \ref{caccin4}. It again remains to treat the case $p<2$.  By looking at $c\mbox{(I)} $, defined in \rif{torna}, as in \rif{confronta}, we have
\begin{eqnarray}
 \notag \mbox{(I)}  &  \leq &
  c\snr{h}^{\beta_{0}\alpha}M^{q-1}|B_{h}|^{(p-1)/p}\rr^{\alpha}\left(\int_{8B_{h}} (\snr{Du_{\rr}}^p+\snr{Dv}^p)\dx\right)^{1/p}\\
 & \stackleq{marciana} &
  c\snr{h}^{\beta_{0}\alpha}M^{q-{\rm b}/p}|B_{h}|^{(p-1)/p}
  \rr^{\alpha}\left(\int_{8B_{h}} (\snr{Du_{\rr}}+1)^{\frac{p(q-1)}{p-1}-p+{\rm b}}\dx\right)^{1/p}\,.\label{servizio}
\end{eqnarray}
Connecting this to \rif{torna}, and the resulting inequality to \rif{1bispp}, yields 
\begin{flalign*}
&\notag \int_{B_h}\snr{\tau_{h}(E(Du_{\rr})-\kk)_{+}}^{2} \dx
\leq  c\snr{h}^{\frac{2\alpha}{2+\alpha}}M^{2\mathfrak{t}(q-p)}\int_{8B_{h}}(E(Du_{\rr})-\kk)_{+}^{2} \dx\nonumber\\
&\, \qquad \qquad  \qquad  +c\snr{h}^{\frac{2\alpha}{2+\alpha}}cM^{(\sst+1)q+\sst(q-p)-{\rm b}/p}\rr^{\alpha}|h|^{\frac{\beta_0n(p-1)}{p}}\left(\int_{8B_{h}} (\snr{Du_{\rr}}+1)^{\frac{p(q-1)}{p-1}-p+{\rm b}}\dx\right)^{1/p}
\end{flalign*}
where $c\equiv c(\data)$. After this, we can proceed as after \rif{ricopriancora}, including the summation argument in \rif{sumarg}, that works verbatim also with the new exponents, finally leading to \rif{xx.10-sotto2eq}. 
\end{proof}
\subsection{A priori H\"older}\label{giustifica}
We briefly justify the validity of \rif{33.1bis}, which is in fact completely standard when $\hhh(\cdot)\equiv 0$ \cite{gg2, manth1, manth2}. We report some details of the proof of the result, that, in the form stated here, does not seem to be explicitly mentioned in the literature, although it can be obtained by standard arguments. 
Here we consider a minimizer of $u\in W^{1,q}(B_{r})$ of the functional $\mathcal F(\cdot, B_{r})$ in \trif{Fx}, and we assume \eqref{asshh22} and that $\FF(\cdot)\equiv  \tilde F(\cdot)$ satisfies conditions \eqref{F_0}. Therefore we are covering the functionals described in Sections \ref{laprima}, \ref{variante3} and \ref{variante2f}. With $B_{\tau}(x_{\rm c})\Subset B_{r}$ being a ball, we define $v \in u +W^{1,q}_0(B_{\tau})$ as the unique minimizer of $w \mapsto \int_{B_{\tau}}\FF(x_{\rm c}, (u)_{B_{\tau}}, Dw)\dx$ in its Dirichlet class. This satisfies the following a priori estimates 
\eqn{beep1}
$$
\|Dv\|_{L^{\infty}(B_{\tau/2})} \leq c \left(\mint_{B_\tau}(\snr{Dv}+1)^q\dx \right)^{1/q} \  \mbox{and} \  
\osc\, (Dv, B_\varrho) \leq c \left(\frac{\rr}{\tau}\right)^{\beta_1} \|Dv\|_{L^{\infty}(B_{\tau})}
$$ 
for every $\rr \leq \tau$, where both $c\geq 1$ and $\beta_1\in (0,1)$ depend on $n,q,L,L_0,\mu,\nu_0$; see for instance \cite[(2.4)-(2.5)]{manth2}. Moreover, by Lemma \ref{oscilemma} it is $\osc (v,B_{\tau}) \leq \osc (u,B_{\tau})$ and by minimality it is $\|Dv\|_{L^q(B_\tau)}\lesssim \|\snr{Du}+1\|_{L^q(B_\tau)}$. Next, note that also thanks to \rif{asshh22}$_2$, the integrand $\FF(\cdot)+ \hhh(\cdot)$ satisfies the growth conditions 
$ \snr{z}^q - |y| -1 \lesssim \FF(x,y,z)+ \hhh(x,y)\lesssim \snr{z}^q + |y| +1$, with implied constants depending on $n,q,L,L_0,\mu,\nu_0$. Therefore, by De Giorgi-Nash-Moser theory it follows that $u \in C^{0, \beta_2}_{\loc}(\Omega)$ for some $\beta_2\equiv \beta_2(n,q,L,L_0,\mu,\nu_0)\in (0,1)$. In turn, by this last fact and Maz'ya-Wiener theory, it follows that there exists another positive exponent $\beta_3\equiv \beta_3(n,q,L,L_0,\mu,\nu_0)\leq \beta_2$,  such that 
$v \in C^{0, \beta_3}_{\loc}(\overline{B_\tau})$. 
For these results, see for instance \cite[Theorem 7.6]{giu} and \cite[Section 7.8]{giu}. Using  these last facts, a modification of the comparison arguments in \cite{gg1, manth1}, taking into account the presence of $\hhh(\cdot)$, gives 
\eqn{beep2}
$$
\int_{B_\tau}(\snr{Du}^2+\snr{Dv}^2+1)^{\frac{q-2}{2}}|Du-Dv|^2\dx  \leq c \tau^{\alpha \beta_3} \int_{B_\tau}(\snr{Du}+1)^q\dx\,,
$$
where $c\equiv c (n,q,L,L_0,\mu,\nu_0)$. See for instance \cite[Lemma 4.9]{KM}. 
Estimates \rif{beep1}-\rif{beep2} can be combined in a by now standard way to prove \rif{33.1bis}, as for instance described in \cite{manth2} or in \cite{AM}.

\section{Theorems \ref{t1} and \ref{t5}}\label{1e5}
We first develop a priori estimates in Section \ref{appi}. Those for Theorem \ref{t1} are in Proposition \ref{priori1}, while Proposition \ref{priori2} contains those for Theorem \ref{t5}. In both cases, the setting is that of Section \ref{laprima}. These estimates, obtained for minima of more regular functionals, are then embedded in an approximation argument contained in Section \ref{generalapp}. At that stage we have proved that $Du$ is locally bounded in Theorems \ref{t1} and \ref{t5}, thereby completing the proof of the latter. The proof that $Du$ is locally H\"older continuous when $f \in L^{\mathfrak{q}}$ for $\mathfrak q > n/\alpha$ is given in Section \ref{holdergrad}, and completes the derivation of Theorem \ref{t1}. Some of the arguments developed in Sections \ref{appi} and \ref{generalapp} will be employed also for the proofs of the remaining theorems of this paper.

\subsection{A priori $L^{\infty}$-bounds for Theorems \ref{t1} and \ref{t5}}\label{appi}
\begin{proposition}\label{priori1}
Let $u\in W^{1,q}(B_{r})$ be a minimizer of the functional in $\mathcal G(\cdot, B_{r})$ in \trif{ggg2}, where $B_{r}\Subset \Omega$ and $r\leq 1$, under assumptions \eqref{32.2}-\eqref{asshh22}. There exists an explicitly computable 
function $\kappa_1(n,p, \alpha, \gamma)$, with 
\eqn{soprasotto}
$$1/5< \kappa_1(\cdot) < 1\,,$$
such that, if
\eqn{bound1dopo}
$$
\frac{q}{p}< 1+\kappa_1(n,p,\alpha,\gamma)\left(1-\frac{\alpha+\gamma}{p}\right)\frac{\alpha}{n}\,,
$$
then 
\eqn{stima1pre0}
$$
\nr{Du}_{L^{\infty}(B_{t})}\le \frac{c}{(s-t)^{\chi_1}}\left[\|Du\|_{L^p(B_s)} +\|f\|_{n/\alpha,1/2;B_{s}}+1\right]^{\chi_2}
$$
holds whenever $B_{t}\Subset B_s\Subset B_{r}$ are concentric balls, where $c\equiv c (\data, \gamma)\geq 1$ and $\chi_1, \chi_2\equiv \chi_1, \chi_2\linebreak (\datae, \gamma)\geq 1$. In particular, by \trif{soprasotto} it follows that condition \trif{bound1dopo} implies \trif{marcbound} and that \trif{bound1} implies \trif{bound1dopo}. Ultimately, if \trif{bound1} holds, then \trif{stima1pre0} holds as well.
\end{proposition}
\begin{proof} The function $\kappa_1(\cdot)$ will be computed in due course of the proof; it will be such that $p,q$ will verify the initial condition in \rif{marcbound}, which is necessary in order to apply Propositions \ref{caccin1}-\ref{caccin2}. Therefore we proceed assuming \rif{marcbound}. By the assumptions considered we observe that $u$ satisfies \rif{33.1bis}; in particular, $Du$ is locally bounded in $B_{r}$, that is what we ultimately need at this stage. Let $B_{t}\Subset B_s\Subset B_{r}$ be balls as in the statement of Proposition \ref{priori1}, and note that we can assume that $\nr{Du}_{L^{\infty}(B_{t})}\ge 1$, otherwise \eqref{stima1pre0} is trivial. Next we consider further concentric balls $B_{t}\Subset B_{\tau_{1}}\Subset B_{\tau_{2}}\Subset B_{s}$, and a generic point $x_{0} \in B_{\tau_{1}}$; note that every one of such points is a Lebesgue point for $E_{\mu}(Du)$ by virtue of \rif{33.1bis}. We let $r_0:=(\tau_{2}-\tau_{1})/8$, so that $B_{2r_0}(x_{0})\Subset B_{\tau_{2}}$. By \eqref{40prima} used with $M\equiv \nr{Du}_{L^{\infty}(B_{\tau_{2}})}\geq 1$, 
we can apply Lemma \ref{revlem} on $B_{r_0}(x_0)$ with $h\equiv 2$, $\kk_{0}\equiv 0$, $v\equiv E_{\mu}(Du)$, $f_1\equiv 1$, $f_2\equiv f$, $M_{0}\equiv M^{\ssf (q-p)/2}$, $M_{1}\equiv M^{(\ssf q+ \alpha+\gamma q/p)/2}$, $M_{2}\equiv M^{(\mathfrak{s}q + \alpha)/2}$, $t\equiv 2$, $\delta_{1}=\delta_{2}\equiv \alpha/2$, $m_1 = m_2\equiv 1$, $\theta_1=\theta_2\equiv 1$; also using \rif{trivialpot}, this yields
\begin{flalign}
\notag E_{\mu}(Du(x_{0})) &\le c\nr{Du}_{L^{\infty}(B_{\tau_{2}})}^{\frac{\ssf (q-p)\chi}{2(\chi-1)}}\left(\mint_{B_{r_0}(x_{0})}[E_{\mu}(Du)]^{2}  \dx\right)^{1/2}+c_{\gggg}
\nr{Du}_{L^{\infty}(B_{\tau_{2}})}^{\frac{\ssf (q-p)\chi}{2(\chi-1)}+\frac{\ssf p+\alpha+\gamma q/p}{2}}r_0^{\frac{\alpha}{2}}\\
&
\qquad +c\nr{Du}_{L^{\infty}(B_{\tau_{2}})}^{\frac{\ssf(q-p)\chi}{2(\chi-1)}+\frac{\mathfrak{s}p+\alpha}{2}}\mathbf{P}^{1,1}_{2,\alpha/2}(f;x_{0},2r_0)  \label{440}
\end{flalign}
with $c\equiv c(\data, \beta)$ and $\chi\equiv \chi(\beta)$ is as in \rif{bechi} with $\alpham:=\alpha$. Since $x_{0} \in B_{\tau_{1}}$ is arbitrary we also gain, after a few elementary manipulations
\begin{flalign}\label{44}
\nr{Du}_{L^{\infty}(B_{\tau_{1}})}&\le\frac{c}{(\tau_{1}-\tau_{2})^{\frac{n}{2p}}}
\nr{Du}_{L^{\infty}(B_{\tau_{2}})}^{\left(\frac qp-1\right)\frac{\ssf \chi}{2(\chi-1)}+\frac 12}\left(\int_{B_{s}}(\snr{Du}+1)^p  \dx\right)^{\frac{1}{2p}}\nonumber \\
&\qquad +c_{\gggg}\nr{Du}_{L^{\infty}(B_{\tau_{2}})}^{\left(\frac qp-1\right)\frac{\ssf \chi}{2(\chi-1)}+\frac\ssf 2+\frac{\alpha+\gamma q/p}{2p}}r_0^{\frac{\alpha}{2p}}\nonumber \\  &\qquad  +c\nr{Du}_{L^{\infty}(B_{\tau_{2}})}^{\left(\frac qp-1\right)\frac{\ssf \chi}{2(\chi-1)}+\frac { \ssf } {2}+\frac {\alpha} {2p}}\nr{\mathbf{P}^{1,1}_{2,\alpha/2}(f;\cdot,2r_0)}_{L^{\infty}(B_{\tau_{1}})}^{1/p}+c\,.
\end{flalign}
We have used that $B_{r_0}(x_0)\subset B_{s}$.  
In the above display, the constant $c$ depends on $\data$ and $\beta$ in such a way that $c\to \infty$ as $\beta$ approaches its upper limit in \rif{bechi} (with $\alpham=\alpha$; this ultimately comes from the application of Lemma \ref{l4} at the end of the proof of Proposition \ref{caccin1}). Note also that, in order to derive \rif{44}, in \rif{440} the second integral in \rif{440} has been estimated as follows (we also use $\nr{Du}_{L^{\infty}(B_{\tau_{2}})} \geq 1$):
\begin{flalign}\label{estimateas}
\notag\mint_{B_{r_0}(x_{0})}[E_{\mu}(Du)]^{2}  \dx  & \leq \|E_{\mu}(Du)\|_{L^{\infty}(B_{r_0}(x_{0}))}
\mint_{B_{r_0}(x_{0})}(\snr{Du}+1)^p \dx\\ & \leq c\nr{Du}_{L^{\infty}(B_{\tau_{2}})}^p \mint_{B_{r_0}(x_{0})}(\snr{Du}+1)^p \dx\,.
\end{flalign}
To proceed with the proof, by \rif{11} and $B_{\tau_{1}+(\tau_{2}-\tau_{1})/4}\subset B_{s}$, we have  
\eqn{vvee1}
$$
\nr{\mathbf{P}^{1,1}_{2,\alpha/2}(f;\cdot,2r_0)}_{L^{\infty}(B_{\tau_{1}})} =\nr{\mathbf{P}^{1,1}_{2,\alpha/2}(f;\cdot,(\tau_{2}-\tau_{1})/4)}_{L^{\infty}(B_{\tau_{1}})} \leq c\|f\|^{1/2}_{n/\alpha,1/2;B_{s}}\,.
$$
Looking at \rif{44}, we are led to consider the term with the highest power of $\nr{Du}_{L^{\infty}(B_{\tau_{2}})}$, which happens to be the second one appearing in the right-hand side. We are now interested in studying the range of parameters $n,p,q,\alpha, \beta, \gamma$ for which the following inequality holds true:
\eqn{condy}
$$
\left(\frac qp-1\right)\frac{\ssf \chi(\beta)}{2[\chi(\beta)-1]}+\frac\ssf 2+\frac{\alpha+\gamma q/p}{2p}=
\left(\frac qp-1\right)\frac{\ssf n}{4\beta}+\frac\ssf 2+\frac{\alpha+\gamma q/p}{2p}<1
 \,.
$$
Recalling \rif{marcexp} and \eqref{bechi}, reformulating \rif{condy} in terms of $q-p$ leads to
$$
\frac{(q-p)}{2}\left(
\frac{n[(q-p)+p]}{\beta[2p-n(q-p)]}+\frac{\gamma}{p}
\right)
+ \frac{p[(q-p)+p]}{2p-n(q-p)} +\frac{\alpha+\gamma}{2}<p\,.
$$
Via routine computations, we arrive at the second order polynomial inequality in $q-p$:
\eqn{roots0}
$$
\mathfrak{A}(\beta)(q-p)^2+ \mathfrak{B}(\beta)(q-p)+ \mathfrak{C}(\beta)<0\,,
$$
where 
\eqn{ABC}
$$
\begin{cases}
\, \mathfrak{A}(\beta):= n(p-\gamma\beta)>0\\
\,  \mathfrak{B}(\beta):= p[np+2\beta(p+\gamma)+n\beta(2p-\alpha-\gamma)]>0\\
\, \mathfrak{C}(\beta):= -2p^2\beta[p-(\alpha+\gamma)]<0\,.
\end{cases}
$$
Computing the roots of \rif{roots0}, yields a first bound on $q-p$, i.e., 
 \eqn{roots1}
$$
 q-p  < \frac{\mathfrak{B}(\beta)}{2\mathfrak{A}(\beta)}\left(
 \sqrt{1-\frac{4\mathfrak{A}(\beta)\mathfrak{C}(\beta)}{[\mathfrak{B}(\beta)]^2}}-1
 \right) =\frac{\mathfrak{B}(\beta)}{2\mathfrak{A}(\beta)} \left(
 \sqrt{1+\left(1-\frac{\alpha+\gamma}{p}\right)\frac{\beta}{n}\mathfrak{H}(\beta)}-1\right)
$$
where, by \rif{ABC}, it is
\eqn{ABCDE}
$$
\mathfrak{H}(\beta):= \frac{8p(p-\gamma \beta)}{[p+2\beta(p+\gamma)/n+\beta(2p-\alpha-\gamma)]^2}< \frac{8}{(1+\beta)^2}\,,
$$
so that
\eqn{ABCD}
$$
 \frac{\mathfrak{B}(\beta)\mathfrak{H}(\beta)}{2p\mathfrak{A}(\beta)} = \frac{4p}{p+2\beta(p+\gamma)/n+\beta(2p-\alpha-\gamma)} < 4 \,.
$$
Introducing the function
\eqn{funkk}
$$
\kappa_{\alpham}(n, p, \beta, \gamma):=\frac{1}{2+\alpham}\frac{\mathfrak{B}(\beta)\mathfrak{H}(\beta)}{2p\mathfrak{A}(\beta)} 
\mathfrak{S}\left(\left(1-\frac{\alpha+\gamma}{p}\right)\frac{\beta}{n}\mathfrak{H}(\beta)\right)\,,
$$
where
\eqn{grandeS}
 $$\mathfrak{S}(t):= \frac{\sqrt{1+t}-1}{t}  \qquad t>0\,,$$
we can rewrite \rif{roots1} as
\eqn{roots2}
$$
 \frac qp < 1+\kappa_{\alpham}(n, p, \beta, \gamma)\frac{\beta (2+\alpham)}{\alpham}\left(1-\frac{\alpha+\gamma}{p}\right)\frac{\alpham}{n}\,.
$$
Note that, although here it is $\alpham=\alpha$, we keep on using the notation $\alpham$ to employ the same computations later on. We determine $\kappa_1(\cdot)$ in \rif{bound1dopo} as 
$\kappa_1(n, p,\alpha, \gamma) := \kappa_{\alpham}(n, p,\gamma, \alpham/(2+\alpham),  \gamma)$
so that \rif{soprasotto} follows from Lemma \ref{funkklemma} below.    
Recalling that $\kappa_{\alpham}(\cdot)$ is a continuous function, we conclude that, if \rif{bound1dopo} holds, then we can find $\beta< \alpham/(2+\alpham)$ such that \rif{roots2}, and therefore \rif{condy}, holds. 
This allows to apply Young's inequality in \rif{44}; taking also into account \rif{vvee1} to estimate the terms containing the potentials, we come up with 
$$
\nr{Du}_{L^{\infty}(B_{\tau_{1}})}\leq \frac{1}{2}\nr{Du}_{L^{\infty}(B_{\tau_{2}})}+\frac{c}{(\tau_2-\tau_1)^{\chi_1}} \left[\|Du\|_{L^p(B_s)}   +\|f\|_{n/\alpha,1/2;B_{s}}+ 1\right]^{\chi_2}\,,
$$
where $c, \chi_1, \chi_2$ depend as described in the statement. 
We recall that $Du$ is locally bounded in $B_{r}$, therefore, applying Lemma \ref{l5} to the last inequality with the choice $h(\tau)\equiv  \|Du\|_{L^{\infty}(B_{\tau})}$, $s \leq \tau \leq t$, yields \rif{stima1pre0}. 
\end{proof}
\begin{lemma}\label{funkklemma} The function $ \kappa_{\alpham}(\cdot)$ defined in \trif{funkk} satisfies $1/5 < \kappa_{\alpham}(n, p,t,  \gamma)<1$, for every $t$ such that $0< t \leq 1/3$ and for every $\alpham \in (0,1]$. 
\end{lemma}
\begin{proof} The function $\mathfrak{S}(\cdot)$ in \rif{grandeS} is decreasing, so that $\mathfrak{S}(\cdot)\leq 1/2$ and by \rif{ABCD} it  follows that 
$
\kappa_{\alpham}(n, p, t, \gamma)< 1
$
for every $t > 0$. Thanks to \rif{ABCDE} and \rif{ABCD},  elementary estimations give 
$$
\frac {12}{7} < \frac{\mathfrak{B}(\beta)\mathfrak{H}(\beta)}{2p\mathfrak{A}(\beta)} \quad \mbox{and} \quad 
\frac{\beta}{n}\left(1-\frac{\alpha+\gamma}{p}\right)\mathfrak{H}(\beta) < \frac{8\beta}{n(1+\beta)^2}< \frac {3}{2n}\leq \frac 34\,.
$$
Using again that $t \mapsto \mathfrak{S}(t)$ is decreasing, and the content of the previous display, we get
$$
\frac 15 < \frac{16}{21}\left(\sqrt{7/4}-1\right)< \frac{1}{2+\alpham}\frac{\mathfrak{B}(\beta)\mathfrak{H}(\beta)}{2p\mathfrak{A}(\beta)} 
\mathfrak{S}(3/4)< \kappa_{\alpham}(n, p, \beta, \gamma)\,,
$$
and the lemma is proved. 
\end{proof}
\begin{proposition}\label{priori2}
Let $u\in W^{1,q}(B_{r})$ be a minimizer of the functional in $\mathcal G(\cdot, B_{r})$ in \trif{ggg2}, where $B_{r}\Subset \Omega$ and $r\leq 1$, under assumptions \eqref{32.2}-\eqref{asshh22}. Assume that \trif{soprasotto}-\trif{bound1dopo} hold and that $p>2n\alpha/[2n-2\alpha+\alpha^2]$.  
Then \eqn{stima1pre}
$$
\nr{Du}_{L^{\infty}(B_{t})}\le \frac{c}{(s-t)^{\chi_1}}\left[\|Du\|_{L^p(B_s)} +\|f\|_{n/\alpha, \textnormal{\texttt{l}};B_{s}}+1\right]^{\chi_2}
$$
holds whenever $B_{t}\Subset B_s\Subset B_{r}$ are concentric balls, where $ \textnormal{\texttt{l}}$ has been defined in \trif{lo}, $c\equiv c (\data, \gamma )$, and $\chi_1, \chi_2\equiv \chi_1, \chi_2(\datae )$. 
In the case $1 < p\leq 2n\alpha/[2n-2\alpha+\alpha^2]$, there exists an explicitly computable 
function $\kappa_{2}(n,p,\alpha, \gamma) $, with $1/5 < \kappa_{2}(\cdot) <1$, such that if
\eqn{bound10dopo}
$$
\frac{q}{p}< 1+\kappa_{2}(n,p,\alpha, \gamma) \left(1-\frac{\alpha+\gamma}{p}\right)\frac{2(p-\alpha)}{p(2-\alpha)}\,,
$$
then \trif{stima1pre} holds. 
\end{proposition}
\begin{proof} 
We argue as for Proposition \ref{priori1}. This time we employ \rif{40} with $M\equiv \nr{Du}_{L^{\infty}(B_{\tau_{2}})}\linebreak \geq 1$, so that, for any $x_{0} \in B_{\tau_{1}}$, 
we apply Lemma \ref{revlem} on $B_{r_0}(x_0)$, with $h\equiv 3$, $\kk_{0}\equiv 0$, $v\equiv E_{\mu}(Du)$, $f_1\equiv 1$, $f_2\equiv f_3 \equiv f$, $M_{0}\equiv M^{\ssf(q-p)/2}$, $M_{1}\equiv M^{(\ssf q+ \alpha+\gamma q/p)/2}$, $M_{2}\equiv M^{\ssf q/2}$, $M_{3}\equiv \una M^{\ssf q/2+\alpha(2-p)/[2(2-\alpha)]}$, $t\equiv 2$, $\delta_{1}\equiv \alpha/2$, $\delta_{2}\equiv p\alpha/[2(p-\alpha)]$, $\delta_{3}\equiv \alpha/(2-\alpha)$, $m_1\equiv 1$, $m_2= m_3\equiv \ppp/(\ppp-\alpha)$, $\theta_1 \equiv 1$, $\theta_2 \equiv \theta(\ppp)$ and $\theta_3 \equiv \sigma (\ppp)$. Similarly to Proposition \ref{priori1}, we arrive at
\begin{flalign}\label{44ll}
\nr{Du}_{L^{\infty}(B_{\tau_{1}})}&\le\frac{c}{(\tau_{1}-\tau_{2})^{\frac{n}{2p}}}
\nr{Du}_{L^{\infty}(B_{\tau_{2}})}^{\left(\frac qp-1\right)\frac{\ssf \chi}{2(\chi-1)}+\frac 12}\left(\int_{B_{s}}(\snr{Du}+1)^p  \dx\right)^{\frac{1}{2p}}\nonumber \\
&\ +c_{\gggg}\nr{Du}_{L^{\infty}(B_{\tau_{2}})}^{\left(\frac qp-1\right)\frac{\ssf \chi}{2(\chi-1)}+\frac\ssf2+\frac{\alpha+\gamma q/p}{2p}}r_0^{\frac{\alpha}{2p}}\nonumber \\ 
 &\  +c\nr{Du}_{L^{\infty}(B_{\tau_{2}})}^{\left(\frac qp-1\right)\frac{\ssf \chi}{2(\chi-1)}+\frac \ssf2}\left\|\mathbf{P}^{\frac{\ppp}{\ppp-\alpha},\theta(\ppp)}_{2,\frac{p\alpha}{2(p-\alpha)}}(f;\cdot,(\tau_{2}-\tau_{1})/4)\right\|_{L^{\infty}(B_{\tau_{1}})}^{1/p}\nonumber \\  
 &\   +c\una \nr{Du}_{L^{\infty}(B_{\tau_{2}})}^{\left(\frac qp-1\right)\frac{\ssf \chi}{2(\chi-1)}+\frac \ssf 2+\left(\frac 2p-1\right)\frac{\alpha}{2(2-\alpha)}}\left\|\mathbf{P}^{\frac{\ppp}{\ppp-\alpha},\sigma(\ppp)}_{2,\frac{\alpha}{2-\alpha}}(f;\cdot,(\tau_{2}-\tau_{1})/4)\right\|_{L^{\infty}(B_{\tau_{1}})}^{1/p}+c\,.
\end{flalign}
In the following we shall always take $\ppp$ to be such that 
\eqn{sceltappp}
$$
\ppp>\frac{n\alpha}{n-\alpha}=:\pppm <2\,,
$$
and proceed checking that Lemma \ref{crit} can be applied to estimate the last two terms from \rif{44ll}. This ultimately boils down to check that \rif{lo.2.1} is satisfied with the current choice of parameters. Indeed we have 
$$
\frac{n\theta }{t\delta} \equiv  \frac{n\theta(\ppp) }{2\delta_{2}}=\frac{n\sigma(\ppp) }{2\delta_{3}}=\frac{n(\ppp-\alpha)}{\ppp\alpha}\stackrel{\rif{sceltappp}}{>}1\,.
$$
Therefore Lemma \ref{crit} applies, giving, in any case $p>1$
\eqn{tt1}
$$
\left\|\mathbf{P}^{\frac{\ppp}{\ppp-\alpha},\theta(\ppp)}_{2,\frac{p\alpha}{2(p-\alpha)}}(f;\cdot,(\tau_{2}-\tau_{1})/4)\right\|_{L^{\infty}(B_{\tau_{1}})}\leq  c\|f\|^{\frac{p}{2(p-\alpha)}}_{\frac{n}{\alpha},\frac{p}{2(p-\alpha)};B_{s}} \stackrel{\rif{lo},\rif{lorentzbasic}}{\leq} c\|f\|^{\textnormal{\texttt{l}}}_{n/\alpha,\textnormal{\texttt{l}};B_{s}}  
$$
and (this one occurring only when $p<2$)
\eqn{tt1p}
$$
\left\|\mathbf{P}^{\frac{\ppp}{\ppp-\alpha},\sigma(\ppp)}_{2,\frac{\alpha}{2-\alpha}}(f;\cdot,(\tau_{2}-\tau_{1})/4)\right\|_{L^{\infty}(B_{\tau_{1}})}\leq  c\|f\|^{\frac{1}{2-\alpha}}_{\frac{n}{\alpha},\frac{1}{2-\alpha};B_{s}}\stackrel{\rif{lo}}{=} c\|f\|^{\textnormal{\texttt{l}}}_{n/\alpha,\textnormal{\texttt{l}};B_{s}} 
\,,
$$
where $c\equiv c(n,p,\alpha)$.
Back to \rif{44ll}, and noting that
\eqn{notache}
$$
\left(\frac qp-1\right)\frac{\ssf \chi}{2(\chi-1)}+\frac \ssf2+\left(\frac 2p-1\right)\frac{\alpha}{2(2-\alpha)} \leq 
\left(\frac qp-1\right)\frac{\ssf \chi}{2(\chi-1)}+\frac\ssf 2+\frac{\alpha+\gamma q/p}{2p}\,,
$$
we are again led to determine $\beta <\alpham/(2+\alpham)$ such that \rif{condy} holds. For this, we distinguish three different cases. 
\begin{itemize}
\item Case 1: $p>\pppm$. We take $\mathfrak{p}=p$, that gives $\mathfrak{a}(\mathfrak{p})=\mathfrak{a}(p)=1$. It follows that $\alpham=\alpha$ in \rif{alphagen} and we conclude exactly as in Proposition \ref{priori1}, using the bound \rif{roots2} with $\alpham=\alpha$. Specifically, we take $\beta<\alpham/(2+\alpham)$ in order to satisfy \rif{condy}, then, using Young's inequality in \rif{44ll}, and taking also into account \rif{tt1}-\rif{tt1p}, we come up with 
$$
\nr{Du}_{L^{\infty}(B_{\tau_{1}})}\leq \frac{1}{2}\nr{Du}_{L^{\infty}(B_{\tau_{2}})}+\frac{c}{(\tau_2-\tau_1)^{\chi_1}} \left[\|Du\|_{L^p(B_s)}    +\|f\|_{n/\alpha, \textnormal{\texttt{l}};B_{s}}+ 1\right]^{\chi_2}
$$
where $c, \chi_1, \chi_2$ depend as described in the statement. Applying Lemma \ref{l5} yields \eqref{stima1pre} and the proof is complete. 
\item Case 2: $2n\alpha/(2n-2\alpha+\alpha^2)<p\leq \pppm$. The lower bound on $p$ implies $2\mathfrak{a}(\pppm)/(2-\alpha)>1$. It follows we can take $\mathfrak{p}>\pppm$ close enough to $\pppm$ in order to have $\alpham=\alpha$ in \rif{alphagen}. We then argue as in Case 1. 
\item Case 3: $1<p\leq 2n\alpha/(2n-2\alpha+\alpha^2)$. The upper bound on $p$ implies that $2\mathfrak{a}(\pppm)/(2-\alpha)\leq 1$. Recalling that $t \mapsto \mathfrak{a}(t)$ is decreasing, we have that
$$
\ppp> \pppm \Longrightarrow \alpham (\ppp)=\frac{2\alpha \mathfrak{a}(\mathfrak{p})}{2-\alpha}<
\frac{2\alpha \mathfrak{a}(\pppm)}{2-\alpha}=\alpham(\pppm)= \frac{2(p-\alpha)}{p(2-\alpha)} n\leq \alpha \leq 1\,.
$$
We formally take the limiting value $\ppp = \pppm$ and consider the corresponding version of \rif{roots2} with $\alpham\equiv \alpham(\pppm)$, $\beta= \alpham(\pppm)/(2+\alpham(\pppm))$, and finally set $\kappa_2(n, p,\alpha, \gamma) := \kappa_{\alpham}(n, p,\gamma, \alpham(\pppm)/(2+\alpham(\pppm)),  \gamma)$ ($> 1/5$ by Lemma \ref{funkklemma}). This leads to \rif{bound10dopo}. Summarising, if \rif{bound10dopo} holds, then we can find $\ppp>\pppm$ and $\beta < \alpham(\ppp)/(2+\alpham(\ppp))$ such that \rif{roots2}, this time holds with $\alpham \equiv \alpham(\ppp)$, and therefore also \rif{condy} hold. We then conclude as in Case 1. 
\end{itemize}\end{proof}
\begin{remark}[Refinements] \label{raffina}\emph {The lower bound $\kappa_1(\cdot)> 1/5$ in \rif{soprasotto} can be improved by using Taylor expansions of the function $\mathfrak{S}(\cdot)$ in \rif{grandeS}. This eventually leads to a slightly better bound  than that in \rif{bound1}. Anyway, this does not lead to a different asymptotic in terms of the ratio $\alpha/n$ and we will not pursue this path here. Further improvements come when considering more specific structures, as in for instance in \rif{ggg}. Let us for simplicity consider the range $p\geq 2$. In this case, looking at \rif{44ll} and recalling that $c_{\gggg}=0$, we arrive at
\eqn{condy2}
$$
\left(\frac qp-1\right)\frac{\ssf n}{4\beta}+\frac\ssf 2<1\,,
$$
that replaces \rif{condy} in this setting. Then, taking also into account the content of Remark \ref{differente}, and performing the same reasoning of \rif{condy}-\rif{roots2} and Lemma \ref{funkklemma}, we obtain a refinement of \rif{bound1}, but still preserving the same asymptotic with respect to $\alpha/n$.} 
\end{remark}
\subsection{Approximation} \label{generalapp} Let $u\in W^{1,p}_{\loc}(\Omega)$ be a (local) minimizer of $\mathcal{G}$ as in Theorem \ref{t5}; in particular, $F(\cdot), \gggg(\cdot), \hhh(\cdot)$ satisfy \rif{assFF}-\rif{asshh}.  With $B_{r}\Subset \Omega$, $r \leq 1$, 
in the following by $\eps, \delta\equiv \{\eps\}, \{\delta\}$ we denote two decreasing sequences of positive numbers such that $\eps, \delta \to 0$, $\eps \leq  \dist(B_{r}, \partial \Omega)/10$ and $\eps,\delta \leq 1$; we will several times extract subsequences and these will still be denoted by $\eps, \delta$. We denote by $\texttt{o}(\varepsilon,B)$ a quantity, also depending on a considered ball $B$, such that $\texttt{o}(\varepsilon,B)\to 0$ as $\eps \to 0$. Similarly, we denote by $\texttt{o}_{\eps}(\delta,B)$ a quantity, depending both on $\eps$ and $\delta$, such that $\texttt{o}_{\eps}(\delta,B)\to 0$ as $\delta \to 0$ for each fixed $\eps$. As usual, the exact value of such quantities might change on different occurences. In the following we will abbreviate 
\eqn{abbrevia}
$$
\mathbb{G}_{\eps}(u,B_{r}):=\mathcal{G}(u, B_{r})+\nr{\hhh(\cdot,u)}_{L^1(B_{r})}+\nr{f}_{n/\alpha, \textnormal{\texttt{l}};(1+\eps)B_{r}}^{p/(p-\alpha)}+1
$$
and denote $\FF(x,y,z):= F(z)+  \gggg(x,y,z)$. We fix a family of radially symmetric mollifiers $\{\phi_{\varepsilon}\}\subset C^{\infty}(\mathbb{R}^n)$, defined in a standard way as $\phi_{\varepsilon}(x):=\phi(x/\varepsilon)/\varepsilon^n$, where $\phi\in C^{\infty}_{c}(\BB)$, $ \nr{\phi}_{L^{1}(\mathbb{R}^n)}=1$, $\mathcal B_{3/4} \subset \supp\,  \phi$. We next define
\eqn{recap}
$$\tilde{u}_{\varepsilon}=u*\phi_{\eps} \,, \quad f_{\varepsilon}(x):=\mint_{\BB} f(x+\eps \lambda)\dla+1\,, \quad \hhh_{0,\varepsilon}(x):=\mint_{\BB} |\hhh(x+\eps \lambda,0)|\dla\,,$$
and
\begin{flalign}
\displaystyle \FF_{\eps, \delta}(x,y,z)  &:=(\FF(x,y,\cdot)*\phi_{\delta})(z)+\sigma_{\varepsilon}[H_{\mu_\delta}(z)]^{q/2}\notag \\ &=\int_{\BB}\FF(x,y,z+\delta \lambda)\phi(\lambda)\dla+\sigma_{\varepsilon}[H_{\mu_\delta}(z)]^{q/2}=: \FF_{\delta}(x,y,z) + \sigma_{\varepsilon}[H_{\mu_\delta}(z)]^{q/2}\,,\label{nuoviF}
\end{flalign}
where 
\begin{flalign}\label{48}
\sigma_{\varepsilon}:=\left(1+\varepsilon^{-1}+\nr{D\tilde{u}_{\varepsilon}}_{L^{q}(B_{r})}^{2q}\right)^{-1}
\Longrightarrow \sigma_{\varepsilon}\int_{B_{r}}[H_{\mu_\delta}(D\tilde{u}_{\varepsilon})]^{q/2}\dx\stackrel{\eps \to 0}{\to} 0\,,
\end{flalign}
uniformly with respect to $\delta$; we recall that $\mu_{\delta}=\mu+\delta \in (0, 2]$. Note that $\{u_{\eps}\}$ is bounded in $W^{1,p}(B_{r})$. 
Finally, we define $\hhh_{\varepsilon}\colon B_{r}\times \mathbb{R}\to \mathbb{R}$ by 
\eqn{defigg}
$$
\hhh_{\varepsilon}(x,y):=\tilde{\hhh}_{\varepsilon}(x,y)+\snr{\tilde{u}_{\eps}(x)-y}^{\alpha}\,, \quad 
\displaystyle \tilde{\hhh}_{\varepsilon}(x,y):=\int_{\BB}\hhh(x+\eps \lambda,y)\phi(\lambda)\dla\,.
$$
From \rif{nuoviF} and Lemma \ref{marclemma} it easily follows that 
\eqn{lippi}
$$|\FF_{ \delta}(x,y,z)
-\FF(x,y,z)|\lesssim \delta (\snr{z}+1)^{q-1}$$
(see also Lemma \ref{converlemma}) so that, since $D\tilde u_{\eps}$ is bounded for every $\eps$, we have
\eqn{jensen00} 
$$\|  \FF_{\delta}(\cdot, \tilde u_{\eps}, D\tilde u_{\eps})  - \FF(\cdot, \tilde u_{\eps}, D\tilde u_{\eps})\|_{L^1(B_{r})} =\texttt{o}_{\eps}(\delta,B_{r})\,,
$$
for every $\eps$.
Up to not relabelled subsequences, we can also assume that 
\eqn{convergenze} 
$$
\begin{cases}
\|D\tilde{u}_{\varepsilon} -Du\|_{L^{p}(B_{r})}=\texttt{o}(\varepsilon,B_{r}) \\
\mbox{$\tilde{u}_{\varepsilon} \to u$  in $L^{\gamma}(B_{r})$ for every $\gamma < p^*$ and a.e.}
\end{cases}
$$
By Remark \ref{jensiremark} below, we have $F(D\tilde u_{\eps}) \to F(Du)$ in $L^1(B_{r})$. On the other hand by \rif{assgg}$_2$,\rif{convergenze} and again Lebesgue domination, it follows  $\gggg(\cdot,\tilde u_{\eps},D\tilde u_{\eps}) \to \gggg(\cdot,  u,Du)$ in $L^1(B_{r})$. We conclude with 
\eqn{jensen0} 
$$
\|\FF(\cdot, \tilde u_{\eps}, D\tilde u_{\eps}) - \FF(\cdot, u, Du)\|_{L^1(B_{r})} =\texttt{o}(\varepsilon,B_{r})
$$
up to subsequences. Next, observe that \rif{asshh} gives
\eqn{crescehh}
$$
|\hhh(x,y)| \leq |\hhh(x, 0)| + L f(x)|y|^\alpha \leq |\hhh(x, 0)| + L [f(x)]^{n/\alpha}+ |y|^{n\alpha/(n-\alpha)}$$
for every $x\in \Omega, y \in \er$.   
This, and \rif{defigg}, easily imply the (nonuniform in $\eps$) estimate
\eqn{piccione}
$$
| \tilde \hhh_{\eps}(x,y)|+ | \hhh_{\eps}(x,y)| \leq c_{\eps} (|y|+1). 
$$
Again by the definitions in \rif{defigg}, given any $x\in B_{r}$ and $y_{1},y_{2}\in \er$, we have
\eqn{assge2}
$$
\snr{\hhh_{\varepsilon}(x,y_{1})-\hhh_{\varepsilon}(x,y_{2})}+\snr{\tilde \hhh_{\varepsilon}(x,y_{1})-\tilde \hhh_{\varepsilon}(x,y_{2})}\leq 2f_{\varepsilon} (x)\snr{y_{1}-y_{2}}^{\alpha} 
\leq 2\nr{f_{\varepsilon}}_{L^\infty(B_{r})}\snr{y_{1}-y_{2}}^{\alpha}.
$$
By \rif{convergenze}$_2$, we now have that 
\eqn{26}
$$
\nr{  \hhh_{\eps}(\cdot,\tilde{u}_{\eps})-\hhh(\cdot,u)}_{L^1(B_{r})}=\nr{\tilde \hhh_{\eps}(\cdot,\tilde{u}_{\eps})-\hhh(\cdot,u)}_{L^1(B_{r})}  =\texttt{o}(\eps,B_{r})\,.
$$
The proof of this fact involves a commutator type estimate on mollifiers, and it is postponed to Lemma \ref{veve} below. We next consider the functional 
\eqn{funzio}
$$
W^{1,q}(B_{r})\ni w\mapsto \mathcal{G}_{\varepsilon,\delta}(w,B_{r}):=\int_{B_{r}}[\FF_{\eps, \delta}(x,w,Dw) +
\hhh_{\eps}(x,w)]\dx,
$$
and note that, using \rif{48},\rif{jensen0},\rif{jensen00} and \rif{26}, we have   
\eqn{conv0}
$$
\mathcal{G}_{\eps, \delta}(\tilde u_{\eps}, B_{r})- \mathcal{G}(u, B_{r})= \texttt{o}(\varepsilon,B_{r}) + \texttt{o}_{\eps}(\delta,B_{r})\,.
$$
It is at this stage worth remarking that the original assumptions \rif{assFF}-\rif{assgg} from Theorem \ref{t1}, imply that $z \mapsto  \FF_{\eps, \delta}(\cdot, z)$ satisfies the following version of \rif{32.2}:
\begin{flalign}\label{32.2bis}
\qquad\begin{cases}
\ \sigma_{\eps}[H_{\mu_{\delta}}(z)]^{q/2}+\tilde \nu [H_{\mu_{\delta}}(z)]^{p/2}\le \FF_{\eps, \delta}(\cdot, z)\le  \tilde L[H_{\mu_{\delta}}(z)]^{q/2}+ \tilde L[H_{\mu_{\delta}}(z)]^{p/2}\\
\    \nu_0(\eps)[H_{\mu_{\delta}}(z)]^{(q-2)/2}\snr{\xi}^{2}+\tilde \nu [H_{\mu_{\delta}}(z)]^{(p-2)/2}\snr{\xi}^{2}\le \partial_{zz}\FF_{\eps, \delta}(\cdot, z)\xi\cdot \xi \\
\ \snr{\partial_{zz}\FF_{\eps, \delta}(\cdot, z)}\le  \tilde L[H_{\mu_{\delta}}(z)]^{(q-2)/2}+ \tilde L[H_{\mu_{\delta}}(z)]^{(p-2)/2},
\end{cases}\,,
\end{flalign}
for new constants $0< \tilde \nu \leq  \tilde L $ as in Section \ref{basicn}, depending on $\data$, that are independent of $\eps, \delta$; note that $\mu_\delta=\mu +\delta>0$. On the contrary, $\nu_0(\eps) =q\min\{q-1,1\} \sigma(\eps)$ depends on $\eps$. These  properties can be checked using the very definition in \rif{nuoviF}, in conjunction with assumptions \rif{assFF}-\rif{assgg}; see for instance \cite[Section 4.5]{dm1} for more details. Similarly, the function $\gggg_{\delta}(\cdot)$ satisfies \rif{assgg2} with $\mu\equiv \mu_{\delta}$, and a suitable $\tilde L$, as it happens in \rif{32.2bis}  (eventually we enlarge $\tilde L$ to fit both \rif{32.2} and \rif{assgg2}). Finally, $\hhh_{\eps}(\cdot)$ satisfies a version of \rif{asshh22} thanks to \rif{piccione}-\rif{assge2}, with $f(\cdot)$ replaced by $f_{\eps}(\cdot)$. In particular, the integrand $z \mapsto \FF_{\eps, \delta}(\cdot, z)$ is convex with $q$-polynomial growth; keeping also in mind \rif{piccione}, we conclude we can apply Direct Methods (lower semicontinuity $\oplus$ coercivity) to define 
$u_{\varepsilon, \delta}\in \tilde{u}_{\varepsilon}+W^{1,q}_{0}(B_{r})$ as a solution to
\eqn{solutio}
$$
u_{\eps, \delta}\mapsto \min_{w\in \tilde{u}_{\varepsilon}+W^{1,q}_{0}(B_{r})}\, \mathcal{G}_{\eps, \delta}(w, B_{r})\,.
$$
Indeed, standard coercivity estimates - see for instance \cite[Pages 986-987]{BM} or directly \rif{coerc0} below - give that
$$
\|Dw\|_{L^q(B_{r})}^q \lesssim \mathcal{G}_{\eps, \delta}(w, B_{r}) +  \|f_{\eps}\|_{L^{n/\alpha}(B_{r})}^{q/(q-\alpha)}+ \|D\tilde u_\eps\|_{L^{q}(B_{r})}^q + \|\hhh_{\eps}(\cdot, \tilde u_{\eps})\|_{L^{1}(B_{r})}$$
holds for every $w\in \tilde{u}_{\varepsilon}+W^{1,q}_{0}(B_{r})$, where the involved constants also depend on $\eps$. Weak lower semicontinuity follows for instance as in \cite[Chapter 4, Remark 4.1]{giu}. 
To proceed, using the minimality of $u_{\eps, \delta}$, Sobolev and Young's inequalities, we find
\begin{eqnarray}
&&\notag \tilde \nu\nr{Du_{\eps, \delta}}_{L^p(B_{r})}^p+  \sigma_{\eps}\nr{Du_{\eps, \delta}}_{L^q(B_{r})}^q   \stackrel{\eqref{32.2bis}_1}{\le} \int_{B_{r}}\FF_{\eps, \delta}(x, u_{\eps, \delta}, Du_{\eps, \delta})\dx\nonumber \\&& \ \ \, \qquad \leq  \mathcal{G}_{\eps, \delta}(u_{\eps, \delta}, B_{r})+\nr{\hhh_{\eps}(\cdot,u_{\eps, \delta})}_{L^1(B_{r})}\nonumber \\
&&\notag \ \  \, \qquad \leq  \mathcal{G}_{\eps, \delta}(\tilde u_{\eps}, B_{r})+\nr{\hhh(\cdot,u)}_{L^1(B_{r})}\\ && \qquad \qquad +\nr{\hhh_{\eps}(\cdot,\tilde u_{\eps})-\hhh(\cdot, u)}_{L^1(B_{r})}+\nr{\hhh_{\eps}(\cdot,\tilde u_{\eps})-\hhh_{\eps}(\cdot,u_{\eps, \delta})}_{L^1(B_{r})}\nonumber \\
&& \stackrel{\eqref{assge2},\eqref{26}, \eqref{conv0}}{\leq}  \mathcal{G}(u, B_{r})+\texttt{o}(\varepsilon,B_{r}) + \texttt{o}_{\eps}(\delta,B_{r})+\nr{\hhh(\cdot,u)}_{L^1(B_{r})}\nonumber \\
&&\hspace{27mm}+c \nr{f_{\eps}}_{L^{n/\alpha}(B_{r})}\left[\nr{D\tilde{u}_{\eps}}^{\alpha}_{L^{p}(B_{r})}+\nr{Du_{\eps, \delta}}_{L^{p}(B_{r})}^{\alpha}\right]\nonumber \\
&& \quad  \stackrel{\eqref{assFF}_2, \eqref{convergenze}_1}{\le}  \frac{\tilde \nu}{2}\nr{Du_{\eps, \delta}}_{L^p(B_{r})}^p +c \mathcal{G}(u, B_{r})+c\nr{\hhh(\cdot,u)}_{L^1(B_{r})}\nonumber \\
&&\hspace{25mm} +c\nr{f+1}_{L^{n/\alpha}((1+\eps)B_{r})}^{p/(p-\alpha)} +\texttt{o}(\varepsilon,B_{r}) + \texttt{o}_{\eps}(\delta,B_{r})\,.\label{coerc0}
\end{eqnarray}
Recalling the notation fixed in \rif{abbrevia}, we therefore conclude with
\eqn{conclude}
$$
\nr{Du_{\eps, \delta}}_{L^p(B_{r})}^p+  \sigma_{\eps}\nr{Du_{\eps, \delta}}_{L^q(B_{r})}^q \leq c  \mathbb{G}_{\eps}(u, B_{r}) +\texttt{o}(\varepsilon,B_{r}) + \texttt{o}_{\eps}(\delta,B_{r})\,,$$
where $c \equiv c(\data)$. In \rif{conclude} we have used the embedding 
$
\nr{f}_{L^{n/\alpha}(B_{(1+\eps)r})}\lesssim \|f\|_{n/\alpha, \textnormal{\texttt{l}};B_{(1+\eps)r}}
$, since $\texttt{l} < n/\alpha$ (recall \rif{lorentzbasic}$_{2,3}$ and \rif{lo}).
\begin{remark} \label{jensiremark} \emph{In the proof of \rif{jensen0}  we have used a standard argument that it is better to recall. By Jensen's inequality we find 
$F(D\tilde u_{\eps}) \leq F(Du)*\phi_{\varepsilon}$ on $B_{r}$, therefore, as $F(Du)\in L^1(B_{r})$, a well-known variant of Lebesgue dominated convergence gives that $F(D\tilde u_{\eps}) \to F(Du)$ in $L^1(B_{r})$. }
\end{remark}

\subsection{Proof of \rif{stima1} and \rif{stima11}, and proof of Theorem \ref{t5} concluded} \label{generalapp2} We start by \rif{stima11}. 
By \rif{32.2bis} and the subsequent discussion, we are in the setting of Sections \ref{laprima} and \ref{appi}. Applying estimate \rif{stima1pre} to $u_{\eps, \delta}$, and using it in conjunction to \rif{conclude}, and yet letting $s\to r$, we find
\eqn{stim1}
$$
\nr{Du_{\eps, \delta}}_{L^{\infty}(B_{t})}\le \frac{c}{(r-t)^{\chi_1}}\left[
\mathbb{G}_{\eps}(u, B_{r})\right]^{\chi_2}+ \frac{\texttt{o}(\varepsilon,B_{r}) + \texttt{o}_{\eps}(\delta,B_{r})}{(r-t)^{\chi_1}}\,,
$$
for a new constant $c\equiv c(\data)$, new exponents $\chi_1, \chi_2\equiv  \chi_1, \chi_2(\datae) $, and whenever $0<t<r$.
All in all, we have proved that for every $0<t<r$ and for every $\eps$, there exists a constant $M(t, \eps)$ such that $\|Du_{\eps,\delta}\|_{L^{\infty}(B_t)}\leq M(t, \eps)$ holds for every $\delta$. By \rif{lippi} we conclude that
\eqn{infla}
$$\|  \FF_{ \delta}(\cdot, u_{\eps,\delta},  Du_{\eps,\delta}) - \FF(\cdot, u_{\eps,\delta}, Du_{\eps,\delta})\|_{L^1(B_t)} = \texttt{o}_{\eps}(\delta, B_t)\
$$
holds for every $\eps$. 
By \rif{conclude} it follows that the sequence $\{u_{\eps, \delta}\}_{\delta}$ is bounded in $W^{1,q}(B_{r})$ for every fixed $\eps$. By \rif{stim1}, passing to subsequences and using standard diagonal arguments, we find that for every $\eps$ there exists $u_{\eps}\in \tilde u_{\eps}+W^{1,q}_{0}(B_{r})$ such that, as $\delta \to 0$, it holds that
\eqn{convyconvy0}
$$
\begin{cases}
\, u_{\eps, \delta}\rightharpoonup u_{\eps} \ \mbox{in}\ W^{1,q}(B_{r})\\
\, u_{\eps, \delta}\rightharpoonup^* u_{\eps} \ \mbox{in $W^{1,\infty}(B_{t})$ for every $t<r$} \\
\, u_{\eps, \delta}\to u_{\eps} \ \mbox{strongly in $L^{n/(n-1)}(B_{r})$ and a.e.}\\
\, \tilde \hhh_{\eps}(\cdot, u_{\eps, \delta}) \to \tilde \hhh_{\eps}(\cdot, u_{\eps})\ \mbox{strongly in $L^1(B_{r})$ \,.}
\end{cases}
$$
Note that \rif{convyconvy0}$_4$ follows from \rif{convyconvy0}$_3$, a well-known variant of Lebesgue dominated convergence and \rif{piccione}. By convexity and non-negativity of $z \mapsto \FF(\cdot, z)$, we can use standard weak lower semicontinuity theorems (see \cite[Theorem 4.5]{giu}) to deduce that 
\eqn{wesemi}
$$
\int_{B_t} \FF(x, u_{\eps}, Du_{\eps })\dx
 \leq \liminf_{\delta}\int_{B_t}  \FF(x, u_{\eps, \delta}, Du_{\eps, \delta})\dx
$$
holds whenever $t \leq r$. 
Next, letting $\delta \to 0$ in \rif{conclude} and \rif{stim1} allows to get a uniform bounds (with respect to $\eps$) on $\nr{Du_{\eps}}_{L^p(B_{r})}$ and  $\nr{Du_{\eps}}_{L^\infty(B_{t})}$ for every $t<r$, respectively, so that, again up to not relabelled subsequences as for \rif{convyconvy0}, we can assume that there exists $v \in u+W^{1,p}_0(B_{r})$ such that
\eqn{convyconvy}
$$
\begin{cases}
\, u_{\eps}\rightharpoonup v \ \mbox{in}\ W^{1,p}(B_{r})\\
\, u_{\eps}\rightharpoonup^* v  \ \mbox{in $W^{1,\infty}(B_{t})$ for every $t<r$} \\
\, u_{\eps }\to v \ \mbox{strongly in $L^{\gamma}(B_{r})$ for every $\gamma <p^*$, and a.e.}\\
\, \hhh(\cdot,u_{\eps})\to \hhh(\cdot,v) \ \mbox{strongly in $L^1(B_{r})$.}
\end{cases}
$$
Note that \rif{convyconvy}$_4$ follows from \rif{crescehh} and \rif{convyconvy}$_3$ via Lebesgue domination. 
Moreover, we have 
\eqn{ancora}
$$ \nr{\tilde \hhh_{\eps}(\cdot,u_{\eps})-\hhh(\cdot,u_{\eps}) }_{L^1(B_{r})}=\texttt{o}(\varepsilon,B_{r})\,,$$
the proof of which is again postponed to Lemma \ref{veve} below. Then observe that 
\begin{flalign*}
 &\liminf_{\delta} \mathcal{G}_{\eps, \delta}(u_{\eps, \delta}, B_{r})  \stackrel{\rif{convyconvy0}_{3,4} }{\geq} \liminf_{\delta} \int_{B_{t}} \mathbb{F}_{\delta}(x,u_{\eps, \delta}, Du_{\eps, \delta}) \dx+ \int_{B_{r}} [\tilde \hhh_{\eps}(x,u_{\eps})+\snr{\tilde{u}_{\eps}(x)-u_{\eps}}^{\alpha}]\dx\\
& \quad   \stackrel{\rif{ancora}}{\geq}  \liminf_{\delta} \int_{B_{t}} \mathbb{F}_{\delta}(x,u_{\eps, \delta}, Du_{\eps, \delta}) \dx  + \int_{B_{r}} [ \hhh(x,u_{\eps})+\snr{\tilde{u}_{\eps}(x)-u_{\eps}}^{\alpha}]\dx+\texttt{o}(\varepsilon,B_{r})\\
& \quad \quad  \ \ \,  \stackrel{\rif{infla},\rif{wesemi}}{\geq}\int_{B_t} \FF(x, u_{\eps}, Du_{\eps })\dx
+\int_{B_{r}} [\hhh(x, u_{\eps})+\snr{\tilde{u}_{\eps}(x)-u_{\eps}}^{\alpha}]\dx+\texttt{o}(\varepsilon,B_{r})\\
& \qquad \qquad \ \   =\mathcal{G}(u_{\eps},B_{t})+
\int_{B_{r}}\snr{\tilde{u}_{\eps}(x)-u_{\eps}}^{\alpha}\dx+\texttt{o}(\varepsilon,B_{r})
\,.
\end{flalign*} Letting $t \to r$ in the above display gives
\eqn{lower1}
$$
 \mathcal{G}(u_{\eps},B_{r})+
\int_{B_{r}}\snr{\tilde{u}_{\eps}(x)-u_{\eps}}^{\alpha}\dx\leq 
\liminf_{\delta} \mathcal{G}_{\eps, \delta}(u_{\eps, \delta}, B_{r}) +\texttt{o}(\varepsilon,B_{r})\,.
$$
Using, in order: \rif{convyconvy} and lower semicontinuity, \rif{lower1}, the minimality of $u_{\eps,\delta}$, 
\rif{conv0} and the minimality of $u$, we obtain
\begin{flalign}
 \notag \mathcal{G}(v, B_{r}) + \int_{B_{r}}\snr{u-v}^{\alpha} \dx & \leq\liminf_{\eps}
  \mathcal{G}(u_{\eps},B_{r}) + \lim_{\eps}\int_{B_t} \snr{\tilde{u}_{\eps}(x)-u_{\eps}}^{\alpha}\dx\\
\notag &\leq\liminf_{\eps}  \liminf_{\delta} \mathcal{G}_{\eps, \delta}(u_{\eps, \delta}, B_{r})\\
& 
 \leq  \liminf_{\eps}\lim_{\delta}  \mathcal{G}_{\eps, \delta}(\tilde u_{\eps},B_{r})
 = \mathcal{G}(u, B_{r})\leq\mathcal{G}(v, B_{r})\,.\label{comein}
\end{flalign}
We deduce 
$u\equiv v$. Letting first $\delta\to 0$, and then $\eps \to 0$ in \rif{stim1}, leads to
\rif{stima11}, that leads, via a standard covering argument, 
to the local Lipschitz continuity of $u$. Finally, we observe that the same arguments developed here and in the preceding Section \ref{generalapp}, apply verbatim to the setting of Theorem \ref{t1}, and in this case one replaces $\|f\|_{n/\alpha, \textnormal{\texttt{l}};B_{r}}$ by $\|f\|_{n/\alpha,1/2;B_{r}}$ in \rif{stima11}.  \hspace{22mm}$\square$
\vspace{1mm}
The following lemma provides the proof of \rif{26} and \rif{ancora}.
\begin{lemma}\label{veve} Let $\{v_{\eps}\}\subset L^{p}(B_{r})$, $p>1$, be a sequence such that $v_{\eps} \to u \in L^{\gamma}(B_{r})$, for every $\gamma < p^*$. Then 
\eqn{statlemma}
$$
\begin{cases}
\, \nr{\tilde \hhh_{\eps}(\cdot,v_{\eps})-\hhh(\cdot,v_{\eps}) }_{L^1(B_{r})}\stackrel{\eps \to 0}{\to}  0\\
\, \nr{\tilde \hhh_{\eps}(\cdot,v_{\eps})-\hhh(\cdot,u)}_{L^1(B_{r})}\stackrel{\eps \to 0}{\to} 0\,.
\end{cases}
$$
\end{lemma}
\begin{proof}
Using \rif{defigg} and \rif{crescehh}, it follows that
$
| \tilde \hhh_{\eps}(x, y)| \lesssim   \hhh_{0,\varepsilon}(x)+  [f_{\eps}(x)]^{n/\alpha}+ |y|^{n\alpha/(n-\alpha)}
$
holds for every $(x,y) \in B_{r} \times \er$. 
By this and the equintegrability of $\{v_{\eps}\}$ in $L^{n/(n-1)}$ (recall that $p>1$), it is sufficient to prove that \rif{statlemma} with $B_{r}$ replaced by any concentric ball $B_{\varrho}\Subset B_{r}$ and considering the range $\eps < (r-\varrho)/10$.  
Let us first show that 
\eqn{euna}
$$
\lim_{\eps}\, \nr{\tilde \hhh_{\eps}(\cdot,u)-\hhh(\cdot,u) }_{L^1(B_{\rr})} =0\,.
$$
For this we estimate
\begin{flalign}
&
\notag \nr{\tilde \hhh_{\eps}(\cdot,u)-\hhh(\cdot,u) }_{L^1(B_\varrho)} \\
& \qquad \leq \nr{\tilde \hhh_{\eps}(\cdot,u)-[ \hhh(\cdot,u)]*\phi_{\eps}}_{L^1(B_\varrho)}+\nr{[ \hhh(\cdot,u)]*\phi_{\eps}-\hhh(\cdot,u) }_{L^1(B_\varrho)}\,.\label{primastim}
\end{flalign}
The last term in \rif{primastim} goes to zero as $\eps  \to 0$, by basic properties of convolutions (note that \rif{crescehh} guarantees that $\hhh(\cdot,u) \in L^1(B_{r})$). For the first one, for every $k \in \en$ and $x \in B_{\varrho}$ we have
\begin{flalign*}
&\snr{ \tilde \hhh_{\eps}(x,u(x))-[[ \hhh(\cdot,u)]*\phi_{\eps}](x)} \leq \int_{\BB} f(x+\eps \lambda)|u(x+\eps \lambda)-u(x)|^{\alpha}\phi(\lambda)\dla\\
& \qquad \qquad \qquad \quad \leq \frac1{k} \int_{\BB}[f(x+\eps \lambda)]^{n/\alpha}\phi(\lambda) + c(k)  \int_{\BB}|u(x+\eps \lambda)-u(x)|^{\frac{n \alpha}{n-\alpha}}\phi(\lambda)\dla\,.
\end{flalign*}
Integrating the previous inequality over $B_{\varrho}$ and using Fubini, we obtain
$$
\nr{\tilde \hhh_{\eps}(\cdot,u)-[ \hhh(\cdot,u)]*\phi_{\eps}}_{L^1(B_\varrho)} \leq \frac{c}{k} \|f\|_{L^{n/\alpha}(B_{r})}^{n/\alpha} + c(k) \texttt{o}(\eps, B_{\varrho})\,.
$$
Using this in \rif{primastim}, letting first $\eps \to 0$ and then $k\to \infty$, yields \rif{euna}. For \rif{statlemma}$_1$ we estimate
\begin{flalign}
\notag \nr{\tilde \hhh_{\eps}(\cdot,v_{\eps})-\hhh(\cdot,v_{\eps}) }_{L^1(B_\varrho)} & \leq \nr{\tilde \hhh_{\eps}(\cdot,v_{\eps})-\tilde \hhh_{\eps}(\cdot,u) }_{L^1(B_\varrho)}
+\nr{\tilde \hhh_{\eps}(\cdot,u)-\hhh(\cdot,u) }_{L^1(B_\varrho)}\\
& \qquad  + \nr{\hhh(\cdot,u)-\hhh(\cdot,v_{\eps}) }_{L^1(B_\varrho)}\,.\label{edue}
\end{flalign}
Then, using \rif{assge2} and Fubini, we get
$$
\nr{\tilde \hhh_{\eps}(\cdot,v_{\eps})- \tilde \hhh_{\eps}(\cdot,u)}_{L^1(B_\varrho)} \leq c\nr{f}_{L^{n/\alpha}(B_{r})}\nr{v_{\varepsilon}-u}_{L^{\frac{n\alpha}{n-\alpha}}(B_{r})}^{\alpha}\stackrel{\eps \to 0}{\to} 0\,,
$$
and \eqref{asshh} gives
\eqn{treated}
$$
\nr{ \hhh (\cdot,u)-\hhh (\cdot,v_{\eps})}_{L^1(B_\varrho)}  \leq c\nr{f}_{L^{n/\alpha}(B_{r})}\nr{v_{\varepsilon}-u}_{L^{\frac{n\alpha}{n-\alpha}}(B_{r})}^{\alpha}\stackrel{\eps \to 0}{\to} 0\,.
$$
Using the content of the last two displays and \rif{euna} in \rif{edue} yields \rif{statlemma}$_1$. As for \rif{statlemma}$_2$, this follows via \rif{statlemma}$_1$ and  \rif{treated} via triangle inequality. 
\end{proof}

\subsection{Gradient H\"older continuity and proof of Theorem \ref{t1} concluded}
\label{holdergrad} Here we assume that $f \in L^{\mathfrak{q}}$ and $q > n/\alpha$ and prove that $Du$ is locally H\"older continuous. This completes the proof of Theorem \ref{t1}. Once we know that the gradient of minimizers is locally bounded, we can use more standard perturbation methods to prove its local H\"older continuity. Still we have to be careful at several points, since the boundedness of the gradient cannot be directly used, but must be transferred with an explicit control on the $L^\infty$-norm. For this, we reconsider some of the arguments explained in Sections \ref{hyb0}-\ref{hyb2} and use the same approximation described in Section \ref{generalapp}. First of all, as we are going to prove a local result, there is no loss of generality in assuming that $\mathcal G(u, \Omega)<\infty$, $h(\cdot, u)\in L^1(\Omega)$ and $f \in L^{\mathfrak {q}}(\Omega)$. Fix an open subset $\Omega_0 \Subset \Omega$, set  $r :=\min\{\dist(\Omega_0, \partial \Omega)/4, 1\}$; choose $B_{r}\equiv B_{r}(x_{\rm c})$ with $x_{\rm c}\in \Omega_0$ (that implies that $B_{r} \Subset \Omega$) and use \rif{stim1} to obtain, for any concentric $B_{\tau}\subset B_{r}$ with $\tau \leq r/2$
\begin{flalign}
\notag\nr{Du_{\eps, \delta}}_{L^{\infty}(B_{\tau})}&\le cr^{-\chi_1}\left[
\mathcal{G}(u, \Omega) +c\nr{\hhh(\cdot,u)}_{L^1(\Omega)}+\|f\|_{L^{\mathfrak{q}}(\Omega)} +1\right]^{\chi_2}\\
&\notag  \qquad  + cr^{-\chi_1}[\texttt{o}(\varepsilon,B_{r}) + \texttt{o}_{\eps}(\delta,B_{r})]\label{stim1bis}
\\ & =:M+ cr^{-\chi_1}[\texttt{o}(\varepsilon,B_{r}) + \texttt{o}_{\eps}(\delta,B_{r})]=: M_{B_{r}}(\eps, \delta)\geq 1 \,,
\end{flalign}
where $c\equiv c (\data, |\Omega|)$. Note that the constant $M$ does not depend on the ball chosen $B_{r}$, but only on the radius $r$, that is, on $\dist (\Omega_0, \partial \Omega)$. 
Recalling the meaning of $\texttt{o}(\varepsilon,B_{r})$ and $\texttt{o}_{\eps}(\delta,B_{r})$ introduced at the beginning of Section \ref{generalapp}, letting first $\delta \to 0$ (keeping $\eps$ fixed), and then $\eps \to 0$, we get that $M_{B_{r}}(\eps, \delta) \to M$ in the sense that 
\eqn{doppiaconv}
$$
M_{B_{r}}(\eps, \delta)  \stackrel{\delta \to 0}{\to} M+cr^{-\chi_1}\texttt{o}(\varepsilon,B_{r}) 
 \stackrel{\eps \to 0}{\to}M\,.
$$
We define $v \in u_{\eps, \delta}+W^{1,q}_0(B_{\tau})$ as the solution to
\begin{flalign}\label{pd-dopo}
 v\mapsto \min_{w \in u_{\eps, \delta}+W^{1,q}_0(B_{\tau})} \int_{B_{\tau}} F_{0}(Dw)  \dx\,, \quad 
 F_{0}(z) :=\mathbb{F}_{\eps, \delta}(x_{\rm c}, (u_{\eps, \delta})_{B_{\tau}}, z)
\end{flalign}
so that, by \rif{nuoviF} it is $\mathbb{F}_{\eps, \delta}(x_{\rm c}, (u_{\eps, \delta})_{B_{\tau}}, z)= F_{\eps, \delta}(z)  +  \gggg_{\delta}(x_{\rm c}, (u_{\eps, \delta})_{B_{\tau}}, z)$. 
Recall that $F_0(\cdot)$ satisfies \rif{32.2bis}, and that $\gggg_{\delta}(\cdot)$ satisfies \rif{assgg2} with $\mu\equiv \mu_{\delta}>0$ (see the discussion after \rif{32.2bis}). Three implied features of such definitions are: 
\begin{itemize}
\item The integrand $F_0(\cdot)$ is exactly of the form considered in \rif{congela} (without scaling, therefore   $\varrho =1$) and we can use some of the estimates developed in Section \ref{hyb0}. 
\item Applying Lemma \ref{marcth}, estimate \rif{stimaprimissima2}, and then minimality of $v$ exactly as in \rif{globale0}-\rif{globale}, and finally recalling \rif{stim1bis}, we gain an analog of \rif{fuckbsc}, i.e., 
\eqn{altrasup}
$$
\nr{Dv}_{L^{\infty}(B_{\tau/2})}  \leq \tilde c[M_{B_{r}}(\eps, \delta)]^\ssf\,, \qquad \tilde c \equiv \tilde c (n,p,q,\nu, L) \,.
$$
\item Let us fix a constant $\mathfrak{M}\geq 1$. The integrand $F_0(\cdot)$ is $C^2$-regular and satisfies 
\begin{flalign}\label{F_0part}
\qquad\begin{cases}
\, \tilde \nu [H_{\mu_\delta}(z)]^{p/2} \leq   F_0(z)\leq  c \mathfrak{M}^{q-p}[H_{\mu_\delta}(z)]^{p/2}\\
\,  \tilde \nu [H_{\mu_\delta}(z)]^{(p-2)/2}\snr{\xi}^{2}\le \partial_{zz}  F_0(z)\xi\cdot \xi  \\
\, \snr{\partial_{zz}  F_0(z)}\le c \mathfrak{M}^{q-p}[H_{\mu_\delta}(z)]^{(p-2)/2} 
\end{cases}
\end{flalign}
whenever $z,\xi \in \er^n$ with $\snr{z}\leq \mathfrak{M}$, where $c\equiv c (n,p,q,\tilde L)\equiv c (n,p,q,L)$. 
\end{itemize} 
We are now in position to apply Lemma \ref{manflemma} below to $v$ in the ball $B_{\tau/2}$, with $\mathfrak{M}\equiv \tilde c [M_{B_{r}}(\eps, \delta)]^{\mathfrak{s}}$, as appearing in \rif{altrasup}. We conclude with
\eqn{holly1}
$$
\mint_{B_\varrho} |Dv-(Dv)_{B_\varrho}|^p \dx \leq c_{\eps, \delta} \left(\frac{\varrho}{\tau}\right)^{p\beta_{\eps, \delta}} $$
that holds whenever $B_{\varrho}\subset B_{\tau/2}$ are concentric balls. Here $c_{\eps, \delta}\equiv c_{\eps, \delta} (n,p,q,\nu, L, M_{B_{r}}(\eps, \delta)):= c_{\rm h}(\mathfrak{M})\geq 1$ and $\beta_{\eps, \delta}\equiv \beta_{\eps, \delta} (n,p,q,\nu, L,, M_{B_{r}}(\eps, \delta) ):= \beta_{\rm h}(\mathfrak{M})\in (0,1)$ are non-decreasing and non-increasing functions of their last argument, respectively. This follows by the monotonicity properties of the functions $ c_{\rm h}(\cdot)$ and $ \beta_{\rm h}(\cdot)$
 asserted in Lemma \ref{manflemma}. By \rif{doppiaconv}, we get, up to not relabelled subsequences
\eqn{doppiaconv2}
$$
c_{\eps, \delta}  \stackrel{\delta \to 0}{\to} c_{\eps }  \stackrel{\eps \to 0}{\to} c_{\rm l}\equiv 
c_{\rm l}(\data,M)
 < \infty\quad \mbox{and} \quad  \beta_{\eps, \delta}  \stackrel{\delta \to 0}{\to} \beta_{\eps }  \stackrel{\eps \to 0}{\to} \beta_{\rm l} \equiv \beta_{\rm l}(\data,M)>0
$$
in the same sense of \rif{doppiaconv}. From now on we continue to denote by $c_{\eps, \delta}$ a double sequence of constants, depending in the most general case on $\data$, $\mathfrak{q}$ and $M_{B_{r}}(\eps, \delta)$, such that \rif{doppiaconv2} takes place. The exact value of $c_{\eps, \delta}$ might change on different occurrences, but still keeping the property in \rif{doppiaconv2}, as it will be clear by the way they will be determined, starting from \rif{doppiaconv} and \rif{doppiaconv2}. To proceed, 
we argue as for the proof of \rif{37bbb}; there replace $8|h|^{\beta_0}\equiv \tau$ and $f_{\rr}$ by $f_{\eps}$ and $\varrho$ by $1$, in order to adapt it to the present setting.  In particular \rif{assgg2}-\rif{asshh22} must be used instead of \rif{assaar}-\rif{assgr}. We find
\begin{flalign}
\notag\mint_{B_{\tau}}\snr{V(Du_{\eps, \delta})-V(Dv)}^{2}  \dx  & \le c \tau^{ \alpha}[M_{B_{r}}(\eps, \delta)]^{\alpha + \gamma q/p}+ c \tau^{ \alpha}[M_{B_{r}}(\eps, \delta)]^\alpha
\mint_{B_{\tau}}\snr{f_\eps}  \dx\\ & \leq c_{\eps, \delta}  \tau^{\alpha}+ c_{\eps, \delta}    \tau^{\alpha-n/\mathfrak{q}}\|f+1\|_{L^{\mathfrak{q}}(B_{r})}\leq  c_{\eps, \delta}  \tau^{\alpha-n/\mathfrak{q}}\,.\label{holly11}
\end{flalign}
Setting
$
\sigma := \min\{1, p/2\}(\alpha-n/\mathfrak{q})>0$, by \rif{sopradue} and H\"older's inequality, we further obtain
\eqn{holly2}
$$
\mint_{B_\tau}\snr{Du_{\eps, \delta}-Dv}^{p}  \dx   \leq c_{\eps, \delta}  \tau^{\alpha-n/\mathfrak{q}}
+c_{\eps, \delta}\una[M_{B_{r}}(\eps, \delta)]^{p(2-p)/2} \tau^{p (\alpha-n/\mathfrak{q})/2}\leq c_{\eps, \delta} \tau^{\sigma} .
$$ 
Estimates \rif{holly1} and \rif{holly2} are the two basic ingredients to apply a standard comparison argument implying local H\"older continuity. For this see for instance \cite[pp. 43-45]{manth1} and \cite[Proof of Theorem 2.2]{AM}. Specifically, we first find
$$
\mint_{B_\varrho} |Du_{\eps, \delta}-(Du_{\eps, \delta})_{B_\varrho}|^p \dx
\leq  c_{\eps, \delta} \left(\frac{\varrho}{ \tau}\right)^{p\beta_{\eps, \delta}} + c_{\eps, \delta} 
\left(\frac \tau \varrho \right)^n \tau^{\sigma} \,,
$$
and this holds whenever $\varrho \leq \tau/2$. 
By taking $\varrho=  (\tau/2)^{1+\sigma/(n+p\beta_{\eps, \delta})}$ we arrive at 
$$
 \mint_{B_\varrho}|Du_{\eps, \delta}-(Du_{\eps, \delta})_{B_\varrho}|^p   \dx
\leq  c_{\eps, \delta} \varrho^{\frac{p\sigma\beta_{\eps, \delta}}{n+\sigma+p\beta_{\eps, \delta}}}\,,\quad  \mbox{for every $\rr \leq (r/4)^{1+\sigma/(n+p\beta_{\eps, \delta})}$}\,.
$$
Letting first $\delta\to 0$ and then $\eps \to 0 $ in the above display, and recalling \rif{convyconvy0},\rif{convyconvy} and \rif{doppiaconv2}, we finally conclude with
\eqn{holdsfor}
$$
 \mint_{B_\varrho}|Du-(Du)_{B_\varrho}|^p  \dx
\leq  c \varrho^{\frac{p\sigma\beta_{\rm l}}{n+\sigma+p\beta_{\rm l}}}\,,\quad  \mbox{for every $\rr \leq (r/4)^{1+\sigma/(n+p\beta_{\rm l})}$} \,,
$$ 
where $c \equiv c (\data, M )$ and the exponent $\beta_{\rm l}\equiv \beta_{\rm l}(\data, M)\in (0,1)$ is defined in \rif{doppiaconv2}.  Summarizing,  we have proved that \rif{holdsfor} holds whenever $B_\varrho$ is centered in $\Omega_0$, and where $r =\min\{ \dist (\Omega_0, \partial \Omega)/4, 1\}$. Since $\Omega_0$ is arbitrary, Campanato-Meyers integral characterization of H\"older continuity yields that for every open subset $\Omega_0 \Subset \Omega$ it holds that 
\eqn{localeh}
$$
[Du]_{0,\alpha_*; \Omega_0} \leq c\,,  \quad \alpha_* := \frac{\sigma\beta_{\rm l}}{n+\sigma+p\beta_{\rm l}}\,,
$$
where $c\equiv c(\data,\mathcal{G}(u, \Omega), \nr{\hhh(\cdot,u)}_{L^1(\Omega)},\nr{f}_{L^{\mathfrak{q}}(\Omega)}, \dist(\Omega_0, \partial \Omega))$. The proof is complete.\hspace{10mm}$\square$
\begin{lemma}[Theorem 2 from \cite{manth1}, revisited] \label{manflemma}Let $v \in W^{1,q }(B)$, for some ball $B \subset \er^n$, be a weak solution to $\diver\, \partial_z F_0(Dv)=0$ in $B$, such that $\|Dv\|_{L^\infty(B)}\leq \mathfrak{M}$, where $\mathfrak{M}\geq 1$ is a fixed constant and $F_0(\cdot)$ is defined in \eqref{pd-dopo}. There exist two constants $c_{\rm h}(\mathfrak{M})\geq 1$ and $\beta_{\rm h}(\mathfrak{M})\in (0,1]$, depending on $n,p,q,\nu, L$ and $\mathfrak{M}$, but otherwise independent of $\mu$, $\eps, \delta, \nu_0$ and the integrand considered $F_0(\cdot)$, such that
\eqn{oscimanf}
$$
\osc\, (Dv,s B) \leq c_{\rm h}(\mathfrak{M}) s^{\beta_{\rm h}(\mathfrak{M})}
$$ 
holds whenever $s \in (0,1]$. Moreover, the functions $\mathfrak{M} \mapsto c_{\rm h}(\mathfrak{M})$ and $\mathfrak{M} \mapsto \beta_{\rm h}(\mathfrak{M})$ are non-decreasing and non-increasing, respectively. 
\end{lemma}
\begin{proof} There are essentially two crucial remarks here. The first is that, when applying the methods for \cite[Theorem 2]{manth1} to $\diver\, \partial_zF_0(Dv)=0$, we only see $z\equiv Dv$ as arguments of $\partial_zF_0(z)$ and $\partial_{zz}F_0(z)$. Therefore we can argue as conditions \rif{F_0part} hold for every $z \in \er^n$ when dealing with weak solutions as in \cite{manth1}. At this point \rif{oscimanf} follows, with the dependence of $c_{\rm h}(\mathfrak{M}), \beta_{\rm h}(\mathfrak{M})$ described in the statement, by tracking the constants in \cite{manth1}. Specifically, the ratio $\gamma_1/\gamma_0$ appearing in \cite{manth1} is in this setting replaced by $c\mathfrak{M}^{q-p}$, where $c\equiv c (n,p,q,\nu, L)$. The second remark is that here we cannot use any approximation procedure, as in \cite{manth1}, since we have to keep the condition $\|Dv\|_{L^\infty(B)}\leq \mathfrak{M}$. Such an approximation is used in \cite{manth1} to deal with $C^2$-solutions; such higher regularity allows  for certain computations. It is not difficult to see that essentially the only point where this enters is the derivation and the testing of the differentiated equation 
$\diver\, (\partial_{zz}F_0(Dv)DD_sv)=0$, $s \in \{1, \dots, n\}$. On the other hand, in order to carry out the computation of \cite{manth1}, it is sufficient to have $Dv\in W^{1,2}_{\loc}(B,\er^n)\cap L^{\infty}_{\loc}(B, \er^n)$. As for $Dv \in L^{\infty}$, this is assumed in Lemma \ref{manflemma}. As for $Dv \in W^{1,2}_{\loc}$, this comes as in Lemma \ref{cacc-class1}, whose application is allowed as $\mu\equiv \mu_\delta >0$ in the present situation. This allows to avoid the approximation in \cite{manth1}, that was in fact implemented to reduce to the case $\mu>0$ (denoted by  $\eps$ in \cite{manth1}). 
\end{proof}

\section{Theorems \ref{t2}, \ref{t4} and Corollary \ref{c3}}
\begin{proposition}\label{priori3}
Let $u\in W^{1,q}(B_{r})$ be a minimizer of the functional $ \mathcal F_{\texttt{x}}(\cdot, B_{r})$ in \eqref{FMx}, where $B_{r}\Subset \Omega$ and $r\leq 1$, under assumptions \eqref{xx.3} and 
\begin{eqnarray}\label{bound3dopo}
\frac{q}{p}<  1+  \textnormal{\texttt{k}}\frac{\alpha^2}{n^2}\,, \qquad  \mbox{where}\  \textnormal{\texttt{k}} :=  
\begin{cases}  
\,4/9  & \ \mbox{if} \  p\geq 2\\
\,8/33 & \  \mbox{if}  \ 1<p<2\,.
 \end{cases}
\end{eqnarray}
Then
\eqn{stimasharp}
$$
\nr{Du}_{L^{\infty}(B_{t})}\le \frac{c}{(s-t)^{\chi_1}}\left[\|Du\|_{L^p(B_s)} +1\right]^{\chi_2}
$$
holds whenever $B_{t}\Subset B_s\Subset B_{r}$ are concentric balls, with $c\equiv c (\data)$ and $\chi_1, \chi_2\equiv \chi_1, \chi_2(\datae)$. 
\end{proposition}
The setting of Proposition \ref{priori3} is the one of Section \ref{variante2f}.  
We will use the following fact
\eqn{totale}
$$
\frac qp <   1+\frac {\alpha^2} {(C+1/4)n^2}
 \Longrightarrow 
 \ssf < 1 + \frac{\alpha}{Cn} \,,
$$
that holds whenever $C\geq 1$ is a fixed number, where $\ssf$ is defined in \rif{marcexp}.
\begin{remark}\label{regolarissime} {\em The standard regularity theory for nonlinear elliptic equations \cite{manth1} gives that, under the assumptions of Proposition \ref{caccin5} (except \rif{marcexp2}), any $W^{1,q}$-solution $u$  to \rif{eulera} satisfies \rif{33.1bis}. The same applies to the minimizers $u$ considered in Proposition \ref{priori3}, as they are energy solutions to the Euler-Lagrange equation of the functional $ \mathcal{F}_{\texttt{x}}$; see comments after \rif{xx.3}. Such equations are of the type in \rif{eulera} considered in Proposition \ref{caccin5}, by assumptions \rif{xx.3}. Therefore in proving Proposition \ref{priori3} we can assume that $Du$ is locally bounded in $B_{r}$. }
\end{remark}
\subsection{Proposition \ref{priori3}, case $p\geq2$} \label{method} We first discuss the case $\alpha <1$; at the end we give the modifications for the case $\alpha =1$. We take balls $B_{t}\Subset B_s$ as in the statement of Proposition \ref{priori3};  we can assume that $\nr{Du}_{L^{\infty}(B_{t})}\ge 1$. Next, we consider further concentric balls $B_{t}\Subset B_{\tau_{1}}\Subset B_{\tau_{2}}\Subset B_{s}$, and a generic point $x_{0} \in B_{\tau_{1}}$. 
By \eqref{xx.10}, with $M\equiv \nr{Du}_{L^{\infty}(B_{\tau_{2}})}$, we apply Lemma \ref{revlem} on $B_{r_{0}}(x_0)\equiv B_{(\tau_1-\tau_1)/8}(x_0)$ with $h\equiv 1$, $\kappa_0\equiv 0$, $v \equiv E_{\mu}(Du)$, $f_1\equiv \snr{Du}+1$, $M_{0}\equiv M^{\ssf(q-p)/2}$, $M_{1}\equiv M^{(\ssf q+p-\mathrm{b})/2}$, $t \equiv2$, $\delta_{1}\equiv \alpha$, $m_1\equiv 2q-2p+\mathrm{b}$, $\theta_1 \equiv 1$. We obtain 
\begin{flalign}
\notag E (Du(x_{0}))&\le c\nr{Du}_{L^{\infty}(B_{\tau_{2}})}^{\frac{\ssf(q-p)\chi}{2(\chi-1)}}\left(\mint_{B_{r_{0}}(x_{0})}[E_{\mu}(Du)]^{2} \dx\right)^{1/2}\\
& \qquad +c\nr{Du}_{L^{\infty}(B_{\tau_{2}})}^{\frac{\ssf (q-p)}{2(\chi-1)}+\frac{\ssf q+p-\mathrm{b}}{2}}\mathbf{P}_{2,\alpha}^{2q-2p+\mathrm{b},1}(\snr{Du}+1;x_{0},(\tau_2-\tau_1)/4)\,,\label{xx.1500}
\end{flalign}
where $\chi\equiv \chi(\beta)= n/(n-2\beta)$, for every $\beta< \alpha/(1+\alpha)$ and $c\equiv c(\data, \beta)$. Using \rif{estimateas}, after a few manipulations we conclude with
\begin{flalign}\label{xx.15}
\nr{Du}_{L^{\infty}(B_{\tau_{1}})}&\le\frac{c}{(\tau_{1}-\tau_{2})^{\frac{n}{2p}}}
\nr{Du}_{L^{\infty}(B_{\tau_{2}})}^{\left(\frac qp-1\right)\frac{\ssf \chi}{2(\chi-1)}+\frac 12}\left(\int_{B_{s}}(\snr{Du}+1)^p  \dx\right)^{\frac{1}{2p}}\nonumber \\
&\quad +c\nr{Du}_{L^{\infty}(B_{\tau_{2}})}^{\left(\frac qp-1\right)\frac{\ssf \chi}{2(\chi-1)}+\frac{\ssf+1}{2}-\frac{\mathrm{b}}{2p}}\nr{\mathbf{P}^{2q-2p+\mathrm{b},1}_{2,\alpha}(\snr{Du}+1;\cdot,(\tau_{2}-\tau_{1})/4)}_{L^{\infty}(B_{\tau_{1}})}^{1/p}+c \,,
\end{flalign}
for every ${\rm b}\in [0,p]$. 
We now want to show that there exists $\beta < \alpha/(1+\alpha)$ such that 
\eqn{csisi1}
$$
 \left(\frac qp-1\right)\frac{\ssf\chi }{2(\chi-1)}+\frac{\ssf+1}{2}-\frac{\mathrm{b}}{2p} =
  \left(\frac qp-1\right)\frac{\ssf n}{4\beta}+\frac{\ssf+1}{2}-\frac{\mathrm{b}}{2p}
 < 1
$$
and the $L^\infty$-norm of $\mathbf{P}^{2q-2p+\mathrm{b},1}_{2,\alpha}$ appearing in \rif{xx.15} can be estimated by the $L^p$-norm of $Du$. In terms of $q/p$, condition \rif{csisi1} translates into
\eqn{csisi22}
$$
\frac qp < 1 + \left(1-\ssf+\frac{{\rm b}}{p}\right) \frac{2\beta}{\ssf n}\,.
$$
For the $\mathbf{P}^{2q-2p+\mathrm{b},1}_{2,\alpha}$-term, we start checking that condition \rif{lo.2.1} is satisfied, that is
\eqn{assumiche}
$$
\frac{n\theta }{t\delta} \equiv  \frac{n}{2\alpha}>1\,.
$$
This is the point where we use that $\alpha<1$, as the above quantity turns out to be equal to one when $\alpha=1$ and $n=2$. As mentioned above, the case $\alpha =1$ will be treated later. Now we get an  $L^\infty$-bound for $\mathbf{P}^{2q-2p+\mathrm{b},1}$. First, thanks to \rif{assumiche}, we apply \rif{11} to get
\eqn{procp}
$$
 \nr{\mathbf{P}^{2q-2p+\mathrm{b},1}_{2,\alpha}(\snr{Du}+1;\cdot,(\tau_{2}-\tau_{1})/4)}_{L^{\infty}(B_{\tau_{1}})}
\leq c\nr{\snr{Du}+1}_{L^{p}(B_s)}^{q-p+\mathrm{b}/2} 
$$
where $c \equiv c(n,p,q, \alpha)$, and provided
\eqn{csisi3}
$$\frac{mn\theta}{t\delta}\equiv \frac{n(2q-2p+\mathrm{b})}{2\alpha}  <p\Longleftrightarrow \frac qp < 1+ \frac{\alpha}{n} -\frac{\mathrm{b}}{2p}\,.$$
This provides a second condition on $(p,q)$, the one in \rif{csisi22} being the first. As specified in the statement of Proposition \ref{caccin4}, \rif{xx.15} and \rif{procp} remain valid for any choice of $\mathrm{b} \in [0, p]$ satisfying \rif{csisi3}. 
We optimize ${\rm b}$ matching both \rif{csisi22} and \rif{csisi3}; this leads to equalize the two right-hand sides in such inequalities and therefore to the choice
\eqn{ilb}
$$
\mathrm{b}\equiv \mathrm{b}(\beta)=\frac{2p[\ssf \alpha  +2\beta(\ssf-1)]}{\ssf n+4\beta}>0\,.
$$
This is admissible provided
\eqn{ammissibile}
$$
 \mathrm{b}  < \frac{2p\alpha}{n} \Longleftrightarrow \mathfrak{s} < 1 +\frac{2\alpha}{n}
$$
that serves to make the (equal) right-hand sides of \rif{csisi22} and \rif{csisi3} larger than one. Note that \rif{ammissibile} implies $\mathrm{b}  <p$. By \rif{totale}, that we use with $C=2$, note that 
\eqn{notabase}
$$
\trif{bound3dopo} \Longrightarrow \frac qp <   1+\frac {4\alpha^2} {9n^2} \Longrightarrow 
 \ssf < 1 + \frac{\alpha}{2n}\leq \frac 54\,.
$$
We conclude
that the values of $\mathrm{b}$ determined in \rif{ilb} are always admissible in \rif{xx.15}, whenever $q/p$ lies in the range fixed by \rif{bound3dopo}; moreover, $\mathrm{b}$ is a decreasing function of $\beta$ (and indeed the admissibility condition \rif{ammissibile} is independent of $\beta$). Plugging $\mathrm{b}$ from \rif{ilb} in \rif{csisi3} yields
\eqn{bound3veroprepre}
$$
\frac qp < 1 + \frac{2\beta}{n}\left[\frac{2\alpha-n(\ssf -1)}{\ssf n+4\beta}
\right]
$$
and the right-hand side is an increasing function of $\beta$. 
Formally taking the limiting value $\beta =\alpha/(1+\alpha)$ in \rif{bound3veroprepre}, we come up with 
\eqn{bound3vero}
$$
\frac qp < 1 + \frac{2\alpha}{(1+\alpha)n}\left[\frac{2\alpha-n(\ssf -1)}{\ssf n+4\alpha/(1+\alpha)}
\right]=: 1+\mathcal{R}_1(n, p,q,\alpha) \frac{\alpha^2}{n^2} \,.
$$
By \rif{notabase} it turns out that $ \mathcal{R}_1(n,p, q,  \alpha) >  2/3> 4/9=\textnormal{\texttt{k}}$ , so that \rif{bound3vero} is again implied by \rif{bound3dopo}. In conclusion, assuming \rif{bound3dopo} leads to find $\beta < \alpha/(1+\alpha)$ such that \rif{bound3veroprepre} and therefore both \rif{csisi1} and \rif{csisi3} hold with the specific choice of $\mathrm{b}$ made in \rif{ilb}. As a consequence, \rif{procp} holds too and inserting this one in \rif{xx.15}, and applying Young's inequality thanks to \rif{csisi1}, we arrive at
$$
\nr{Du}_{L^{\infty}(B_{\tau_{1}})}\leq \frac{1}{2}\nr{Du}_{L^{\infty}(B_{\tau_{2}})}+\frac{c}{(\tau_2-\tau_1)^{\chi_1}} \left[\|Du\|_{L^p(B_s)}  + 1\right]^{\chi_2} \,,
$$
with $c\equiv c (\data)$ and $\chi_1, \chi_2\equiv \chi_1, \chi_2(n,p,q,\alpha)$. 
Using Lemma \ref{l5} with $h(\tau)\equiv \nr{Du}_{L^{\infty}(B_{\tau})}$ (for this recall that $Du$ is locally bounded in $B_{r}$ by Remark \ref{regolarissime}) leads to \rif{stimasharp}, which is now proved when $\alpha <1$. It remains to fix the case $\alpha =1$, which in view of \rif{assumiche} is not covered only when $n=2$ (otherwise the same proof above works). Observe that \rif{xx.10} still holds replacing $\rr^{2\alpha}$, by $\varrho^{2\tilde \alpha}$ for any $\tilde \alpha <1$. In particular, we take $\tilde \alpha<  1$ such that $q/p< 1+ \textnormal{\texttt{k}}\tilde\alpha^2/n^2$. We can now argue as in the case $\alpha <1$ and the proof of Proposition \ref{priori3} is complete when $p\geq 2$.

\subsection{Proposition \ref{priori3}, case $p<2$} \label{method2}
The proof is similar to that   for the case $p\geq 2$,  but we are going to use \eqref{xx.10-sotto2} instead of \eqref{xx.10}. With $M\equiv \nr{Du}_{L^{\infty}(B_{\tau_{2}})}$, we apply Lemma \ref{revlem} on $B_{r_{0}}(x_0)\equiv B_{(\tau_1-\tau_1)/8}(x_0)$ with $h\equiv 1$, $\kk_{0}\equiv 0$, $v \equiv E_{\mu}(Du)$, $f_1\equiv \snr{Du}+1$, $M_{0}\equiv M^{\ssf(q-p)/2}$, $M_{1}\equiv M^{[(\ssf+1)q-\mathrm{b}/p]/2}$, $t \equiv2$, $\delta_{1}\equiv \alpha/2$, $m_1\equiv q-p+\mathrm{b}$ and $\theta_1 \equiv 1/p$. As $p\geq 2$, we obtain
\begin{flalign}\label{xx.15-sotto2}
\nr{Du}_{L^{\infty}(B_{\tau_{1}})}&\le\frac{c}{(\tau_{1}-\tau_{2})^{\frac{n}{2p}}}
\nr{Du}_{L^{\infty}(B_{\tau_{2}})}^{\left(\frac qp-1\right)\frac{\ssf\chi}{2(\chi-1)}+\frac 12}\left(\int_{B_{s}}(\snr{Du}+1)^p  \dx\right)^{\frac{1}{2p}}\nonumber \\
&\quad +c\nr{Du}_{L^{\infty}(B_{\tau_{2}})}^{\left(\frac qp-1\right)\frac{\ssf\chi}{2(\chi-1)}+\frac{\ssf}{2}+\frac{q-\mathrm{b}/p}{2p}}\nr{\mathbf{P}^{q-p+{\rm b},1/p}_{2,\alpha/2}(\snr{Du}+1;\cdot,(\tau_{2}-\tau_{1})/4)}_{L^{\infty}(B_{\tau_{1}})}^{1/p}+c.
\end{flalign} 
Reabsorbing the $\nr{Du}_{L^{\infty}}$-terms and estimating the $\mathbf{P}^{q-p+{\rm b},1/p}_{2,\alpha/2}$-ones in \rif{xx.15-sotto2} (via \rif{11}), leads to impose the conditions
\eqn{csisi1-sotto2} 
$$
 \left(\frac qp-1\right)\frac{\ssf n }{4\beta}+\frac{\ssf}{2}+\frac{q-\mathrm{b}/p}{2p} < 1 
\qquad \mbox{and} \qquad 
\frac{n(q-p+\mathrm{b})}{p\alpha} < p
 $$
for every ${\rm b}\in [0,p]$. These parallel  \rif{csisi22} and \rif{csisi3}, respectively, from the case $p\geq 2$. The potential $\mathbf{P}^{q-p+{\rm b},1/p}_{2,\alpha/2}$ can be used here as we have $
n\theta/(t\delta) = n/(p\alpha)>1$ as now $p<2$, so that Lemma \ref{crit} applies and yields \eqn{xx.11-sotto2}
$$
\nr{\mathbf{P}^{q-p+{\rm b},1/p}_{2,\alpha/2}(\snr{Du}+1;\cdot,(\tau_{2}-\tau_{1})/4)}_{L^{\infty}(B_{\tau_{1}})}
\leq  
c\nr{Du}_{L^{p}(B_s)}^{\frac{q-p+\mathrm{b}}{2p}}  +c
\,.
$$
Arguing as for the case $p\geq 2$ in Section \ref{method}, relations in \rif{csisi1-sotto2} translate into 
\eqn{rela00}
$$
 \frac qp < 1 + \left(1-\ssf +\frac{\rm b}{p^2}\right)\frac{2\beta}{\ssf n+2\beta}\qquad \mbox{and} \qquad 
\frac qp < 1+\frac{\alpha p}{n}-\frac{{\rm b}}{p}
$$
and this leads to consider
\eqn{sceltabb}
$$
\mathrm{b}\equiv \mathrm{b}(\beta)=\frac{p^2}{n}\left[
\frac{p\alpha(\ssf n+2\beta)+2\beta n(\ssf-1)}{2\beta(p+1)+\ssf np}
\right]>0\,.
$$
By \rif{totale}, now used with $C=3$, we have 
\eqn{notabasedopo}
$$
\trif{bound3dopo} \Longrightarrow  \frac qp <   1+\frac {4\alpha^2} {13n^2} \Longrightarrow 
 \ssf < 1 + \frac{\alpha}{3n} \leq  \frac 76
\,.
$$
This time the last inequality implies ${\rm b}\leq \alpha p^2/n \leq p$ (as $p\leq 2$) so that the choice in \rif{sceltabb} is admissible. 

Using $\mathrm{b}$ from \rif{sceltabb} in \rif{rela00}, and formally taking $\beta =\alpha/(2+\alpha)$, we find
$$
\frac qp < 1 
 + \frac{2\alpha p}{n(2+\alpha)}\left[\frac{\alpha-n(\ssf-1)}{\ssf np+2\alpha(p+1)/(2+\alpha)}\right]=:1+ \mathcal{R}_2(n, p, q,\alpha)\frac{\alpha^2}{n^2} \,.
$$
Using \rif{notabasedopo} yields $\mathcal{R}_2(n, p,q,  \alpha)> 8/33=\textnormal{\texttt{k}}$, and we can conclude as in the case $p\geq 2$. 
\subsection{Proof of Corollary \ref{t4}}\label{seziona0} 
Once \rif{stimasharp} is secured, we can modify the arguments of Sections \ref{generalapp}-\ref{generalapp2} to get \rif{stima3}. In Section \ref{generalapp}, replace the ball $B_{r}$ by $\Omega$; by the definition in \rif{rilassato} and the $q$-growth of the integrand $F(\cdot)$, we find  a sequence $\{\tilde u_{\eps}\}\subset W^{1,\infty}(\Omega)$ such that $u_{\eps} \deb u$ in $W^{1,p}(\Omega)$ and 
$ \mathcal F_{\texttt{x}}(\tilde u_{\eps},\Omega)-\overline{\mathcal F_{\texttt{x}}}(u,\Omega) = \texttt{o}(\eps,\Omega)$. Here we need that $\Omega$ is Lipschitz regular. The sequence $\{\tilde u_{\eps}\}$ will play the role of the similarly denoted one defined in \rif{recap}. We can  proceed as in Section \ref{generalapp}, where now it is $\FF(x,y,z)\equiv F(x, z)$. All the terms involving $f(\cdot)$ and $\hhh(\cdot)$, including $
\hhh_{\varepsilon}$ defined in \rif{defigg}, are absent. Defining this time $ \FF_{\eps, \delta}(x,z) :=(\FF(x,\cdot)*\phi_{\delta})(z)+\sigma_{\varepsilon}[H_{\mu_\delta}(z)]^{q/2}$ and $\mathcal{G}_{\varepsilon,\delta}(w,\Omega):=\int_{\Omega}\FF_{\eps, \delta}(x,Dw)
\dx$, we can repeat the arguments of Sections \ref{generalapp}-\ref{generalapp2} with the following replacements. Instead of \rif{conv0}, we have 
$
\mathcal{G}_{\eps, \delta}(\tilde u_{\eps}, \Omega)- \overline{\mathcal F_{\texttt{x}}}(u, \Omega)= \texttt{o}(\varepsilon,\Omega) + \texttt{o}_{\eps}(\delta,\Omega)
$, and \rif{conclude} becomes 
$
\nr{Du_{\eps, \delta}}_{L^p(\Omega)}^p +\sigma_{\eps}\nr{Du_{\eps, \delta}}_{L^q(\Omega)}^q\leq c  \overline{\mathcal F_{\texttt{x}}}(u, \Omega) +\texttt{o}(\varepsilon,\Omega) + \texttt{o}_{\eps}(\delta,\Omega)$. Using this last estimate in combination with \rif{stimasharp}, and a covering argument, we find the analog of \rif{stim1}
$$
\nr{Du_{\eps, \delta}}_{L^{\infty}(\Omega_0)}\le \frac{c(\data)}{[\dist(\Omega_0, \partial \Omega)]^{\chi_1}} \left[
\overline{\mathcal F_{\texttt{x}}}(u, \Omega)+1\right]^{\chi_2}+ \frac{\texttt{o}(\varepsilon,\Omega) + \texttt{o}_{\eps}(\delta,\Omega)}{[\dist(\Omega_0, \partial \Omega)]^{\chi_1}}\,,
$$
that holds for every open subset $\Omega_0 \Subset \Omega$. 
We can now proceed as in Section \ref{generalapp2}, with $u_{\eps, \delta}\deb u_\eps\in \tilde u_{\eps}+W^{1,q}_0(\Omega)$  weakly in $W^{1,q}(\Omega)$, and $u_{\eps}\deb v \in u+W^{1,p}_0(\Omega)$ weakly in $W^{1,p}(\Omega)$, up to the final argument in \rif{comein}, that now only yields $\overline{\mathcal F_{\texttt{x}}}(u, \Omega)=\overline{\mathcal F_{\texttt{x}}}(v, \Omega)$. The equality $u\equiv v$ is then implied by the strict convexity of $w \mapsto \overline{\mathcal F_{\texttt{x}}}(w, \Omega)$, see \cite[pp. 47-50]{sharp}; for this, recall that $u\equiv v$ on $\partial \Omega$. 
This settles the local boundedness of $Du$. For its local H\"older continuity, we revisit the arguments of Section \ref{holdergrad}. Using estimate \rif{stimasharp} instead of \rif{stima1pre}, we derive a suitable analog of \rif{stim1bis}, where the constant $M$ in \rif{stim1bis} is now of the form $M = c r^{-\chi_1}[\overline{\mathcal F_{\texttt{x}}}(u, \Omega)+1]^{\chi_2}$, where $r :=\min\{\dist(\Omega_0, \partial \Omega)/4, 1\}$. The integrand $F_0(\cdot)$ and $v$ are again defined via
\rif{pd-dopo} with the current definition of $\mathbb{F}_{\eps, \delta}(\cdot)$, so that that the properties in \rif{F_0part}-\rif{holly1} are still in force. Finally, proceeding as in the proof of \rif{nonabbiamo} and \rif{nonabbiamo2}, we find
$$\mint_{B_{\tau}}\snr{V(Du_{\eps, \delta})-V(Dv)}^{2}  \dx   \le  c\tau^{2 \alpha} [M_{B_{r}}(\eps, \delta)]^{2q-p} + c\una  \tau^{ \alpha} [M_{B_{r}}(\eps, \delta)]^{q+q/p-1}
 \leq c_{\eps, \delta}  \tau^{\alpha}\,,
$$
that in fact replaces \rif{holly11} in the present setting. After this, the rest of the proof proceeds as in Section \ref{generalapp2} almost unchanged and leads to \rif{localeh} for a different H\"older exponent $\alpha_*$.  

\subsection{Proof of Corollary \ref{c3}} By the discussion before Corollary \ref{c3}, $u$ minimizers the functional $w \mapsto \overline{\mathcal F_{\texttt{x}}}(u,B_{r})$ for every ball $B_{r} \Subset \Omega$. Therefore Corollary \ref{c3} follows from Corollary \ref{t4} applied with $\Omega \equiv B_{r}$ and then a standard covering argument. In particular, estimate \rif{stima3m} follows directly from \rif{stima3} by taking $B_t\Subset B_{r}$ as 
$\Omega'\Subset \Omega$ in Corollary \ref{t4}, and recalling that $\mathcal {L}_{\mathcal{F}_{\texttt{x}}}(w,B_{r})=0$ means that $\overline{\mathcal F_{\texttt{x}}}(u,B_{r})=\mathcal F_{\texttt{x}}(u,B_{r})$.
\subsection{Proof of Theorem \ref{t2}}\label{seziona} 
The integrand $\ccc(x)F(z)$ is of the type considered in Corollary \ref{t4}. Indeed, by Lemma \ref{marclemma}, the convexity of $F(\cdot)$ and \rif{assFF}$_2$ imply 
$ |\partial_{z} F(z)|\leq c(n,L) [H_{\mu}(z)]^{(q-1)/2}+c(n,L) [H_{\mu}(z)]^{(p-1)/2}$, so that assumption \rif{xx.0}$_2$ is verified. At this stage, it is sufficient to prove that $\mathcal L_{\mathcal{S}_{\texttt{x}}}(u, B)=0$ holds whenever $B\Subset \Omega$ is a ball, and Theorem \ref{t2} would follow from Corollary \ref{c3}. For this, observe that the sequence $\{\tilde u_\eps\}$ considered in \rif{recap}, is such that $\ccc(\cdot)F(D\tilde u_\eps)\to \ccc(\cdot)F(D u )$ in $L^1(B)$ by the convolution argument explained in Remark \ref{jensiremark}, and the proof is complete.   
\section{Theorem \ref{t3}} 
\begin{proposition}\label{priori5}
Let $u\in W^{1,q}(B_{r})$ be a minimizer of the functional $\mathcal{S}(\cdot, B_{r})$ in \trif{modellou}, where $B_{r}\Subset \Omega$ and $r\leq 1$, under assumptions \eqref{bound4}, \eqref{assu1} and $p>n$. 
Then
\eqn{stimasharpu}
$$
\nr{Du}_{L^{\infty}(B_{t})}\le \frac{c}{(s-t)^{\chi_1}}\left[\|Du\|_{L^p(B_s)} +\|f\|_{n/\alpha,1/2;B_{s}}+1\right]^{\chi_2}\,,
$$
holds whenever $B_{t}\Subset B_s\Subset B_{r}$ are concentric balls, where $c\equiv c (\data)$, $\chi_1, \chi_2\equiv \chi_1, \chi_2(\datae)$. \end{proposition}
\begin{proof} Here the setting is the one of Section \ref{variante3}. By the discussion after \rif{F_0}, this time with $\tilde F(x,y,z)\equiv \ccc(x,y)F(z)$, and the assumptions considered in Proposition \ref{priori5}, $u$ satisfies \rif{33.1bis}. This allows us to use that $Du$ is bounded in $B_s$. 
As in Proposition \ref{priori3}, we can assume that $\nr{Du}_{L^{\infty}(B_{t})}\ge 1$ and consider further concentric balls $B_{t}\Subset B_{\tau_{1}}\Subset B_{\tau_{2}}\Subset B_{s}$, $x_{0} \in B_{\tau_{1}}$. 
By \eqref{xx.10xu} with $M\equiv \nr{Du}_{L^{\infty}(B_{\tau_{2}})}$, we apply Lemma \ref{revlem} on $B_{r_{0}}(x_0)\equiv B_{(\tau_1-\tau_1)/8}(x_0)$ with $h\equiv 2$, $\kk_{0}\equiv 0$, $v \equiv E_{\mu}(Du)$, $f_1\equiv \snr{Du}+1$, $f_2\equiv f$, $M_{0}\equiv M^{\ssf(q-p)/2}$, $M_{1} \equiv M^{(\ssf q+p+\alpha-\mathrm{b})/2}$, $M_2=M^{(\ssf q+\alpha)/2}$, $t \equiv 2$, $\delta_{1}=\delta_{2}\equiv \alpha/2$, $m_1\equiv q-p+\mathrm{b}$, $m_2\equiv 1$, $\theta_1 =\theta_2\equiv 1$. Proceeding as for \rif{xx.1500}-\rif{xx.15} leads to
\begin{flalign}\label{xx.15p2}
\nr{Du}_{L^{\infty}(B_{\tau_{1}})}&\le\frac{c}{(\tau_{1}-\tau_{2})^{\frac{n}{2p}}}
\nr{Du}_{L^{\infty}(B_{\tau_{2}})}^{\left(\frac qp-1\right)\frac{\ssf \chi}{2(\chi-1)}+\frac 12}\left(\int_{B_{s}}(\snr{Du}+1)^p  \dx\right)^{\frac{1}{2p}}\nonumber \\
&\quad +c\nr{Du}_{L^{\infty}(B_{\tau_{2}})}^{\left(\frac qp-1\right)\frac{\ssf\chi}{2(\chi-1)}+\frac{\ssf +1}{2}+\frac{\alpha-\mathrm{b}}{2p}}\nr{\mathbf{P}_{2,\alpha/2}^{q-p+\mathrm{b},1}(\snr{Du}+1; \cdot, (\tau_{2}-\tau_{1})/4)}_{L^{\infty}(B_{\tau_{1}})}^{1/p}\notag 
\\
& \quad  +c\nr{Du}_{L^{\infty}(B_{\tau_{2}})}^{\left(\frac qp-1\right)\frac{\ssf \chi}{2(\chi-1)}+\frac {\ssf }{2}+\frac {\alpha}{2p}}\nr{\mathbf{P}^{1,1}_{2,\alpha/2}(f;\cdot,(\tau_{2}-\tau_{1})/4)}_{L^{\infty}(B_{\tau_{1}})}^{1/p}+c\,,
\end{flalign}
where $c\equiv c(\data, \beta)$. Note that the highest power of $\nr{Du}_{L^{\infty}(B_{\tau_{2}})}$ appears in the second term of the above display since it is ${\rm b}\leq p$. As in Proposition \ref{priori3}, recalling also condition \rif{lo.2.1} to apply Lemma \ref{crit} (observe that \rif{lo.2.1} is automatically satisfied since $n\geq 2$), this time we impose 
\eqn{csisi1u}
$$
 \left(\frac qp-1\right)\frac{\ssf n}{4\beta}+\frac{\ssf +1}{2} + \frac{\alpha-\mathrm{b}}{2p}< 1 \quad \mbox{and} \quad \frac{n(q-p+\mathrm{b})}{\alpha} < p\,,
$$
that is
\eqn{csisi2}
$$
\frac qp < 1 + \left(1-\ssf +\frac{{\rm b-\alpha}}{p}\right)\frac{2\beta}{\ssf n}\quad \mbox{and} \quad 
 \frac qp < 1+ \frac{\alpha}{n} -\frac{\mathrm{b}}{p}\,,
$$
respectively. 
Equalizing the right-hand sides leads to consider
\eqn{valoreb}
$$
{\rm b}\equiv {\rm b}(\ssf)=\frac{p(2\beta +\alpha )(\ssf-1)+\alpha(2 \beta+p)}{\ssf n+2\beta}\,.
$$
Now, note that 
\eqn{usingss0}
$$
 \frac{\alpha}{n} -\frac{\mathrm{b}}{p} = \frac{\alpha}{n}  -\frac{2\beta +\alpha}{\ssf n+2\beta}(\ssf -1) - \frac{\alpha(2 \beta+ p)}{p(\ssf n+2\beta)}
=\frac{2\alpha\beta(1-n/p)}{n(\ssf n+2\beta)}- \frac{2\beta(\ssf-1)}{\ssf n+2\beta}\,.
$$
On the other hand, observe that 
\begin{flalign}
\notag \frac{2\beta(\ssf-1)}{\ssf n+2\beta} < \frac{\alpha\beta(1-n/p)}{n(\ssf n+2\beta)}& \Longleftrightarrow  \ssf -1 < \frac{\alpha(1-n/p)}{2n}\\
& \stackrel{\rif{marcexp}}{ \Longleftrightarrow }
\frac{q}{p} < 1+ \frac{2\alpha (1-n/p)}{n[2(n+2)+\alpha (1-n/p)]} \Longleftarrow \trif{bound4}\,.
\label{usingss}
\end{flalign}
Using the first inequality in \rif{usingss} in \rif{usingss0} yields that ${\rm b} < p\alpha/n < p$, so that the value of ${\rm b}$ in \rif{valoreb} is admissible in \rif{xx.10xu} and in \rif{xx.15p2}.  To proceed, we use
\rif{usingss0}-\rif{usingss} in \rif{csisi2} and conclude that, in order to verify \rif{csisi2} for some $\beta< \alpha /(2+\alpha)$, it suffices to verify
\eqn{dett33}
$$
\frac qp  < 1+ \frac{\alpha\beta(1-n/p)}{n(\ssf n+2\beta)}\,.
$$
By formally taking the limiting value $\beta= \alpha /(2+\alpha)$ here, we come to
\eqn{dett333}
$$
\frac qp  < 1+ \frac{1-n/p}{[\ssf+2\alpha/(n(2+\alpha))](2+\alpha)}\frac{\alpha^2}{n^2}
=: 1+\mathcal{R}_3(n,p,q, \alpha)\left(1-\frac{n}{p}\right) \frac{\alpha^2}{n^2} \,.
$$ 
Noting that it is $\ssf < 5/4$ by the second inequality in display \rif{usingss}, we infer the lower bound $\mathcal{R}_3(n,p,q, \alpha)> 4/19$, so that \rif{dett333} is again implied by \rif{bound4}. We conclude we can find $\beta < \alpha /(2+\alpha)$ such that \rif{dett33} and therefore \rif{csisi2} and \rif{csisi1u} are satisfied. Finally, thanks to the second inequality in \rif{csisi1u} we can use \eqref{11}, that yields
$$
\nr{\mathbf{P}_{2,\alpha/2}^{q-p+\mathrm{b},1}(\snr{Du}+1; \cdot, (\tau_{2}-\tau_{1})/4)}_{L^{\infty}(B_{\tau_{1}})}
\leq c\nr{\snr{Du}+1}_{L^{ p}(B_{s})}^{\frac{q-p+\mathrm{b}}{2}} \,.
$$
Using this last inequality and \rif{vvee1} in \rif{xx.15p2}, we can now conclude as for Proposition \ref{priori3}. \end{proof}

Once the a priori estimate of Proposition \ref{priori5} is available, we can proceed as for the proof of the Theorems \ref{t1} and \ref{t2}. Specifically, the proof of \rif{stima2mm} is totally analogous to the one of \rif{stima1}, via the approximation arguments of Sections \ref{generalapp}-\ref{generalapp2} applied with the new definition 
$\FF_{\eps, \delta}(x,y,z)  :=\ccc(x,y)(F*\phi_{\delta})(z)+\sigma_{\varepsilon}\ccc(x,y)[H_{\mu_\delta}(z)]^{q/2}$. This guarantees that the approximating integrands still preserve the product structure used in Proposition \ref{caccin3}; the convergence in \rif{conv0} still takes place (see also the remarks in the proof of Theorem \ref{t2} in Section \ref{seziona}). As for the local H\"older continuity of $Du$, we proceed exactly as in Section \ref{holdergrad}, with \rif{pd-dopo} used with the current definition of $\mathbb{F}(\cdot)$. In this case the analogs of \rif{holly11}-\rif{holly2} can be obtained estimating as in \rif{xx.6}.

\section{Theorem \ref{t6}}\label{equazioniprova}
\begin{proposition}\label{priori6}
Let $u\in W^{1,q}(B_{r})$ be a weak solution to \trif{eulera}, under assumptions \eqref{assAA} with $0 < \mu \leq 2$ and $\nu, L$ replaced by $\tilde \nu, \tilde L$ (as in Section \ref{basicn}), and assume also \eqref{xx.3equ}. If \eqref{bound5} is in force, then \eqref{stimasharp} 
holds whenever $B_{t}\Subset B_s\Subset B_{r}$ are concentric balls, where $c\equiv c (\data)$, $\chi_1, \chi_2\equiv \chi_1, \chi_2(\datae)$. 
\end{proposition}
\begin{proof}
The setting of Proposition \ref{priori6} is the one of Section \ref{variante2e} and by Remark \ref{regolarissime} we have that $Du$ is locally bounded in $B_{r}$. Note that \eqref{bound5} implies \rif{marcexp2} and therefore Proposition \ref{caccin5} can be used. The proof closely follows the one of Proposition \ref{priori3}. We therefore confine ourselves to give a sketch of it. First, we note that
$$
\frac{q}{p} \leq 1 + \frac{p-1}{p}\frac{\alpha}{2C n} \Longleftrightarrow \frac{q-1}{p-1} \leq  1 + \frac{\alpha}{2Cn}
$$ whenever $C \geq 1$. 
Using this, and the definition of $\mathfrak{t}$ in \rif{marcexpeq}, it is not difficult to see that 
\eqn{alge11}
$$
\frac{q}{p} \leq  1+ \frac {p-1}{p} \frac{\alpha^2}{4Cn^2} \Longrightarrow 
\frac{q}{p} <  1+\frac{\alpha^2}{2Cn(n+1)} 
\Longrightarrow \sst <1+\frac{\alpha}{Cn}\,.
$$
We take $C=1$ when $p\geq 2$, so that the right-hand side inequality in \rif{alge11} is implied by \rif{bound5}. Proceeding as in Section \ref{method} for the case $p\geq 2$, but using \rif{xx.10eq} instead of \rif{xx.10}, we arrive at the following analog of \rif{xx.15}:
\begin{flalign*}
\nr{Du}_{L^{\infty}(B_{\tau_{1}})}&\le\frac{c}{(\tau_{1}-\tau_{2})^{\frac{n}{2p}}}
\nr{Du}_{L^{\infty}(B_{\tau_{2}})}^{\left(\frac qp-1\right)\frac{\sst \chi}{\chi-1}+\frac 12}\left(\int_{B_{s}}(\snr{Du}+1)^p  \dx\right)^{\frac{1}{2p}}\nonumber \\
&\qquad +c\nr{Du}_{L^{\infty}(B_{\tau_{2}})}^{\left(\frac qp-1\right)\frac{\sst \chi}{\chi-1}+\frac{\sst+1}{2}-\frac{\mathrm{b}}{2p}}\nr{\mathbf{P}^{2q-2p+\mathrm{b},1}_{2,\alpha}(\snr{Du}+1;\cdot,(\tau_{2}-\tau_{1})/4)}_{L^{\infty}(B_{\tau_{1}})}^{1/p}+c\,,
\end{flalign*}
valid for every ${\rm b}\in [0,p]$, and this leads to consider the conditions
\eqn{csisi22eq}
$$
\frac qp < 1 + \left(1-\sst+\frac{{\rm b}}{p}\right) \frac{\beta}{\sst n} \quad \mbox{and} \quad \frac qp < 1+ \frac{\alpha}{n} -\frac{\mathrm{b}}{2p}\,.
$$
As done in Proposition \ref{priori3}, we can restrict to the case $\alpha <1$. We choose $
\mathrm{b}\equiv \mathrm{b}(\beta)=2p[\sst \alpha  +\beta(\sst-1)]/[\sst n+2\beta]$, which is admissible by \rif{alge11} (with $C=1$). Using such $\mathrm{b}$ in \rif{csisi22eq}, and formally taking $\beta =\alpha/(1+\alpha)$, we get
\eqn{bound3veroeq}
$$
\frac qp < 1 + \frac{\alpha}{(1+\alpha)n}\left[\frac{2\alpha-n(\sst -1)}{\sst n+2\alpha/(1+\alpha)}
\right]=: 1+ \mathcal{R}_4(n,p,q, \alpha)\frac{\alpha^2}{n^2} \,.
$$
By \rif{alge11} it is $ \mathcal{R}_4(n,p,q, \alpha)>  1/4$, so that \rif{bound3veroeq} is implied by \rif{bound5}. We now come to the case $p< 2$, where we use \rif{alge11} with $C=2$, and \rif{xx.10-sotto2eq} gives this time 
\begin{flalign*}
&\nr{Du}_{L^{\infty}(B_{\tau_{1}})}\le\frac{c}{(\tau_{1}-\tau_{2})^{\frac{n}{2p}}}
\nr{Du}_{L^{\infty}(B_{\tau_{2}})}^{\left(\frac qp-1\right)\frac{\sst\chi}{\chi-1}+\frac 12}\left(\int_{B_{s}}(\snr{Du}+1)^p  \dx\right)^{\frac{1}{2p}}\nonumber \\
&\qquad +c\nr{Du}_{L^{\infty}(B_{\tau_{2}})}^{\left(\frac qp-1\right)\frac{\sst\chi}{\chi-1}+\frac{\sst}{2}+\frac{q-\mathrm{b}/p}{2p}}\nr{\mathbf{P}^{p(q-p)/(p-1)+{\rm b},1/p}_{2,\alpha/2}(\snr{Du}+1;\cdot,(\tau_{2}-\tau_{1})/4)}_{L^{\infty}(B_{\tau_{1}})}^{1/p}+c\,,
\end{flalign*} 
as an analog of \rif{xx.15-sotto2}. Therefore we consider the conditions
$$
\left(\frac qp-1\right)\frac{\sst n }{2\beta}+\frac{\sst}{2}+\frac{q-\mathrm{b}/p}{2p} < 1 \quad \mbox{and} \quad 
 \, \frac{n}{p\alpha}\left[\frac{p(q-1)}{p-1}-p+\mathrm{b}\right] < p\,,
 $$
that are equivalent to
\eqn{rela00eq}
$$
\begin{cases}
\, \frac qp < 1 + \left(1-\sst +\frac{\rm b}{p^2}\right)\frac{\beta}{\sst n+\beta}\\
\, \frac {q-1}{p-1} < 1+\frac{\alpha p}{n}-\frac{{\rm b}}{p} \Longleftrightarrow 
\frac qp < 1 + \left(\frac{\alpha}{n}-\frac{{\rm b}}{p^2}\right)(p-1)\,.
\end{cases}
$$
We take ${\rm b}=14p^2\alpha/(15n)\leq p$; using this value in \trif{rela00eq}, and \rif{alge11} with $C=2$, makes the two inequalities in \rif{rela00eq} implied by \rif{bound5} provided we take $\beta$ close enough to $\alpha/(2+\alpha)$, and we conclude again as in Proposition \ref{priori3}.  
\end{proof}
\begin{proposition}\label{apxsharp}
Let $u\in W^{1,q}(B_{r})$ be a weak solution to \trif{eulera}, under assumptions \eqref{assAA} with $0 < \mu \leq 2$ and $\nu, L$ replaced by $\tilde \nu, \tilde L$ (as in Section \ref{basicn}), and assume also \eqref{xx.3equ}.
Then
\eqn{higherint}
$$
[Du]_{\tilde \alpha, p;B_{t}}\le \frac{c}{(s-t)^{\chi_1}}\left[\|Du\|_{L^q(B_s)} +1\right]^{\chi_2}
$$
holds whenever $B_{t}\Subset B_s\Subset B_{r}$ are concentric balls and $\tilde \alpha<\min\{1/p,1/2\}\alpha$, where $ c\equiv c (\data)$ and $\chi_1, \chi_2\equiv \chi_1, \chi_2 (\datae)$. \end{proposition}
\begin{proof} In the case of minimizers, \rif{higherint} is hidden in \cite[Proof of Theorem 4]{sharp}. The arguments in \cite{sharp} rely on the use of the Euler-Lagrange equation and they work in the case of the general equations considered here. Indeed, from \cite[(51)]{sharp} we have that 
\eqn{puff1}
$$
\int_{B_t} \snr{\tau_{h}V(Du)}^2\dx \leq \frac{c |h|^{\alpha}}{(s-t)^{\theta}} \int_{B_{s}} (\snr{Du}+1)^q\dx
$$
holds whenever $B_{t}\Subset B_s\Subset B_{r}$ are concentric balls and $h \in \er^n$ such that $|h| \leq (s-t)/4$, with $c \equiv c (\data)$ and $\theta\equiv \theta(\datae)$. Using \rif{puff1} with  \rif{sopradue}, we obtain 
\begin{flalign*}
\int_{B_t} \snr{\tau_{h}Du}^p\dx  &\leq c\int_{B_t} \snr{\tau_{h}V(Du)}^2\dx +  c\una
\int_{B_t} \snr{\tau_{h}V(Du)}^p(\snr{Du}+1)^{p(2-p)/2}\dx  \\
 & \leq c\int_{B_t} \snr{\tau_{h}V(Du)}^2\dx +c\una\left(\int_{B_t}\snr{\tau_{h}V(Du)}^2\dx\right)^{\frac{p}{2}}
\left(\int_{B_t}(\snr{Du}+1)^p\dx\right)^{\frac{2-p}{2}}\\
&\leq   \frac{c |h|^{\min\{1,p/2\}\alpha}}{(s-t)^{\theta}} \int_{B_{s}} (\snr{Du}+1)^q\dx
\end{flalign*}
for $c\equiv c (\data)$, $\theta \equiv \theta(n,p,q,\alpha)$. The information in the last display and Lemma \ref{l4} now imply \rif{higherint}. Note that \rif{higherint} does not require any upper bound on $q/p$. 
\end{proof} 
We now complete the proof of Theorem \ref{t6}. With the same notation on sequences $\{\eps\}$ and mollifiers $\{\phi_{\eps}\}$ of Section \ref{generalapp}, we set
\eqn{vettoriA}
$$A_{\eps} (x,z):=\left(A(x, \cdot)*\phi_\eps\right)(z)+\eps [H_{\mu_{\eps}}(z)]^{(q-2)/2}z\,,\quad (x,z)\in \Omega \times \er^n\,,$$  
compare with \rif{nuoviF}. We then have 
\begin{lemma}\label{converlemma} Under assumptions \eqref{assAA}, $| A_{\eps}(x, z)-A(x, z)|\leq c\eps^{\min\{1, p-1\}}$ holds for $(x, z)\in \Omega \times \er^n$, provided $\snr{z}\leq M$, where $M\geq 1$ and $c$ is independent of $\eps$. Moreover, $|A_{\eps} (x,z)|\leq c[H_1(z)]^{(q-1)/2}$ holds for every $(x, z)\in \Omega \times \er^n$, where $c$ is again independent of $\eps$. 
\end{lemma}
\begin{proof}
We have
\begin{flalign*}
& |A_{\eps}(x, z)-A(x, z)|  \leq  c\eps \int_{\BB}\int_0^1 [H_{\mu}(t(z+\eps \lambda)+(1-t)z)]^{\frac{q-2}{2}}\, dt \,  \phi_\eps(y) \dla  \\
& \hspace{35mm}+  c\eps \int_{\BB}\int_0^1 [H_{\mu}(t(z+\eps \lambda)+(1-t)z)]^{\frac{p-2}{2}}]\, dt \,\phi_\eps(y) \dla  + c\eps M^{q-1}\\
& \qquad \quad  \leq  c\eps \int_{\BB}\left[(\snr{z+\eps \lambda}^2+\snr{z}^2+\mu^2 )^{\frac{q-2}{2}}+(\snr{z+\eps \lambda}^2+\snr{z}^2 +\mu^2)^{\frac{p-2}{2}}\right]\phi_1(y)\dla  + c\eps M^{q-1}\\
& \qquad \quad =:  c\eps \mathbb{I}_{p}(z) +   c\eps \mathbb{I}_{q}(z) + c\eps M^{q-1} \leq c\eps^{\min\{1, p-1\}}c(M) \,.
\end{flalign*}
Note that, to get the estimate in the first line of the above display, we have used the definition in \rif{vettoriA} and the pointwise bound on $\partial_zA(\cdot)$ assumed in \eqref{assAA}$_1$. For the second one, we have used \rif{elemint}. It remains to justify the estimate in the last line. We treat $\eps \mathbb{I}_{q}(z)$, the estimate for $\eps \mathbb{I}_{p}(z)$ being completely similar. When $q\geq 2$, we find $\eps\mathbb{I}_{q}(z)\leq c\eps [H_{\mu_{\eps}}(z)]^{(q-2)/2}\leq c \eps M^{q-2} \leq \eps c(M)$. When $q<2$, instead we further  distinguish two cases. The first is when $\snr{z}\geq \eps$; then we have 
$\eps \mathbb{I}_{q}(z) \leq c \eps \snr{z}^{q-2} \leq c\eps^{q-1}$. Finally, if $q<2$ and $\snr{z} \leq \eps$, then we have 
$\eps \mathbb{I}_{q}(z)\leq c \eps \int_{\BB}\snr{z+\eps\lambda}^{q-2}\, d\lambda \leq c  
\eps^{1-n}\int_{\mathcal B_{2\eps}}|\lambda|^{q-2}\, d\lambda \leq c \eps^{q-1}$. The first assertion in the lemma is proved. The second one trivially follows from \rif{assAA}$_1$. 
\end{proof}
To proceed, we define $u_{\eps} \in u_0 + W^{1,q}_0(\Omega)$ as the (unique) solution to $\diver\, A_{\eps} (x,Du_{\eps})=0$ in $\Omega$, and such that 
$u_{\eps}\in u_0+W^{1,q}_0(\Omega)$, 
where 
$
A_{\eps} (\cdot)$ is defined in \rif{vettoriA}. 
For every $\eps$, the vector field $A_{\eps}(\cdot)$ satisfies the assumptions required on $A(\cdot)$ in Proposition \ref{caccin5}. In particular, \eqref{assAA} are satisfied with $\mu$ replaced by $\mu_{\eps}:= \mu+\eps>0$, and for new constants $0 <\tilde \nu \leq  \tilde  L$ as in Section \ref{basicn}, replacing $\nu, L$, and independent of $\eps$. This can be easily proved using the arguments of \cite[Section 4.5]{dm1}. Using \rif{coerpq} yields the uniform bound $\|Du_{\eps}\|_{L^{p}(\Omega)} \lesssim \||Du_{0}|+1\|_{L^{p(q-1)/(p-1)}(\Omega)}^{(q-1)/(p-1)} 
$ (note that $p(q-1)/(p-1)\geq q$). Combining this with \rif{stimasharp}, we get the local estimate
\eqn{stimasharpdopo}
$$
\nr{Du_{\eps}}_{L^{\infty}(B/2)}\le  \frac{c}{|B|^{\chi_1}} \left(\int_{\Omega} (|Du_0|+1)^{\frac{p(q-1)}{p-1}}\dx +1 \right)^{ \chi_2}\,.
$$
This holds whenever $B \Subset \Omega$ is a ball, where $c\equiv c (\data)\geq 1$, $\chi_1, \chi_2\equiv \chi_1, \chi_2(\datae)\geq 1$ are otherwise independent of $\eps$. 
Up to not relabelled subsequences we can therefore assume that $u_{\eps}\deb u$ in $W^{1,p}(\Omega)$, for some $u \in u_0 +W^{1,p}_0(\Omega)$, and that $\{Du_{\eps}\}$ is bounded in $L^{\infty}_{\loc}(\Omega, \er^n)$. This and Proposition \ref{apxsharp} yield a uniform bound on $\{Du_{\eps}\}$ in $W^{\tilde \alpha, p}_{\loc}(\Omega,\er^n)$. Again up to subsequences, we can assume that $Du_{\eps} \to Du$ in $L^{\gamma}_{\loc}(\Omega,\er^n)$ for some $\gamma < np/(n-p \tilde \alpha)$, and a.e. Using this last fact and \rif{stimasharpdopo}, by interpolation it then follows that $Du_{\eps} \to Du$ in $L^{\gamma}_{\loc}(\Omega, \er^n)$ for every $\gamma < \infty$. This is sufficient to claim that $u$ is a distributional solution to \rif{dir1}$_1$, by Lemma \ref{converlemma} and dominated convergence. Letting $\eps \to 0$ in \rif{stimasharpdopo} leads to \rif{stima} via coverings. Finally, the local H\"older continuity of $Du$ follows by the methods employed in Section \ref{holdergrad}. Note that in defining the comparison functions $v$ in \rif{pd-dopo}, now we take $v$ as the solution to\,$\diver\, A_\eps(x_{\rm c}, Dv)=0$ such that $v\equiv u_{\eps}$ on $B_{\tau}$. We remark that the analog of \rif{stim1bis} is obtained via \rif{stimasharpdopo}, and the one of \rif{altrasup} follows as in \rif{fuckbsceq}. Finally, the analog of \rif{holly11} follows estimating as for  \rif{nonabbiamo} and \rif{servizio}. 
\section{Corollaries \ref{c1}, \ref{c2} and \ref{c4}}\label{reloaded}
\subsection{Proof of Corollary \ref{c1}}\label{S1} Under the assumptions of Corollary \ref{c1}, 
once $Du$ is known to be locally H\"older continuous, its H\"older exponent can be upgraded up to the maximal one via a combination of a few classical regularity arguments and estimates for nonuniformly elliptic problems. Here we give the details. Since the result is local in nature, we can assume that 
\eqn{starty}
$$\|Du\|_{L^\infty(\Omega)}+[Du]_{0, \beta;\Omega}=:M <\infty$$ and that $M\geq 1$; here $\beta$ is the H\"older exponent of $Du$ provided by Theorem \ref{t2}. We can also assume that $\beta< \alpha$, otherwise there is nothing to prove. In the following, the constants denoted by $c$ will depend on $\data$; additional dependences will be emphasized in parentheses. We take a ball $B_{r}\equiv B_{r}(x_{\rm c})\Subset \Omega$ such that $r\leq 1$, denote $A(x, z)= \ccc(x)\partial_zF(z)$ and $A_{r}(z):=A(x_{\rm c},z)+r^{\alpha}[H_{\mu}(z)]^{(q-2)/2}z$. 
Since $A_{r}(\cdot)$ is $q$-monotone, we can take $v\in u+W^{1,q}_0(B_{r})$ such that 
$
\diver\, A_{r}(Dv) =0$ in $B_{r}$. As in the proof of \rif{nonabbiamo} in Proposition \ref{caccin4}, we find
\eqn{campcomp0}
$$\mint_{B_{r}}\snr{V(Du)-V(Dv)}^{2}\dx\leq cM^{2q-p}
r^{2\alpha}\,.$$
Using \rif{ineV} in the above inequality, and recalling that $p\geq 2$ and $\mu>0$, yields
\eqn{campcomp}
$$
 \mint_{B_{r}}\snr{Du-Dv}^{2}\dx  \leq  c\mu^{2-p}M^{2q-p}r^{2\alpha}\,.$$
This last inequality and \rif{starty} gives
\eqn{findy0}
$$
  \mint_{B_{r}} |Dv-(Du)_{B_{r}}|^2 \dx\leq c(\mu, M)r^{2\beta}\,.
$$
By $\mu>0$, standard regularity theory, or Lemma \ref{cacc-class1}, give $Dv\in L^{\infty}_{\loc}(B_{r})\cap W^{1,2}_{\loc}(B_{r})$. Note that \rif{starty}  and \rif{campcomp0} imply $\|Dv\|_{L^p(B_{r})}^p\leq c(M)r^n$.
This and \rif{stimaprimissima} give
\eqn{findy}
$$\|Dv\|_{L^\infty(B_{r/2})}\leq c(M)\,.$$
Moreover, every component $\mathfrak{v}\equiv D_{s}v$, $s\in \{1, \ldots, n\}$ is an energy solution to the linear elliptic equation $\diver\, (\mathbb{A}(x)D\mathfrak{v})=0$, where $[\mathbb{A}(x)]_{ij}:= \partial_{z_j}A_{r}^i(x_{{\rm c}},Dv(x))(=\partial_{z_iz_j}F_{r}(x_{{\rm c}},Dv(x)))$. Again by $\mu>0$ and thanks to \rif{findy}, we find $\lambda \equiv \lambda (\data, \mu, M)>0$, independent of $r$, such that $\lambda \mathds{I}_{\rm d} \leq \mathbb{A}(x)\leq (1/\lambda)\mathds{I}_{\rm d}$ holds for a.e. $x \in B_{r/2}$. This allows to apply De Giorgi-Nash-Moser theory, that yields $\beta_0 \equiv \beta_0 (\data, \mu, M)\in (0,1)$, and $c \equiv c (\data, \mu, M)\geq 1$, such that 
\eqn{findy2}
$$
r^{2-n}\|DD_sv\|_{L^2(B_{r/4})}^2+
r^{2\beta_0}[D_sv]_{0, \beta_0; B_{r/4}}^2 \leq c
\mint_{B_{r/2}}|D_sv-a|^2\dx\,,
 $$ 
holds for every $a \in \er$ and  $s\in \{1, \ldots, n\}$. 
Note that we have also incorporated in \rif{findy2} the standard Caccioppoli inequality for linear elliptic equations \cite[Theorem 6.5]{giu}. As there is no loss of generality in assuming that $\beta_0 \leq \beta$, choosing $a\equiv (D_su)_{B_{r}}$ and using \rif{findy0} in \rif{findy2}, yields 
$[Dv]_{0, \beta_0; B_{r/4}} \leq c_*\equiv c_*(\data, \mu, M)$.  We now denote by $\omega(\cdot)$ the modulus of continuity of $\partial_{z} A_{r}(x_{{\rm c}}, \cdot)$ on $\mathcal B_{M}$. This is independent of $r$ and of the point $x_{\rm c}$. Recalling the definition of $\mathbb{A}(\cdot)$, we therefore have $|\mathbb{A}(x)-\mathbb{A}(y)| \leq \omega(c_*|x-y|^{\beta_0})=: \tilde \omega (|x-y|)$, whenever $x, y \in B_{r/4}$, so that the entries $\mathbb{A}(\cdot)$ are continuous, with a modulus of continuity that depends in a quantitative way on $\data, \mu$ and $M$. Campanato's perturbation theory and \rif{findy2}, now imply
$$
\int_{B_\varrho} |D^2v|^2 \dx \leq\frac {c}{r^2}\left[\left(\frac{\rr}{r}\right)^{n}+[\tilde \omega(r)]^2\right]\int_{B_{r}} |Dv-(Du)_{B_{r}}|^2 \dx
$$
holds whenever $\varrho \leq r/4$, where $c \equiv c (\data, \mu, M)$; here we have also used \cite[(10.42)]{giu}. Poincar\'e inequality now gives, this time whenever $\varrho \leq r$
\eqn{camp}
$$
\int_{B_\varrho} |Dv-(Dv)_{B_\varrho}|^2 \dx \leq c\left[\left(\frac{\rr}{r}\right)^{n+2}+\left(\frac{\rr}{r}\right)^{2}[\tilde \omega(r)]^2\right]\int_{B_{r}} |Dv-(Du)_{B_{r}}|^2 \dx\,.
$$ 
Using \rif{camp} in combination with \rif{campcomp}, and a standard comparison argument, we conclude with 
$$
\int_{B_\varrho} |Du-(Du)_{B_\varrho}|^2 \dx  \leq c\left[\left(\frac{\rr}{r}\right)^{n+2}+[\tilde \omega(r)]^2\right]\int_{B_{r}} |Du-(Du)_{B_{r}}|^2 \dx+ cr^{2\alpha+n}
$$
again for every $\varrho \leq r $, where $c \equiv c (\data, \mu, M)$. Since $h(\varrho):=\|Du-(Du)_{B_{\varrho}}\|_{L^{2}(B_{\rr})}^2$ is non-decreasing, we are in position to use Lemma \ref{l5bis}. It follows that there exists $r_0$ depending on $\data, \mu, M$ and on the local modulus of continuity of $\partial_{zz}F(\cdot)$, such that
$
\int_{B_\varrho} |Du-(Du)_{B_\varrho}|^2 \dx \leq c \varrho^{n+2\alpha}
$
holds provided $\varrho \leq r_0$, where $c $ depends on $\data, \mu, M$ and the modulus of continuity of $\partial_{zz}F(\cdot)$. At this stage, the local $C^{1, \alpha}$ continuity of $Du$ follows from Campanato-Meyers integral characterization of H\"older continuity. This completes the proof of Corollary \ref{c1}. 
\subsection{Proof of Corollary \ref{c2}}\label{S11} The proof is a modification of the one for Corollary \ref{c1}, and we keep the notation used there. Without loss of generality we can assume that $\|f\|_{L^{\infty}(\Omega)}\leq 1$ and $\beta<\alpha/2$. This time we define $v\in u+W^{1,q}_0(B_{r})$ as the unique minimizer of $w\mapsto \int_{B_{r}}F_r(Dw)\, dx$ in its Dirichlet class, where $F_r(z):= \ccc(x_{\rm c}, u(x_{\rm c}))F(z) + r^\alpha [{H_{\mu}(z)}]^{q/2}$. We proceed as for \rif{nonved}-\rif{xx.6} (with $8|h|^{\beta_0}\equiv r$ and $\rr\equiv 1$), getting
\begin{flalign*}
& \int_{B_{r}}\snr{V(Du)-V(Dv)}^{2}  \dx \leq c\int_{B_{r}} \left[ F_{r}(Du)- F_{r}(Dv)\right]  \dx \\
& \qquad \leq c\int_{B_{r}} \ccc(x_{\rm c}, u(x_{\rm c}))[F(Du)- F(Dv)]  \dx + c r^{\alpha} \int_{B_{r}} [H_{\mu}(Du)]^{q/2}\dx \\
& \qquad  \leq c
 \int_{B_{r}} \ccc(x_{\rm c}, u(x_{\rm c}))[F(Du)- F(Dv)]  \dx +c M^q r^{n+\alpha} \\
& \qquad  \leq c
 M^{q+\alpha}r^{n+\alpha}  +c M^q r^{n+\alpha} \leq c (M)r^{n+\alpha} \,.
 \end{flalign*}
 In turn, via \rif{ineV} this implies $\int_{B_{r}}\snr{Du-Dv}^{2}  \dx \leq c (\mu,M)r^{n+\alpha}$. 
These estimates are the counterparts of \rif{campcomp0} and \rif{campcomp} and from this point on the proof develops as in Corollary 
\ref{c1}, replacing $\alpha$ by $\alpha/2$ everywhere. 
\subsection{Proof of Corollary \ref{c4}}\label{S111} Let $u$ be a distributional solution to $\diver \, A(x, Du)=0$, satisfying \rif{starty}, and under the assumptions of Corollary \ref{c4}. It is easy to see that the proof of Corollary \ref{c1} applies verbatim to this situation, as it does not use the minimality of $u$ beyond the fact that $u$ solves the Euler-Lagrange equation. Indeed, up to passing to inner domains of $\Omega$, we can assume that $\partial_z A(\cdot)$ is uniformly continuous on $ \Omega \times \mathcal B_{M}$, so that we can find a modulus of continuity of $z \mapsto \partial_z A (x_{\rm c}, z)$, which is independent of the chosen point $x_{\rm c}$. We conclude that Corollary \ref{c4} follows from Theorem \ref{t6} (used to get \rif{starty} up to passing to smaller open subsets).


\begin{thebibliography}{}

\bibitem {AM} E. Acerbi \& G. {Mingione}, Regularity results for a class of functionals with non-standard growth, {\em Arch. Ration. Mech. Anal.} 156, 121--140 (2001). 

\bibitem{BCM} P. Baroni, M. Colombo \& G. Mingione, Regularity for general functionals with double phase, {\em Calc. Var. \& PDE} 57:62 (2018). 
 
\bibitem{BM} L. Beck \& G. Mingione, Lipschitz bounds and nonuniform ellipticity, \emph{Comm. Pure Appl. Math.} 73, 944-1034 (2020).

\bibitem{BS0} P. Bella \& M. Sch\"affner, Local boundedness and Harnack inequality for solutions of linear nonuniformly elliptic equations, \emph{Comm. Pure Appl. Math.} 74, 453-477 (2021). 

\bibitem{BS} P. Bella \& M. Sch\"affner, On the regularity of minimizers for scalar integral functionals with $(p,q)$-growth, \emph{Anal. \& PDE} 13, 2241--2257 (2020). 

\bibitem{BB} L. Bousquet \& L. Brasco, Lipschitz regularity for orthotropic functionals with nonstandard growth conditions, \emph{Rev. Mat. Iberoam.} 36, 1989–2032 (2020). 

\bibitem {BO0} 
S.\,S.~Byun \& J.~Oh, Global gradient estimates for nonuniformly elliptic equations, {\em Calc. Var. \& PDE} 56:42 (2017).

\bibitem {BO} 
S.\,S.~Byun \& J.~Oh, Regularity results for generalized double phase functionals, {\em Anal. \& PDE} 13, 1269--1300 (2020).

\bibitem {cacc1} R. Caccioppoli, Sulle equazioni ellittiche a derivate parziali con $n$ variabili
indipendenti, {\em Atti Acc. Naz. Lincei. Rendiconti Lincei: Matematica e Appl. (VI)} 19, 83-89 (1934). 

\bibitem{camp} S. Campanato, Equazioni ellittiche del ${\rm II}$ ordine
e spazi $\mathfrak{L}\sp{(2,\lambda )}$,  {\em Ann. Mat. Pura Appl. (IV)} 69, 321--381 (1965). 


\bibitem {Ce} A. Cellina, Strict convexity and the regularity of solutions to variational problems, {\em ESAIM: COCV} 22  862–871 (2016). 

\bibitem {Ce1} A. Cellina, A case of regularity of solutions to nonregular problems, 
{\em SIAM J. Control Optim.} 53, 2835–2845 (2015).

\bibitem{CM0}
A.~Cianchi \& V.\,G.~Maz'ya, Global Lipschitz regularity for a class of quasilinear equations,  
\emph{Comm. PDE} 36, 100--133 (2011). 

\bibitem{CM1}
A.~Cianchi \& V.\,G.~Maz'ya, Global boundedness of the gradient for a
class of nonlinear elliptic systems, \emph{Arch. Ration. Mech. Anal.} 212, 129--177 (2014). 

\bibitem{colombotione}  M. Colombo \& R. Tione, Non-classical solutions to the $p$-Laplace equation. {Preprint 2022}. 


\bibitem{cammello}  C. De Filippis, Quasiconvexity and partial regularity via nonlinear potentials,  \emph{J. Math. Pures Appl.}, to appear. \url{https://arxiv.org/abs/2105.00503v2}



\bibitem{ciccio}  C. De Filippis \& G. Mingione, Lipschitz bounds and nonautonomous integrals, \emph{Arch. Ration. Mech. Anal.} 242, 973-1057 (2021). 



\bibitem{dm1} C. De Filippis \& G. Mingione, On the regularity of minima of non-autonomous functionals, \emph{J. Geom.  Anal.} 30, 1584-1626 (2020).

\bibitem{DMforth} C. De Filippis \& G. Mingione, Forthcoming.


\bibitem{EMM} M. Eleuteri, P. Marcellini \& E. Mascolo, Lipschitz estimates for systems with ellipticity conditions at infinity, {\em Ann. Mat. Pura Appl. (IV)} 195, 1575--1603 (2016).


\bibitem {sharp} L. Esposito, F. Leonetti \& G. Mingione, Sharp regularity for functionals with $(p,q)$ growth, \emph{J. Diff. Equ.} 204, 5-55 (2004).


\bibitem {Fre} J. Frehse,  A note on the H\"older continuity of solutions of variational
problems, {\em Abh. Math. Sem. Univ. Hamburg} 43, 59--63 (1975).

\bibitem {gg1}  M. Giaquinta \& E. Giusti,
On the regularity of the minima of variational integrals, {\em Acta
Math.} 148,  31--46 (1982). 

\bibitem{gg2} M. Giaquinta \& E. Giusti, Differentiability of minima of nondifferentiable functionals, \emph{Invent. Math.} 72, 285-298 (1983).

\bibitem{gg3} M. Giaquinta \& E. Giusti, 
Global $C^{1, \alpha}$-regularity for second order quasilinear elliptic equations in divergence form, 
{\em J. Reine Angew. Math. (Crelle J.)} 351, 55--65 (1984).

\bibitem{gg4} M. Giaquinta \& E. Giusti, Sharp estimates for the derivatives of local minima of variational integrals, {\em Boll. Un. Mat. Ital. A (6)} 3, 239–248 (1984). 

\bibitem{gt} D. Gilbarg \& N. Trudinger, {\em Elliptic partial differential equations of second order. Second edition.} Grundlehren der Mathematischen Wissenschaften, 224. Springer-Verlag, Berlin, 1983. xiii+513 pp.


\bibitem{giu} E. Giusti, {\em Direct Methods in the calculus of variations}, World Scientific Publishing Co., Inc., River Edge, (2003).

\bibitem {giumi}E. Giusti \& M. Miranda, Un esempio di soluzioni
discontinue per un problema di minimo relativo ad un integrale
regolare del calcolo delle variazioni, {\em Boll.~Un.~Mat.~Ital.
(IV)} 1, 219--226 (1968). 

\bibitem{gra} L. Grafakos, {\em Classical Fourier analysis. Second edition}. Graduate Texts in Mathematics, 249. Springer, New York, 2008. xvi+489 pp.

\bibitem{ha} C. Hamburger, Regularity of differential forms minimizing degenerate elliptic functionals, \emph{J. Reine Angew. Math. (Crelles J.)} 431, 7-64 (1992).

\bibitem{hast} P. Hartman \& G. Stampacchia, On some non-linear elliptic differential-functional equations, {\em Acta Math.} 115, 271–310 (1966). 

\bibitem{ivanov0} A.V. Ivanov, The Dirichlet problem for second order quasilinear nonuniformly elliptic equations, {\em Trudy Mat. Inst. Steklov.} 116, 34--54 (1971). 

\bibitem{ivanov1} A.V. Ivanov, Local estimates of the maximum modulus of the first derivatives of the solutions of quasilinear nonuniformly elliptic and nonuniformly parabolic equations and of systems of general form, {\em Proc. Steklov Inst. Math.} No. 110 (1970). AMS, 1972, 48–71. 

\bibitem{ivanov2} A.V. Ivanov, Quasilinear degenerate and nonuniformly elliptic and parabolic equations of second order, {\em Proc. Steklov Inst. Math.} 1984, no. 1 (160), xi+287 pp.

\bibitem{ivert1} P.A. Ivert, Regularitätsuntersuchungen von L\"osungen elliptischer Systeme von quasilinearen Differentialgleichungen zweiter Ordnung, {\em manuscripta math.} 30, 53--88 (1979/80). 



\bibitem{ivert2} P.A. Ivert, Partial regularity of vector valued functions minimizing variational integrals. Preprint Univ. Bonn (1982)

\bibitem{IO} N.M. Ivočkina \& A.P. Oskolkov, Nonlocal estimates of the first derivatives of solutions of the Dirichlet problem for nonuniformly elliptic quasilinear equations, {\em Zap. Naučn. Sem. Leningrad. Otdel. Mat. Inst. Steklov. (LOMI)} 5, 37--109, (1967). 

\bibitem{HHT} P. Harjulehto, P. H\"ast\"o \& O. Toivanen, Hölder regularity of quasiminimizers under generalized growth conditions, {\em Calc. Var. \& PDE} 56:22 (2017). 


\bibitem {HM} V.~Maz'ya \& M.~Havin,  A nonlinear potential theory, \emph{Russ.~Math.~Surveys}, 27,  71--148 (1972). 

\bibitem{HO} P. H\"ast\"o \& J. Ok, Maximal regularity for minimizers of nonautonomous functionals, \emph{J. Europ. Math. Soc.}, to appear. 

\bibitem{HO2} P. H\"ast\"o \& J. Ok, Regularity theory for non-autonomous partial differential equations without Uhlenbeck structure, \url{https://arxiv.org/abs/2110.14351} 


\bibitem{HS} J. Hirsch \& M. Sch\"affner, Growth conditions and regularity, an optimal local boundedness result, {\em Comm. Cont. Math.} 2050029 (2020). 

\bibitem{hopf} E. Hopf, Bemerkungen zu einem satze von S. Bernstein aus del' theorie del' elliptischen differentialgleichungen, {\em Math. Z.} 29, 744-745 (1928).


\bibitem {kilp} T. Kilpel\"ainen \& J. Mal\'y, The Wiener test and potential
estimates for quasilinear elliptic equations, {\em Acta Math.} 172, 137--161 (1994). 

\bibitem {KM} J. Kristensen \& G. Mingione, The singular set of minima of integral functionals, \emph{Arch. Rat. Mech. Anal.} 180, 331-398 (2006). 

\bibitem {KMbou} J. Kristensen \& G. Mingione,  Boundary regularity in variational problems,  \emph{Arch. Rat. Mech. Anal.} 198, 369–455 (2010). 


\bibitem{KMstein} T. Kuusi \& G. Mingione, A nonlinear Stein theorem, 
\emph{Calc. Var. \& PDE} 51, 45--86 (2014). 

\bibitem {LU} O.A. Ladyzhenskaya \& N.\,N. Ural'tseva, {\em Linear and
quasilinear elliptic equations}, Academic Press, New York-London, 1968. 


\bibitem {LUcpam} O.A. Ladyzhenskaya \& N.\,N. Ural'tseva, 
Local estimates for gradients of solutions of nonuniformly elliptic
and parabolic equations, {\em Comm.~Pure Appl.~Math.} 23, 677--703 (1970). 

\bibitem{le1} J.\,L. Lewis, Smoothness of certain degenerate elliptic equations. {\em Proc. Amer. Math. Soc.} 80, 259--265 (1980).


\bibitem{liebe0} G.\,M. Lieberman,  Interior gradient bounds for nonuniformly parabolic equations, {\em Indiana Univ. Math. J.} 32, 579--601 (1983)

\bibitem{li} G.\,M. Lieberman, Boundary regularity for solutions of degenerate elliptic equations, \emph{Nonlinear Anal.} 12, 1203--1219 (1988).

\bibitem{liebe} G.\,M. Lieberman, The natural generalization of the
natural conditions of Ladyzhenskaya and Ural'tseva
 for elliptic equations, {\em Comm. PDE} 16, 311--361 (1991). 
 
 
 \bibitem{lieberev} G.\,M. Lieberman, Review of \cite{gg3}, Math. Rev. MR0749677. 
 


\bibitem {manth1} J.\,J. Manfredi, Regularity of the gradient for a class of nonlinear possibly
degenerate elliptic equations, {\em Ph.D. Thesis}, University of
Washington, St. Louis.

\bibitem {manth2} J.\,J. Manfredi, Regularity for minima of functionals with $p$-growth, \emph{J.~Diff.~Equ.} 76, 203--212 (1988). 

\bibitem{M0} P.~Marcellini, On the definition and the lower
semicontinuity of certain quasiconvex integrals, {\em
Ann.~Inst.~H.~Poincar\'e Anal.~Non Lin\'eaire} 3, 391--409 (1986). 

\bibitem{ma5} P. Marcellini, The stored-energy for some discontinuous deformations in nonlinear elasticity. \emph{Partial Differential Equations and the Calculus of Variations} vol. II, Birkh\"auser Boston Inc., (1989).



\bibitem{M1} P.~Marcellini, Regularity of
minimizers of integrals of the calculus of variations with non
standard growth conditions, {\em Arch.~Ration.~Mech.~Anal.} 105, 267--284 
(1989). 

\bibitem{M2} P.~Marcellini, Regularity and
existence of solutions of elliptic equations with $p,q$-growth
conditions, \emph{J.~Diff.~Equ.} 90, 1--30 (1991).

\bibitem{M3} P.~Marcellini, Regularity for elliptic equations
with general growth conditions, \emph{J.~Diff.~Equ.} 105, 
296--333 (1993). 


\bibitem{mi} G. Mingione, Gradient potential estimates, \emph{J. Eur. Math. Soc.} 13, 459-486 (2011).

\bibitem{oneil} R. O'Neil, Integral transforms and tensor products on Orlicz spaces and $L(p,q)$ spaces, {\em J. Analyse Math.} 21, 1--276 (1968). 

\bibitem{osk} A.\,P. Oskolkov, A priori estimates for the first derivatives-of solutions of the Dirichlet problem for non-uniformly elliptic quasilinear equations, {\em Trudy Mat. Inst. Steklov (Akad. Nauk. SSSR)} 102, 105-127 (1967). 

\bibitem{phil} D. Phillips, A minimization problem and the regularity of solutions in the presence of a free boundary, {\em Indiana Univ. Math. J.} 32, 1–17 (1983). 


\bibitem{sha} M. Sch\"affner, Higher integrability for variational integrals with nonstandard growth, {\em Calc. Var. \& PDE} 
 60:77 (2021).
\bibitem{js} J. Schauder, Über lineare elliptische differentialgleichungen zweiter ordnung, {\em Math. Z.} 38, 257--282,(1934).

\bibitem{js2} J. Schauder, Numerische absh\"atzunger in elliptischen linearen differentialgleichungen equations, {\em 
Studia Math.} 5, 34--42 (1934). 

\bibitem{serrin} J. Serrin, The problem of Dirichlet for quasilinear elliptic differential equations with many independent variables, {\em Philos. Trans. Roy. Soc. London Ser. A} 264, 413--496 (1969). 


\bibitem{Simongrad} L. Simon, Interior gradient bounds for nonuniformly elliptic equations, 
{\em Indiana Univ. Math. J.} 25, 821--855 (1976). 

\bibitem{simon2} L. Simon, Schauder estimates by scaling, {\em Calc. Var. \& PDE} 5, 391--407 (1997). 

\bibitem {stamp} G. Stampacchia, On some regular multiple integral problems in the calculus of variations, \emph{Comm. Pure Appl. Math.} 16, 383--421 (1963). 

\bibitem {St} E.\,M.~Stein, 
Editor's note: the differentiability of functions in $\er^n$, 
\emph{Ann.~of Math. (II)}  113, 383--385 (1981). 


\bibitem{sw} E.\,M. Stein \&  G. Weiss, {\em Introduction to Fourier Analysis on Euclidean spaces|}, Princeton Univ. Press,  (1971).

\bibitem{tromba} A. Tromba: Morse theory and the calculus of variations, {\em Adv. Calc. Var.} (2021). 

\bibitem{tru1967} N. Trudinger, The Dirichlet problem for nonuniformly elliptic equations, {\em Bull. Amer. Math. Soc.} 73, 410–413  (1967). 

\bibitem{tru} N. Trudinger, Harnack inequalities for nonuniformly elliptic divergence structure equations, {\em Invent. Math.} 64, 517--531 (1981). 

\bibitem{trusc} N. Trudinger, A new approach to the Schauder estimates for linear elliptic equations, {\em Proc. Centre for Mathematics and its Applications} 1986:52-59 (1986).



\bibitem{Ur} N.\,N. Ural'tseva, Degenerate quasilinear elliptic systems, {\em Zap.~Na.
Sem.~Leningrad.~Otdel.~Mat.~Inst.~Steklov.~(LOMI)} 7, 
184--222 (1968).


\bibitem{UU} N.\,N. Ural'tseva \& A. B. Urdaletova, The boundedness of
the gradients of generalized solutions of degenerate quasilinear
nonuniformly elliptic equations, {\em Vestnik Leningrad Univ. Math.
19} 16, 263-270 (1984).

\bibitem{Z0}  V.\,V.~Zhikov, On Lavrentiev's Phenomenon, {\em Russian J. Math. Phys.} 3, 249--269 (1995). 
 
\bibitem{Z1} V.\,V.~Zhikov, On some variational problems, \emph{Russian J. Math. Phys.}  5, 105--116 (1997).

\bibitem{Z2} V.\,V.~Zhikov, Averaging of functionals in the Calculus of Variations and elasticity theory, {\em Math. of USSR-Izvestia.} 3, 249--269 (1995). 


\end{thebibliography}
\end{document}